\documentclass[reqno]{amsart}
\usepackage{amsthm,amsmath,amsfonts,amssymb}
\usepackage{epsfig,graphicx}
\usepackage{dsfont,bbm}
\usepackage{enumitem}
\usepackage[T1]{fontenc}
\usepackage{todonotes}
\usepackage{hyperref}
\usepackage{tikz}
\usepackage{mathrsfs} 
\usepackage{caption}
\usepackage{subcaption}
\usepackage{cite}

\newtheorem{definition}{Definition}[section]
\newtheorem{theorem}[definition]{Theorem}
\newtheorem{proposition}[definition]{Proposition}
\newtheorem{lemma}[definition]{Lemma}
\newtheorem{corollary}[definition]{Corollary}

\newtheorem{remark}[definition]{Remark}

\newtheorem{notation}[definition]{Notation}
\newtheorem{assumption}[definition]{Assumption}

\newcommand{\dl}{\ensuremath{\mathrm{d}}}
\newcommand{\Id}{\ensuremath{\mathrm{Id}}}

\renewcommand{\geq}{\ensuremath{\geqslant}}
\renewcommand{\leq}{\ensuremath{\leqslant}}

\newcommand{\lm}{\ensuremath{\kappa}} 
\newcommand{\Ceta}{\ensuremath{C_{\eta}}} 
\newcommand{\Clm}{\ensuremath{C_{\kappa}}} 
\newcommand{\lambdarem}{\hat\lambda^{\mathrm{re}}} 
\newcommand{\Lambdarem}{\hat\Lambda^{\mathrm{re}}} 
\newcommand{\lambdare}{\lambda^{\mathrm{re}}} 
\newcommand{\Lambdare}{\Lambda^{\mathrm{re}}} 
\newcommand{\drift}{\mathcal A} 
\newcommand{\diff}{\mathcal B} 

\newcommand{\bE}{{\mathbb E}}
\newcommand{\bN}{{\mathbb N}}
\newcommand{\bP}{{\mathbb P}}
\newcommand{\bR}{{\mathbb R}}
\newcommand{\cF}{{\mathcal F}}
\newcommand{\cD}{{\mathcal D}}
\newcommand{\cM}{{\mathcal M}}
\newcommand{\cO}{{\mathcal O}}
\newcommand{\sL}{{\mathscr L}}
\newcommand{\bbullet}{\small{\, \mathop{\mbox{\tiny$^{_\bullet}$}}\, }}
\newcommand{\dist}{{\operatorname{dist}}}

\allowdisplaybreaks
\begin{document}

\markboth{F. Lindner and H. Stroot}{Strong approximation of stochastic mechanical systems with holonomic constraints}

\title[Strong convergence of a half-explicit scheme]{Strong convergence of a half-explicit Euler scheme for constrained stochastic mechanical systems}

\author{Felix Lindner}
\address{Felix Lindner, Technische Universit\"at Kaiserslautern, Fachbereich Mathematik, 67653 Kaiserslautern, Germany\\ lindner@mathematik.uni-kl.de}
\author{Holger Stroot}
\address{Holger Stroot, Fraunhofer-Institut f\"ur Techno- und Wirtschaftsmathematik, Fraunhofer Platz 1, 67663 Kaiserslautern, Germany\\ holger.stroot@itwm.fraunhofer.de}

\begin{abstract}
This paper is concerned with the numerical approximation of stochastic mechanical systems with nonlinear holonomic constraints.
Such systems are described by second order stochastic differential-algebraic equations involving an implicitly given Lagrange multiplier process.
The explicit representation of the Lagrange multiplier leads to an underlying stochastic ordinary differential equation, the drift coefficient of which is typically not globally one-sided Lipschitz continuous. 
We investigate a half-explicit drift-truncated Euler scheme which fulfills the constraint exactly. 
Pathwise uniform $L_p$-convergence is established.
The proof is based on a suitable decomposition of the discrete Lagrange multipliers and on norm estimates for the single components, enabling the verification of consistency, semi-stability and moment growth properties of the scheme.
To the best of our knowledge, the presented result is the first strong convergence result for a constraint-preserving scheme in the considered setting.
\end{abstract}

\maketitle
\textsc{AMS-Classification.} Primary 60H35, 74Hxx; Secondary 60H10, 58J65, 65C30. 


\section{Introduction}\label{sec:intro}

Both the numerical approximation of differential-algebraic equations (DAEs) and of stochastic differential equations (SDEs) have been extensively studied in the literature.
Convergence results for higher index DAEs 
can be found, 
e.g., in \cite{brenan:b:1989,gear:p:1985,hairer:b:1989,hairer:b:1991,loetstedt:p:1986}, 
and the convergence analysis of numerical schemes for SDEs with non-globally Lipschitz continuous coefficients has been an active and rapidly evolving field of research within the last years, see, e.g., \cite{beyn:p:2016,beyn:p:2017,fang:p:2016,gyongy:p:1998,higham:p:2002,hutzenthaler:p:2014,hutzenthaler:p:2015,hutzenthaler:p:2011,hutzenthaler:p:2012,mao:p:2015, mao:p:2016,sabanis:p:2013,sabanis:p:2015}. 
In contrast, the convergence analysis of numerical schemes for stochastic differential-algebraic equations (SDAEs) is far less developed.
In this article, we combine key concepts from both areas, numerics of DAEs and numerics of SDEs,  
and
prove strong convergence of a constraint-preserving numerical scheme for a large class of second order SDEs with nonlinear algebraic constraints.

We are interested in the dynamics of constrained stochastic mechanical systems which are modelled by SDAEs of the type
\begin{align}\label{eq:mainSDAE}
	\begin{split}
		\dl r(t)&= v(t)\,\dl t\\
		M\dl v(t)&= a(r(t),v(t))\,\dl t+ B(r(t),v(t))\, \dl w(t) + \nabla g(r(t))\,\dl\mu(t)\\
		g(r(t))&=0,
	\end{split}
\end{align}
with $n$-dimensional
position and velocity processes $r=(r(t))_{t\geq0}$ and $v=(v(t))_{t\geq0}$, positive definite and symmetric mass matrix $M\in\bR^{n\times n}$,
$\ell$-dimensional
Brownian motion $w=(w(t))_{t\geq 0}$, coefficient functions $a\colon \bR^n\times\bR^n\to\bR^n$ and $B\colon \bR^n\times\bR^n\to\bR^{n\times\ell}$, and a sufficiently smooth constraint function $g\colon\bR^n\to\bR^m$, $m<n$.
By $\nabla g(x)\in\bR^{n\times m}$ we denote the transpose of the Jacobian matrix $Dg(x)\in\bR^{m\times n}$ of $g$ at $x\in\bR^n$, which is assumed to be of full rank for all $x$ in a neighborhood of the 
constraint manifold $\cM=\{x\in\bR^n:g(x)=0\}.$
The process $\mu=(\mu(t))_{t\geq0}$ is an 
$m$-dimensional
semimartingale, implicitly given as the Lagrange multiplier to the holonomic constraint $g(r(t))=0$. 
It thus determines the constraint force, which, according to d'Alembert's principle, acts orthogonal to the tangent space $T_x\cM=\{y\in\bR^n:Dg(x)y=0\}$ at the current position $x\in\cM$.
Constrained SDEs of the form \eqref{eq:mainSDAE} occur in various applications,
ranging from molecular dynamics 
\cite{leimkuhler:p:2016,leimkuhler:b:2015,lelievre:b:2010,lelievre:p:2012,ciccotti:p:2006,walter:p:2011}
to models for the dynamics of fibers in turbulent airflows in the context industrial production processes of non-woven textiles \cite{lindner:p:2016,lindner:p:2018}.

Our assumptions on the coefficients $a$, $B$ and the constraint function $g$ in \eqref{eq:mainSDAE} are weak enough to cover a large variety of practically relevant examples.
We assume that $a$ and $B$ are locally Lipschitz continuous, of polynomial and linear growth, respectively, and that the mapping $(x,y)\mapsto(y,a(x,y))$ satisfies a one-sided linear growth condition.
Note that the diffusion coefficient $B$ is allowed to depend on both position and velocity, which is motivated by \cite{lindner:p:2016,lindner:p:2018}.
The derivatives of $g$ are assumed to satisfy suitable boundedness and non-degeneracy conditions in a neighborhood of the manifold.
Our setting allows in particular for quadratic constraint functions, which are relevant in various applications, cf.~Section~\ref{sec:examples}.
The precise assumptions are stated in Section~\ref{subsec:assumptions}.
By a slight generalization of the existence result in \cite{lindner:p:2016}, we know that for all initial conditions $r_0\in\cM$, $v_0\in T_{r_0}\cM$ there exists a unique global strong solution $(r,v,\mu)$ to the SDAE \eqref{eq:mainSDAE}, see  Section~\ref{sec:solutiontheory} for details.

There seem to be no strong convergence results available in the literature so far for constraint-preserving schemes for SDAEs of the type \eqref{eq:mainSDAE} with quadratic constraint functions $g$.
A particular class of such SDAEs are the constrained Langevin-type equations considered in molecular dynamics. Various constraint-preserving numerical schemes have been proposed and applied in this context, see, e.g., \cite{leimkuhler:p:2016,lelievre:b:2010,lelievre:p:2012,ciccotti:p:2006} and the references therein.
Typically, in
the corresponding sections of these works the focus lies mainly on the practical efficiency of the algorithms but not so much on a fully rigorous convergence analysis.
Often the proposed partly implicit schemes are even known to be not always solvable, cf.~\cite[Section~2]{leimkuhler:p:2016}.
We are also not aware of 
fully completed
proofs concerning weak convergence. 
Further works related to our problem concern SDEs which are given in an explicit form, without an implicit Lagrange multiplier process as in \eqref{eq:mainSDAE}, and the analysis of structure-preserving algorithms for their numerical approximation, see, e.g., \cite{hong:p:2011,malham:p:2008,milstein:p:2002,milstein:p:20022,milstein:b:2004,zhou:p:2016}. 
The theoretical results in these works rely on
classical global Lipschitz assumptions on the drift and diffusion coefficients. 
We note that the Lagrange multiplier process $\mu$ in \eqref{eq:mainSDAE} can be represented explicitly in terms of $r$ and $v$, so that the SDAE can be equivalently reformulated as an inherent SDE
 which does not involve an implicit Lagrange multiplier, see Eq.~\eqref{eq:inherentSDE} in Section~\ref{sec:solutiontheory}. 
However, the drift coefficient appearing in this inherent SDE is typically neither globally Lipschitz continuous nor globally one-sided Lipschitz continuous, even if the coefficients $a$ and $B$ in \eqref{eq:mainSDAE} are chosen to be constant, see Section~\ref{subsec:pendel} for a simple toy example.
Thus, the results in the mentioned works are not applicable in our setting.

Our numerical scheme combines ideas from different areas:
On the one hand, we consider a \emph{Gear-Gupta-Leimkuhler (GGL) reformulation} of \eqref{eq:mainSDAE} to simplify the approximation problem 
and use a
\emph{half-explicit method} where only the algebraic variables are discretized in an implicit manner.
These concepts are standard in 
DAE theory, 
cf.~\cite{brenan:b:1989,gear:p:1985,hairer:b:1989,hairer:b:1991,lubich:p:1989,ostermann:p:1993}. 
On the other hand, we follow the \emph{taming} or \emph{truncation approach}  
developed in recent years in the context of numerical methods for SDEs with non-globally Lipschitz continuous coefficients, cf.~\cite{hutzenthaler:p:2015,hutzenthaler:p:2012,mao:p:2015,mao:p:2016,sabanis:p:2013}. 
We refer to Section~\ref{subsec:scheme} for a short discussion of the concepts.
The combination of these ideas motivates 
the following
half-explicit drift-truncated Euler scheme for the approximation of the SDAE~\eqref{eq:mainSDAE}.
Given initial conditions $r_0\in\cM$, $v_0\in T_{r_0}\cM$, a finite time interval $[0,T]$, and a number of time steps $N\in\bN$, we set $h=T/N$ and  approximate the solution processes $(r(t))_{t\in[0,T]}$ and $(v(t))_{t\in[0,T]}$ in \eqref{eq:mainSDAE} by the time-discrete processes $(r_k)_{k\in\{0,\ldots,N\}}=(r_k^N)_{k\in\{0,\ldots,N\}}$ and $(v_k)_{k\in\{0,\ldots,N\}}=(v_k^N)_{k\in\{0,\ldots,N\}}$ iteratively defined by
\begin{align}\label{eq:approximationintro}
	\begin{split}
		r_{k+1}&= r_k+\eta_k\,v_k\,h+M^{-1}\nabla g(r_k)\,\kappa_{k+1}\\
		M v_{k+1}&=Mv_k+\eta_k\,a(r_k,v_k)\,h+B(r_k,v_k)\,\Delta_h w_k+\nabla g(r_k)\,\lambda_{k+1}\\
		g(r_{k+1})&=0\\
		Dg(r_{k+1})v_{k+1}&=0.
	\end{split}
\end{align}
Here $\eta_k$ is a scalar truncation term, chosen as 
\begin{align}\label{eq:defetaintro}
	\eta_k=\min\Big(1\,,\;\frac{C_\eta}{\max(\|v_k\|,\|v_k\|^2,\|a(r_k,v_k)\|)\,h}\,\Big),
\end{align}
with a suitable constant $C_\eta\in(0,\infty)$,
see Section~\ref{subsec:scheme} for details and a discussion.
The $m$-dimensional Lagrange multipliers $\kappa_{k+1}$ and $\lambda_{k+1}$ are implicitly determined by the constraints in the third and forth line of \eqref{eq:approximationintro},
and $\Delta_h w_k=w((k+1)h)-w(kh)$ is the increment of the driving Brownian motion on the time interval $[kh,(k+1)h]$. Since the Lagrange multiplier $\lambda_{k+1}$ 
formally corresponds to the infinitesimal increment $\dl\mu(t)$ in \eqref{eq:mainSDAE}, 
it is natural to use the time-discrete process $(\mu_k)_{k\in\{0,\ldots,N\}}=(\mu_k^N)_{k\in\{0,\ldots,N\}}$ defined by 
\begin{align}\label{eq:defmuintro}
	\mu_k=\sum_{j=1}^k\lambda_j
\end{align}
as an approximation of the Lagrange multiplier process $(\mu(t))_{t\in[0,T]}$ in \eqref{eq:mainSDAE}.
We remark that the presence of the truncation factor $\eta_k$ in \eqref{eq:approximationintro} fulfills two purposes: It ensures the solvability of the scheme and moreover allows for the derivation of moment bounds needed to obtain strong convergence.

In this article, we verify the strong convergence of the scheme~\eqref{eq:approximationintro} towards the system~\eqref{eq:mainSDAE} in a pathwise uniform $L_p$-sense as described below.
To the best of our knowledge, this is the first proof of strong convergence for a constraint-preserving scheme for SDAEs of the type \eqref{eq:mainSDAE} with possibly quadratic constraint functions~$g$. 
Moreover, it also seems to be the first proof of convergence at all for a constraint-preserving scheme in the considered general setting.

Our \emph{first main result}, Theorem~\ref{theo:mainonestep}, states that 
the scheme \eqref{eq:approximationintro} is uniquely solvable for 
all choices of
$N\in\bN$ 
and 
$r_0\in\cM$, $v_0\in T_{r_0}\cM$,
in the sense that there exists exactly one 
solution $\big((r_k,v_k,\kappa_k,\lambda_k)\big)_{k\in\{1,\ldots,N\}}$
such that the Lagrange multiplier process $(\kappa_k)_{k\in\{1,\ldots,N\}}$ satisfies a specific boundedness condition.
This existence result is non-trivial, in particular in view of the fact mentioned above that alternative schemes considered in the literature are often known to be not always solvable. 
The proof of Theorem~\ref{theo:mainonestep} is based on a homotopy argument and relies on the presence of the truncation term $\eta_k$ in \eqref{eq:approximationintro}.
Note that we do not impose any condition on the step-size $h$ and moreover no truncation of the noise term $B(r_k,v_k)\Delta_hw_k$ is needed.
Besides the existence and uniqueness of a numerical solution, 
Theorem~\ref{theo:mainonestep} also provides a suitable decomposition of the Lagrange multiplier $\lambda_{k+1}$ as well as suitable norm estimates for the single components of $\lambda_{k+1}$ and for $\kappa_{k+1}$, crucial for proving convergence of the scheme.

Our \emph{second main result}, Theorem~\ref{theo:mainend},
concerns the strong convergence of the scheme~\eqref{eq:approximationintro}.
Let $\tilde r^N=(\tilde r^N(t))_{t\in[0,T]}$, $\tilde v^N=(\tilde v^N(t))_{t\in[0,T]}$ and $\tilde \mu^N=(\tilde \mu^N(t))_{t\in[0,T]}$  be defined by piecewise constant or piecewise linear interpolation of the corresponding time-discrete processes 
$(r_k^N)_{k\in\{0,\ldots,N\}}$, $(v_k^N)_{k\in\{0,\ldots,N\}}$ and $(\mu_k^N)_{k\in\{0,\ldots,N\}}$
in \eqref{eq:approximationintro} and \eqref{eq:defmuintro}. 
For instance, in the piecewise linear case we have
$\tilde r^N(t)=\frac{(k+1)h-t}hr_k^N+\frac{t-kh}h\,r_{k+1}^N$ for all $t\in[kh,(k+1)h]$, $k\in\{0,\ldots,N-1\}$.
Theorem~\ref{theo:mainend} states that the time-interpolated solution $(\tilde r^N,\tilde v^N,\tilde\mu^N)$ to  \eqref{eq:approximationintro}, \eqref{eq:defmuintro} converges strongly towards the solution $(r,v,\mu)$ to the SDAE~\eqref{eq:mainSDAE} with initial conditions $r_0,v_0$, in the sense that
\begin{align}\label{eq:convergence_intro}
	\lim_{N\to\infty}
	\bE\bigg[\sup_{t\in[0,T]}\Big(
	\big\|r(t)-\tilde r^N(t)\big\|^p
	+\big\|v(t)-\tilde v^N(t)\big\|^p
	+\big\|\mu(t)-\tilde \mu^N(t)\big\|^p
	\Big)\bigg]=0
\end{align}
for all $p\in[1,\infty)$.
The main idea of the proof is to use the existence,  uniqueness and decomposition result from Theorem~\ref{theo:mainonestep} to formally rewrite the scheme \eqref{eq:approximationintro} as 
a fully-explicit one-step approximation 
of the underlying inherent SDE,
an explicit drift-truncated Euler scheme with an additional explicit perturbation term,
and to use the norm estimates for the Lagrange multipliers $\kappa_{k+1}$ and $\lambda_{k+1}$ from Theorem~\ref{theo:mainonestep} to verify suitable consistency, semi-stability and moment growth conditions for the discrete solution.
This enables us to apply a general convergence result from ~\cite{hutzenthaler:p:2015} to obtain that
$\lim_{N\to\infty}
\sup_{t\in[0,T]}
\bE\big(
	\|r(t)-\tilde r^N(t)\|^p
	+\|v(t)-\tilde v^N(t)\|^p
\big)=0$.
The pathwise uniform strong convergence of $(\tilde r^N,\tilde v^N)$ and  
the convergence of $\tilde\mu^N$ to the Lagrange multiplier process $\mu$ are then proven in separate steps, 
by further exploiting the assumptions on the coefficients $a$ and $B$, 
the decomposition and norm estimates from Theorem~\ref{theo:mainonestep} for the Lagrange multipliers $\kappa_{k+1}$, $\lambda_{k+1}$, 
and the equivalence of \eqref{eq:mainSDAE} and the underlying inherent SDE.


Let us shortly discuss the question of convergence rates. 
Most strong approximation results with rates for multi-dimensional SDEs are based on at least a global monotonicity assumption on the coefficients and thus in particular on a global one-sided Lipschitz assumption on the drift coefficient, see, e.g.,~\cite{higham:p:2002,hutzenthaler:p:2012,sabanis:p:2013}.
As already mentioned, the drift coefficient of the inherent SDE associated to the SDAE~\eqref{eq:mainSDAE} typically fails to satisfy such a condition, cf.~Section~\ref{sec:solutiontheory} and the example in Section~\ref{subsec:pendel}.
A recently developed strategy to obtain strong convergence rates in the case of SDEs with non-globally monotone coefficients makes use of exponential integrability properties of suitably tamed/truncated schemes \cite{hutzenthaler:p:2014}. 
This approach might potentially also be useful in the context of SDAEs of the type \eqref{eq:mainSDAE} but lies beyond the scope of the present work.

To complete the picture, let us also note that the approximation of so-called index one stochastic differential-algebraic equations has been analyzed in several papers, mainly in the context of electric circuits, cf.~\cite{kupper:p:2012,kupper:p:2015,roemisch:p:2003,schein:p:1998,winkler:p:2003} and the references therein. 
These equations are of a different structure than \eqref{eq:mainSDAE}. In particular, the standard assumptions used 
in the context of index one SDAEs, e.g., that the constraints are globally uniquely solvable for the algebraic variables \cite{winkler:p:2003}, are not fulfilled in our setting. One can think of Theorem \ref{theo:mainonestep} as partly substituting such assumptions.

The article is organized as follows:
In Section~\ref{sec:preliminaries} we describe our setting in detail (Subection~\ref{subsec:assumptions}) and state an existence and uniqueness result which ensures the global strong solvability of the SDAE~\eqref{eq:mainSDAE} as well as the equivalence of \eqref{eq:mainSDAE} and the above mentioned underlying inherent SDE (Subsection~\ref{sec:solutiontheory}).
In Section~\ref{sec:approximation} we first specify and discuss 
our approximation scheme
(Subsection~\ref{subsec:scheme}) before we give a detailed analysis of its solvability in Theorem~\ref{theo:mainonestep} and Corollary~\ref{cor:applyingonestep}
 (Subsection~\ref{sec:solvingapproximation}).
The strong convergence result \eqref{eq:convergence_intro} is proven in Section~\ref{sec:strongconv}; a complete formulation is given in Theorem~\ref{theo:mainend}.
Its proof is based on the results from Section~\ref{sec:approximation} as well as a suitable reformulation of the problem (Subsection~\ref{subsec:SDEsetting}) and a collection of auxiliary results from the literature 
(Subsection~\ref{subsec:auxiliary}).
We deduce the convergence of $(\tilde r^N,\tilde v^N)$ in a non-pathwise sense by verifying specific consistency, semi-stability and moment growth conditions (Subsection~\ref{subsec:convergence_rv}) and show the pathwise uniform convergence of $(\tilde r^N,\tilde v^N,\tilde \mu^N)$ in a separate step (Subsection~\ref{subsec:convergence_mu}).
Concrete examples for SDAEs of the type \eqref{eq:mainSDAE} are given in Section~\ref{sec:examples}. 
Finally, a proof of the global strong solvability of \eqref{eq:mainSDAE} is sketched in Appendix~\ref{sec:appsolv} and a globalized version of the implicit function theorem used in the proof of Theorem~\ref{theo:mainonestep} is presented in Appendix~\ref{sec:appGIFT}.

\section{Preliminaries}

\label{sec:preliminaries}
\noindent
\emph{General notation.} 
$\bN=\{1,2,\ldots\}$ is the set of natural numbers excluding zero. In the sequel, let $d,d_1,d_2\in\bN$.
By $\|\cdot\|$ and $\langle\cdot,\cdot\rangle$ we denote the Euclidean norm and the corresponding inner product in finite-dimensional real vector spaces. For instance,  we have $\|x\|:=\big(\sum_{i=1}^d x_i^2\big)^{1/2}$ for $x=(x_1,\ldots,x_d)\in\bR^d$, $\|(x,y)\|:=\big(\sum_{i=1}^{d_1} x_i^2+\sum_{j=1}^{d_2} y_j^2\big)^{1/2}$ for $(x,y)\in\bR^{d_1}\times\bR^{d_2}$, and $\|B\|:=\big(\sum_{i=1}^{d_1}\sum_{j=1}^{d_2} B_{i,j}^2\big)^{1/2}$ for $B=(B_{i,j})\in\bR^{d_1\times d_2}$.
If $f\colon\bR^{d_1}\to\bR^{d_2}$ is a differentiable function, we write $Df(x)\in\bR^{d_2\times d_1}$ for its Jacobian matrix at a point $x\in\bR^{d_1}$ and $\nabla f(x):=(Df(x))^\top\in \bR^{d_1\times d_2}$ for the transpose thereof.
By $\sL(\bR^{d_1},\bR^{d_2})$ and $\sL^{(k)}(\bR^{d_1},\bR^{d_2})$, $k\in\bN$, we denote the spaces of linear operators from $\bR^{d_1}$ to $\bR^{d_2}$ and $k$-fold multilinear operators from $(\bR^{d_1})^{k}$ to $\bR^{d_2}$, respectively, with norms $\|L\|_{\sL(\bR^{d_1},\bR^{d_2})}:=\sup_{x\in\bR^{d_1},\|x\|\leq 1}\|Lx\|$ 
and $\|L\|_{\sL^{(k)}(\bR^{d_1},\bR^{d_2})}:=\sup_{y_1,\ldots,y_k\in\bR^{d_1},\|y_1\|\leq1,\ldots,\|y_{k}\|\leq 1}\|L(y_1,\ldots,y_k)\|$.
The Jacobian matrix $Df(x)$ of a sufficiently smooth function $f\colon\bR^{d_1}\to\bR^{d_2}$ at a point $x\in\bR^{d_1}$ is identified with the corresponding element in $\sL(\bR^{d_1},\bR^{d_2})$, and, accordingly, the higher derivatives $D^kf(x)$, $k\in\bN$, are elements in $\sL^{(k)}(\bR^{d_1},\bR^{d_2})$.
If $X\colon\Omega\to\bR^d$ is a random variable on a probability space $(\Omega,\cF,\bP)$ and $p\geq1$ we write $\|X\|_{L_p}:=(\bE\|X\|^p)^{1/p}\in[0,\infty]$.

\subsection{Setting and assumptions}\label{subsec:assumptions}

Here we describe our setting concerning the SDAE ~\eqref{eq:mainSDAE} in detail. 
Let $n\in\{2,3,\ldots\}$, $m\in\{1,\ldots,n-1\}$, and $M\in\bR^{n\times n}$ be a symmetric and (strictly) positive definite matrix. The constraint function $g\colon\bR^n\to\bR^m$ in \eqref{eq:mainSDAE} determines the constraint manifold
\begin{align}\label{eq:manifoldM}
	\cM:=\{x\in\bR^n:g(x)=0\}
\end{align}
and satisfies the regularity conditions stated in Assumption~\ref{ass:g} below. We first introduce some notation associated with $\cM$.
\begin{notation}[Constraint manifold]
	The following notation associated with the constraint manifold $\cM$ in \eqref{eq:manifoldM} is used throughout the article.
	\begin{itemize}[leftmargin=7mm]
		\item
		The tangent space $T_x\cM\subset\bR^n$ at a point $x\in\cM$ and the tangent bundle $T\cM\subset\bR^n\times\bR^n$  are given by
		\begin{align*}
			T_x\cM:=\{y\in\bR^n:Dg(x)y=0\},\quad T\cM:=\bigcup_{x\in\cM}\{x\}\times T_x\cM.
		\end{align*}
		\item
		For $x\in\bR^n$ we denote by 
		$$
		G_M(x):=Dg(x)\,M^{-1}\,\nabla g(x)\,\in\bR^{m\times m}
		$$
		the Gram matrix associated with the constraint, where $M\in\bR^{n\times n}$ is the positive definite and symmetric mass matrix appearing in \eqref{eq:mainSDAE}.  
		\item
		For all $x\in\bR^n$ such that $G_M(x)$ is invertible we set
		$$
		P_M(x):=\Id-\nabla g(x)\,G_M^{-1}(x)\,Dg(x)\,M^{-1}\,\in\bR^{n\times n},
		$$
		where $G_M^{-1}(x):=(G_M(x))^{-1}$. Note that for $x\in\cM$ the matrix $P_M(x)$ represents the orthogonal projection of $\bR^n$ onto $M(T_x\cM):=\{My:y\in T_x\cM\}$, corresponding to the inner product $\langle M^{-1/2}\cdot,M^{-1/2}\cdot\rangle$.
		\item
		For $\varepsilon>0$ we introduce the environments 
		$$
		\cM^\varepsilon:=\bigcup_{x\in\cM} B_\varepsilon(x),\quad (T\cM)^\varepsilon:=\bigcup_{(x,y)\in T\cM}B_{\varepsilon}(x)\times B_\varepsilon(y)
		$$
		of the constraint manifold $\cM$ and the tangent bundle $T\cM$. Here $B_\varepsilon(x):=\{z\in\bR^n:\|z-x\|<\varepsilon\}$ is the open ball in $\bR^n$ with radius $\varepsilon$ and center $x$.
	\end{itemize}
\end{notation}

With this notation at hand, we are able to state the regularity assumptions on the constraint function $g$ in detail.
\begin{assumption}[Constraint function $g$]\label{ass:g}
	The constraint function $g\colon \bR^n\to\bR^m$ in \eqref{eq:mainSDAE} is three times continuously differentiable. There exists $\varepsilon>0$ such for all $x\in\cM^\varepsilon$ the Jacobian matrix $Dg(x)\in\bR^{m\times n}$  has full rank, 
	and 
	$$
	\sup_{x\in\cM^\varepsilon}\max\big(\|Dg(x)\|,\|G_M^{-1}(x)\|\big) <\infty.
	$$
	Moreover, the higher order derivative mappings $D^2g\colon\bR^n\to \sL^{(2)}(\bR^n,\bR^m)$ and $D^3g\colon\bR^n\to \sL^{(3)}(\bR^n,\bR^m)$ are bounded.
\end{assumption}

In the sequel we work with the finite constant
\begin{equation}\label{eq:C_g}
	C_g:=
	\left\{
	\begin{aligned}
		&\sup_{x_1\in\cM^\varepsilon}\sup_{x_2\in\bR^n}\max\Big(\|M^{-1} \nabla g(x_1)\|_{\sL(\bR^n,\bR^m)},\|G_M(x_1)\|_{\sL(\bR^m)},\vspace{2mm}\\
		&\qquad\|G_M^{-1}(x_1)\|_{\sL(\bR^m)},\|D^2g(x_2)\|_{\sL^{(2)}(\bR^n,\bR^m)},\|D^3g(x_2)\|_{\sL^{(3)}(\bR^n,\bR^m)}\Big)
	\end{aligned}
	\right. ,
\end{equation}
where $\varepsilon>0$ is as in Assumption~\ref{ass:g}.

We assume that $(\Omega,\cF,\bP)$ is a complete probability space, $(\cF_t)_{t\geq0}$ a filtration of sub-$\sigma$-algebras of $\cF$ satisfying the usual conditions, and  the process $w=(w(t))_{t\geq0}$ in \eqref{eq:mainSDAE} is an $\ell$-dimensional standard $(\cF_t)$-Wiener process on $(\Omega,\cF,\bP)$, for some $\ell\in\bN$.
For the coefficient functions $a$ and $B$ we assume the following.

\begin{assumption}[Coefficient functions $a$ and $B$]\label{ass:aB}
	The mappings $a\colon\bR^n\times \bR^n\to\bR^n$ and $B\colon\bR^n\times \bR^n\to\bR^{n\times\ell}$ in \eqref{eq:mainSDAE} 
	fulfill the following conditions.
	\begin{itemize}[leftmargin=7mm]
		\item 
		Local Lipschitz continuity: 
		For all $R\in(0,\infty)$ there exist a constant $C_{R}\in[0,\infty)$ such that, for all $(x,y),(\tilde x,\tilde y)\in B_R((0,0))$, 
		\begin{align*}
			\qquad
			\|a(x,y)-a(\tilde x,\tilde y)\|+\|B(x,y)-B(\tilde x,\tilde y)\|\leq C_R\|(x,y)-(\tilde x,\tilde y)\|.
		\end{align*}
		Here $B_R((0,0))$ is the open ball in $\bR^n\times\bR^n$ with radius $R$ and center $(0,0)$.
		\item 
		Growth conditions for $a$:
		The function $a$ satisfies the following one-sided linear growth and polynomial growth conditions. 
		There exist constants $C_a\in[0,\infty)$ and $p_a\in[1,\infty)$ such that, for all $(x,y)\in\bR^n\times\bR^n$,
		\begin{equation*}
			\qquad
			\begin{aligned}
				\big\langle(x,y)\,,\,(y,a(x,y))\big\rangle&\leq C_a\,\big(1+\|(x,y)\|^2\big),\\
				\|a(x,y)\|&\leq C_a\,\big(1+\|(x,y)\|^{p_a}\big).
			\end{aligned}
		\end{equation*}
		\item
		Linear growth of $B$:
		There exists a constant $C_B\in[0,\infty)$ such that, for all $(x,y)\in\bR^n\times\bR^n$,
		\begin{align*}
			\qquad
			\|B(x,y)\|&\leq C_B\,\big(1+\|(x,y)\|\big).
		\end{align*}
	\end{itemize} 
\end{assumption}

Finally, we specify the the intitial conditions $r(0)=r_0$ and $v(0)=v_0$ for the position and velocity processes $(r(t))_{t\geq0}$ and $(v(t))_{t\geq0}$ in \eqref{eq:mainSDAE}.
\begin{assumption}[Initial conditions]\label{ass:init}
	The initial conditions $r_0\colon\Omega\to\bR^n$, $v_0\colon\Omega\to\bR^n$ are $\cF_0$-measurable random variables, $p$-integrable for all $p\in[1,\infty)$, such that $\bP((r_0,v_0)\in T\cM)=1$.
\end{assumption}

\subsection{Solvability of the SDAE}\label{sec:solutiontheory}

Under the assumptions in Subsection~\ref{subsec:assumptions} there exists a unique global strong solution $(r,v,\mu)$ to the SDAE \eqref{eq:mainSDAE}. We specify the notion of a solution as follows.

\begin{definition}\label{def:solutionSDAE}
	A \emph{(global strong) solution to the SDAE \eqref{eq:mainSDAE}} with initial conditions $r_0,v_0$ fulfilling Assumption~\ref{ass:init} is a triple $(r,v,\mu)$ consisting of $\bR^n$-valued continuous $(\cF_t)$-adapted processes $r=(r(t))_{t\geq0}$ and $v=(v(t))_{t\geq0}$ as well as an $\bR^m$-valued continuous $(\cF_t)$-semimartingale $\mu=(\mu(t))_{t\geq0}$ with $\mu(0)=0$ such that, $\bP$-almost surely, the following equalities hold for all $t\geq0$:
	\begin{align*}
	r(t)&=r_0+\int_0^t v(s)\,\dl s,\\
	Mv(t)&=Mv_0+\int_0^t a(r(s),v(s))\,\dl s+\int_0^t B(r(s),v(s))\,\dl w(s)\\
	&\quad+\int_0^t \nabla g(r(s))\,\dl \mu(s),\\
	g(r(t))&=0.
	\end{align*}
\end{definition}

The following result concerning the global strong solvability of SDAEs of the type  \eqref{eq:mainSDAE} is a slight generalization of \cite[Theorem 3.1]{lindner:p:2016}.
Compared to \cite{lindner:p:2016} we consider weaker assumptions on the drift coefficient $a$, a more general class of constraint functions $g$, $M$ is not assumed to be the identity matrix, and finite moments of all orders are established.

\begin{theorem}[Existence and uniqueness]\label{theo:solutionandinherent}
	Let the assumptions in Subsection~\ref{subsec:assumptions} be fulfilled. 
	Then there exists a unique (up to indistinguishability) global strong solution $(r,v,\mu)$ to the SDAE \eqref{eq:mainSDAE} with initial conditions $r_0,v_0$ in the sense of Definition~\ref{def:solutionSDAE}.
	With probability one, the following equality holds for all $t\geq0$: 
	\begin{align}\label{eq:defmu}
	\begin{split}
	\mu(t)&=-\int_0^t G_M^{-1}(r(s)) \Big\{ Dg(r(s))M^{-1} \big[a(r(s),v(s))\dl s+ B(r(s),v(s)) \dl w(s)\big]\\
	&\qquad+D^2g(r(s))(v(s),v(s))\dl s\Big\}.
	\end{split}
	\end{align}
	Moreover, for all $p,T\in[1,\infty)$  the $p$-th moment $\bE\big(\sup_{t\in[0,T]}\|(r(t),v(t),\mu(t))\|^p\big)$ is finite.
\end{theorem}

The proof of Theorem~\ref{theo:solutionandinherent} is mostly analogous to that of \cite[Theorem 3.1]{lindner:p:2016}. 
A sketch of the key steps of the proof and the differences to \cite{lindner:p:2016} can be found in Appendix~\ref{sec:appsolv}.

As a direct consequence of Theorem~\ref{theo:solutionandinherent} we know that, under the assumptions in Subsection~\ref{subsec:assumptions}, the following holds: If $(r,v,\mu)$ is the unique global strong solution to the SDAE~\eqref{eq:mainSDAE} with initial conditions $r_0,v_0$,
then $(r,v)$ solves the so-called \emph{inherent} 
or \emph{underlying SDE}
\begin{align}\label{eq:inherentSDE}
	\begin{split}
		\dl r(t)&= v(t)\,\dl t\\
		M \dl v(t)&= P_M(r(t))\,a(r(t),v(t))\,\dl t+ P_M(r(t))\,B(r(t),v(t))\,\dl w(t)\\
		&\quad-\nabla g (r(t))\,G_M^{-1}(r(t))\,D^2g(r(t))(v(t),v(t))\,\dl t.
	\end{split}
\end{align}
Conversely, consider arbitrary locally Lipschitz continuous extensions of the coefficients in the inherent SDE~\eqref{eq:inherentSDE}
to the whole space $\bR^n\times\bR^n$, and let $(r,v)$ be a global strong solution to \eqref{eq:inherentSDE} with initial conditions $r_0,v_0$.
Then it is not difficult to check that $\bP\big(g(r(t))=0\text{ for all }t\geq0\big)=1$ and, if further $\mu$ is the continuous $\bR^{m}$-valued semimartingale defined by \eqref{eq:defmu}, then $(r,v,\mu)$ is the unique global strong solution to the SDAE~\eqref{eq:mainSDAE}.

\section{Analysis of the numerical scheme}\label{sec:approximation}

We shortly discuss our numerical scheme in Subsection~\ref{subsec:scheme} before we present a detailed analysis of its solvability in Subection~\ref{sec:solvingapproximation}. The results in this section are essential for the convergence analysis in Section~\ref{sec:strongconv}.

\subsection{A half-explicit drift-truncated Euler scheme}\label{subsec:scheme}
Suppose that the assumptions in Section~\ref{subsec:assumptions} are fulfilled and recall from Section~\ref{sec:intro} the half-explicit drift-truncated Euler scheme, which we specify and slightly reformulate as follows: Given $N\in\bR^n$ we are searching for $\bR^n$-valued processes $(r_k)_{k\in\{0,\ldots,N\}}$, $(v_k)_{k\in\{0,\ldots,N\}}$ and $\bR^m$-valued Lagrange multiplier processes $(\kappa_k)_{k\in\{1,\ldots,N\}}$, $(\lambda_k)_{k\in\{1,\ldots,N\}}$ defined recursively by
\begin{subequations}\label{eq:approximation}
	\begin{align}
	\begin{split}\label{eq:approximationup}
	r_{k+1}&=r_k+\eta(r_k,v_k,h)v_k h+M^{-1}\nabla g(r_k)\kappa_{k+1}\\
	g(r_{k+1})&=0
	\end{split}\\
	\begin{split}\label{eq:approximationlow}
	v_{k+1}&=v_k+M^{-1}\Big[\eta(r_k,v_k,h) a(r_k,v_k)h+B(r_k,v_k)\Delta_hw_k+\nabla g(r_k) \lambda_{k+1}\Big]\\
	Dg(r_{k+1})v_{k+1}&=0.
	\end{split}
	\end{align}
\end{subequations}
Here $h:=T/N$ is the step size, $\Delta_hw_k:=w((k+1)h)-w(kh)$ is a Brownian increment, and the truncation function $\eta:\mathbb R^n\times \mathbb R^n \times [0,\infty)\to(0,1]$ is given by
\begin{align}\label{eq:defeta}
\eta(r,v,h):=\min\Big(1\,,\;\frac{C_\eta}{\max(\|v\|,\|v\|^2,\|a(r,v)\|)\,h}\,\Big),
\end{align}
with $C_\eta:=1/(4C_g^3)$ depending on the constant $C_g$ introduced in \eqref{eq:C_g}.
Let us remark that the choice of considering both $\|v\|$ and $\|v\|^2$ in the maximum in \eqref{eq:defeta}, and not solely $\|v\|^2$, is mainly for convenience reasons as it guarantees a precise control of the norm of $\eta(r_k,v_k,h)v_kh$ also for small values of $\|v_k\|$, independently of the step size $h$; alternatively one could impose a suitable step size restriction.

In the following remarks we sketch the main concepts the scheme~\eqref{eq:approximation} is based on. As these are well-known in the respective scientific communities, we refer to the mentioned literature and the references therein for more detailed expositions. The concepts presented in the first two remarks are standard in the context of DAEs.
\begin{remark}[Gear-Gupta-Leimkuhler formulation]
Compared to the SDAE \eqref{eq:mainSDAE} the scheme \eqref{eq:approximation} involves the additional Lagrange multiplier $\kappa_{k+1}$ in \eqref{eq:approximationup} as well as the additional constraint $Dg(r_{k+1})v_{k+1}=0$ in \eqref{eq:approximationlow}.
It is clear that the solution $(r,v,\mu)$ to \eqref{eq:mainSDAE} also solves the SDAE
\begin{equation}\label{eq:GGLSDAE}
	\begin{aligned}
	\dl r(t)&= v(t)\,\dl t+M^{-1}\nabla g(r(t))\,\dl \nu(t)\\
	M\dl v(t)&= a(r(t),v(t))\,\dl t+ B(r(t),v(t))\,\dl w(t)+\nabla g(r(t))\,\dl\mu(t)\\
	g(r(t))&=0\\
	Dg(r(t))v(t)&=0
	\end{aligned}
\end{equation}
if we choose the integrator process $\nu$ to be identically zero.
This is the so-called Gear-Gupta-Leimkuhler (GGL) reformulation of \eqref{eq:mainSDAE}.
Observe that the scheme \eqref{eq:approximation} is a discrete version of \eqref{eq:GGLSDAE}, with $\kappa_{k+1}$ and $\lambda_{k+1}$ corresponding to the infinitesimal increments $\dl\nu(t)$ and $\dl\mu(t)$.
The GGL stabilization is a standard index reduction technique for deterministic mechanical systems, compare \cite{gear:p:1985,hairer:b:1989} and \cite[Chapter~VII]{hairer:b:1991}.
Loosely speaking, it reduces the influence that a perturbation of the constraint has on the Lagrange multiplier, see \cite[Chapter 1]{hairer:b:1989}. 
\end{remark}

\begin{remark}[Half-explicit schemes] 
		The main idea of half-explicit methods for deterministic DAEs is to discretize the differential variables in an explicit manner and only the algebraic variables in an implicit manner, see~\cite{lubich:p:1989,ostermann:p:1993} and \cite[Section~VII.6]{hairer:b:1991}. Thereby, the dimension of the system of equations which has to be solved implicitly in each time step is kept minimal. 
		In our reformulated SDAE~\eqref{eq:GGLSDAE} one can interpret $r$, $v$ as the ``differential variables'' and the formal time derivatives of $\nu$, $\mu$ as the ``algebraic variables''. 
		As a particularly useful consequence of our half-explicit approach, the position $r_{k+1}$ in \eqref{eq:approximation} is stochastically independent of the Brownian increment $\Delta_hw_k$.
\end{remark}

Next, we explain the concept of truncating or taming, which has been used and studied extensively in the last years in the context of SDEs with non-globally Lipschitz continuous coefficients, cf.~\cite{beyn:p:2016,fang:p:2016,higham:p:2002,hutzenthaler:p:2014,hutzenthaler:p:2015,hutzenthaler:p:2012,mao:p:2015,mao:p:2016,sabanis:p:2013}.

\begin{remark}[Truncated and tamed schemes]
	The concept of truncation or taming is used to obtain strongly convergent explicit methods for SDEs whose coefficients do not fulfill global Lipschitz conditions. 
	In the one-dimensional setting it has been shown that the classical Euler-Maruyama scheme may diverge if the coefficient functions are of superlinear polynomial growth, see \cite{hutzenthaler:p:2011} for details.
	The divergence follows from the existence of a sequence of events of exponentially small probability on which the numerical approximations grow at least double-exponentially fast.
	The idea of tamed or truncated schemes is to adjust the coefficient functions in such a way that this growth behaviour is avoided while the adjustment is negligible with high probability.
	In our setting, the presence of the truncation term $\eta(r_k,v_k,h)$ in \eqref{eq:approximation} enables the derivation of suitable moment bounds for the numerical solution, but it also ensures the existence of a numerical solution at all, compare the example in Section~\ref{subsec:pendel}.
	Let us further note that the standard drift-truncated explicit Euler scheme for the inherent SDE~\eqref{eq:inherentSDE} involves the truncation function
	\begin{align*}
	\tilde \eta(r,v,h)=\min\Big(1,\Big[\big(\|v\|^2+\|P_{M}(r)a(r,v)+Dg(r) G_M^{-1}(r)D^2g(r)(v,v)\|^2\big)^{1/2}h\Big]^{-1}\Big)
	\end{align*}
compare \cite[Section~3.6]{hutzenthaler:p:2015}.
The choice \eqref{eq:defeta} for our half-implicit scheme is a simplification of this truncation function, making use of the boundedness properties of the constraint function $g$ from Assumption \ref{ass:g}. This simplification reduces the computational effort significantly since the calculations of $P_M$, $G_M^{-1}$ and $D^2g$ are avoided in each time step.
\end{remark}

\subsection{Solvability of the scheme and Lagrange multiplier estimates}\label{sec:solvingapproximation}

	Here we verify the unique solvability of the half-explicit drift-truncated Euler scheme \eqref{eq:approximation} and derive a suitable decomposition as well as norm estimates for the Lagrange multipliers $\kappa_{k+1}$, $\lambda_{k+1}$.
	Let us stress that, although the main result in this subsection, Theorem~\ref{theo:mainonestep}, is formulated in a deterministic setting, it is tailor-made for the analysis of our stochastic problem, cf.~Remark~\ref{rem:solvingapproximation}.
	Its application to the scheme~\eqref{eq:approximation} is described in Corollary~\ref{cor:applyingonestep}.
	
	For $r,v\in\bR^n$, $h\in(0,\infty)$ and $w\in\bR^\ell$, consider the system of equations
	\begin{subequations}\label{eq:nonlinearsystem}
		\begin{align}
			\begin{split}\label{eq:nonlinearupperconstraint}
				\hat r&=r+\eta(r,v,h)vh+M^{-1}\nabla g(r)\hat \lm\\
				g(\hat r)&=0
			\end{split}\\
			\begin{split}\label{eq:nonlinearlowerconstraint}
				\hat v&=v+M^{-1}\Big[\eta(r,v,h)a(r,v)h+B(r,v)w+\nabla g(r)\hat \lambda\Big]\\
				Dg(\hat r)\hat v&=0
			\end{split}
		\end{align}	
	\end{subequations}
	whose solution consists of the points $(\hat r,\hat v)\in\bR^n\times\bR^n$ and the Lagrange multipliers $(\hat\lm,\hat\lambda)\in\bR^m\times\bR^m$.

	\begin{theorem}[Solvability, Lagrange multiplier estimates]\label{theo:mainonestep}
		Let Assumptions~\ref{ass:g} \linebreak and \ref{ass:aB} be fulfilled, let $C_g\in[1,\infty)$ be the constant given by \eqref{eq:C_g} and set $\Ceta:=1/(4C_g^3)$, $\Clm:=1/(8C_g^4)$.  Let $\eta\colon\mathbb R^n\times \mathbb R^n \times (0,\infty)\to (0,1]$ be a continuous function satisfying 
		\begin{align}\label{eq:assumptioneta}
			&\sup_{(x,y)\in (T\mathcal M)^\varepsilon,\,h>0} \Big(\eta(x,y,h)\,\max\big(\|y\|,\|y\|^2,\|a(x,y)\|\big)\,h\Big) \leq\Ceta.
		\end{align}
		Then there exists an open neighborhood $\cD\subset (T\cM)^\varepsilon\times (0,\infty)$ of $T\cM\times (0,\infty)$ such that, for all $(r,v,h,w)\in\cD\times \bR^\ell$, the system \eqref{eq:nonlinearsystem} has a unique solution
		\begin{align}\label{eq:discretesolution}
			(\hat r,\hat v,\hat \lm,\hat \lambda)=\big(\hat r(r,v,h),\hat v(r,v,h,w),\hat \lm(r,v,h),\hat \lambda(r,v,h,w)\big)
		\end{align}
		in $T\cM\times B_{\Clm}(0)\times\bR^m$, where $B_{\Clm}(0)$ is the open ball in $\bR^m$ with radius $\Clm$ and center zero. 
		The solution \eqref{eq:discretesolution} depends continuously on $(r,v,h,w)\in\cD\times \bR^\ell$,
		and the Lagrange multiplier $\hat \lambda=\hat \lambda(r,v,h,w)\in\bR^m$ can be represented as
		\begin{align}
			\begin{split}\label{eq:nonlineardecompose}
				\hat\lambda(r,v,h,w)&=-G^{-1}_M(r)\Big\{Dg(r)M^{-1}\big[\eta(r,v,h)a(r,v) h+B(r,v)w\big]\\
				&\quad+\eta(r,v,h)D^2g(r)(v,v)h\Big\}+\lambdarem(r,v,h)+\Lambdarem(r,v,h)w
			\end{split}
		\end{align}
		with remainder terms $\lambdarem(r,v,h)\in\bR^m$, $\Lambdarem(r,v,h)\in\bR^{m\times\ell}$ depending continuously on $(r,v,h)\in\cD$.
		Moreover, there exist constants $C,\,p_\lambda,\,p_\Lambda\in(0,\infty)$ 
		such that
		\begin{align}\label{eq:estLambdarem}
			\begin{split}
				\|\hat \lm(r,v,h)\|&\leq C\min\Big(\big(1+\|(r,v)\|^{2}\big)h^2,\;1\,\Big),\\
				\|\lambdarem(r,v,h)\|&\leq C\min\Big(\big(1+\|(r,v)\|^{p_\lambda}\big)h^2,\;1\,\Big),\\
				\|\Lambdarem(r,v,h)\|&\leq C\min\Big(\big(1+\|(r,v)\|^{p_\Lambda}\big)h\,,\;1+\|(r,v)\|\,\Big).
			\end{split}
		\end{align}
		for all $(r,v,h)\in\cD$.	
	\end{theorem}

	\begin{proof}
		For the sake of clarity we divide the proof into several steps: After defining a suitable neighborhood $\cD$ of $T\cM\times (0,\infty)$ in Step 1, we show  in Step 2 the existence of a solution $(\hat r,\hat\lm)\in \cM\times B_{\Clm}(0)$ to the system \eqref{eq:nonlinearupperconstraint} as well as the first estimate in \eqref{eq:estLambdarem}. The uniqueness of this solution is verified in Step 3. In Step 4 we show the existence of a unique solution $(\hat v,\hat\lambda)\in\bR^n\times\bR^m$ to the system \eqref{eq:nonlinearlowerconstraint}. The decomposition \eqref{eq:nonlineardecompose} of $\hat\lambda$ and the corresponding estimates in \eqref{eq:estLambdarem} are derived in Step 5. In the final Step 6 we prove that $\hat\lm,\,\lambdarem\in\bR^m$ and $\Lambdarem\in\bR^{m\times\ell}$ depend continuously on $(r,v,h)\in\cD$. For the sake of readability, we use the notation $a\wedge b$ and $a\vee b$ for the minimum and maximum of numbers $a,b\in\bR$.
		\smallskip
					
		\noindent
		{\bf Step 1:}
		We begin by choosing an open neighborhood $\cD\subset (T\cM)^\varepsilon\times (0,\infty)$ of $T\cM\times (0,\infty)$ is such a way that
		\begin{align}\label{eq:defD}
			\sup_{(x,y,h)\in\cD}\max\left(\frac{\big\|g(x)+\eta(x,y,h)h
			Dg(x)y\big\|}{\eta(x,y,h)\big(h\wedge \Ceta\big)^2},\frac{\big\|Dg(x)y\big\|}{\big(h\wedge \Ceta\big)^2}\right)
			< \frac{C_g}{2}.
		\end{align}
		Note that the term $\max(...\,,\,...)$ on the left hand side of \eqref{eq:defD} is continuous as a function of $(x,y,h)\in(T\cM)^\varepsilon\times (0,\infty)$ and equals zero for $(x,y,h)\in T\cM\times (0,\infty)$. 
		Thus we can define $\cD$ as the preimage of the open interval $(0,C_g/2)$ with respect to this function.  
		The inequality \eqref{eq:defD} is used in Step 2 and in Step 5.
		We remark that similar assumptions are used in the analysis of numerical schemes for deterministic DAEs, see, e.g., \cite[Theorem VII.4.1]{hairer:b:1991}. 
		
		Unless stated otherwise, we assume throughout this proof that $(r,v,h,w)\in\cD\times\bR^\ell$ is given and fixed.
		
		\smallskip
		\noindent
		\textbf{Step 2:} 
			In order to find a solution $(\hat r,\hat\lm)$ to \eqref{eq:nonlinearupperconstraint}, we use a homotopy ansatz as commonly used in the DAE context, see for instance \cite[Theorem 4.1]{hairer:b:1989}. It  
			is particularly fruitful in our setting due to the presence of the truncation function~$\eta$.
			Note that the system \eqref{eq:nonlinearupperconstraint} is not influenced by $(\hat v,\hat \lambda)$.
			Our goal is to apply the globalized implicit function theorem (gIFT),  see Theorem \ref{theo:globalIFT}, to the function 
			$F\colon(-\delta,1+\delta)\times\bR^m\to \mathbb R^m$ defined by
			\begin{align}\label{eq:defF}
				F(\tau,\gamma):=g\big(r+\eta(r,v,h) v h+M^{-1}\nabla g(r) \gamma\big)+(\tau-1)g\big(r+\eta(r,v,h) v h\big),
			\end{align}
			where $\delta>0$ is an arbitrary small positive number.
			To this end, we are going to show that, for all $(\tau,\gamma)\in[0,1]\times B_{2\Clm}(0)$ with $F(\tau,\gamma)=0$, the Jacobian matrix $D_\gamma F(\tau,\gamma)\in\bR^{m\times m}$ is invertible and
			\begin{align}\label{eq:step2_1}
				\begin{split}
					\big\|(D_\gamma F(\tau,\gamma))^{-1}D_\tau F(\tau,\gamma)\big\|
					&< C_g^2\big(\eta(r,v,h)(\Ceta\wedge h)^2+\eta(r,v,h)^2\|v\|^2 h^2\big)\\
					&\leq 2C_g^2\Ceta^2=\Clm
				\end{split}
			\end{align}
			Since $F(0,0)=0$, Theorem~\ref{theo:globalIFT} (gIFT) then implies that there exists a continuously differentiable function $\gamma\colon[0,1]\to B_{\Clm}(0)\subset\bR^m$ such that $\gamma(0)=0$ and $F(\tau,\gamma(\tau))=0$ for all $\tau\in[0,1]$. 
			Thus, we obtain a solution $(\hat r,\hat\lm)\in\cM\times B_{\Clm}(0)$ to \eqref{eq:nonlinearupperconstraint} by setting $\hat\lm=\hat\lm(r,v,h):=\gamma(1)$. 
			In view of \eqref{eq:step2_1} and the identity $\gamma(\tau)=\int_0^\tau -\big(D_\gamma F(s,\gamma(s))\big)^{-1} D_\tau F(s,\gamma(s))\,\dl s$ we also have the inequality
			\begin{align}\label{eq:estihatnu}
				\|\hat \lm(r,v,h)\|< C_g^2 \big(\eta(r,v,h)(\Ceta\wedge h)^2 +\eta(r,v,h)^2\|v\|^2 h^2\big)\leq \Clm,
			\end{align}
			which implies the first estimate in \eqref{eq:estLambdarem} and will be useful later on.
			In the sequel, let $(\tau,\gamma)\in[0,1]\times B_{2\Clm}(0)$ with $F(\tau,\gamma)=0$ be fixed.
			
			We first show that $D_\gamma F(\tau,\gamma)\in\bR^{m\times m}$ is invertible and estimate the norm of its inverse.
			Using the chain rule, a first order Taylor expansion of $Dg$ at $r$, and the identity $G_M(r)=Dg(r)M^{-1}\nabla g(r)\in\bR^{m\times m}$, we obtain
			\begin{align}\label{eq:nonlinearestidefxi}
				\begin{split}
					D_\gamma F(\tau,\gamma)
					&=Dg\big(r+\eta(r,v,h)v h+M^{-1}\nabla g(r) \gamma\big)M^{-1}\nabla g(r)\\
					&= G_M(r)+D^2g(\chi)\big(\eta(r,v,h) v h+M^{-1}\nabla g(r)\gamma,M^{-1}\nabla g(r)\bbullet\big)
				\end{split}
			\end{align}
			for some $\chi=r+\alpha(\eta(r,v,h)v h+M^{-1}\nabla g(r)\gamma)\in\bR^n$ with $\alpha\in[0,1]$.
			Here and below we denote for $x,y\in\mathbb R^n$ and $B\in\bR^{n\times m}$ by $D^2g(x)(y,B\bbullet)\in\bR^{m\times m}$ the matrix corresponding to the linear operator $\bR^m\ni z\mapsto D^2g(x)(y,Bz)\in\bR^m$.	
			Note that $G_M(r)\in\bR^{m\times m}$ is invertible since $r\in\cM^\varepsilon$. 
			Thus, \eqref{eq:nonlinearestidefxi} implies that $D_\gamma F(\tau,\gamma)\in\bR^{m\times m}$ is invertible if, and only if, the matrix
			\begin{align}\label{eq:defAstep2}
				\begin{split}
					G^{-1}_M(r)D_\gamma F(\tau,\gamma)&=\Id+G_M^{-1}(r)D^2g(\chi)\big(\eta(r,v,h) v h+M^{-1}\nabla g(r)\gamma,M^{-1}\nabla g(r)\bbullet\big)\\&=:\Id+A(r,v,h,\gamma)
				\end{split}
			\end{align}
			is invertible. 
			Recall that if $A\in\mathbb R^{m\times m}$ is such that $\|A\|_{\sL(\bR^m)}<1$, then $\Id+A\in\bR^{m\times m}$ is invertible with inverse given by the Neumann series
			$(\Id+A)^{-1}=\sum_{i=0}^\infty (-A)^i$.
			Due to our assumptions on $g$, $\eta$ and since $\|\gamma\|<2\Clm$, we have
			\begin{align}\label{eq:estAstep2}
				\begin{split}
					\|A(r,v,h,\gamma)\|_{\sL(\bR^m)}
					&\leq 
					\big\|G_M^{-1}(r)D^2g(\chi)\big(\eta(r,v,h) v h,M^{-1}\nabla g(r)\bbullet\big)\big\|_{\sL(\bR^m)}\\
					& \quad
					+\big\|G_M^{-1}(r)D^2g(\chi)\big(M^{-1}\nabla g(r)  \gamma,M^{-1}\nabla g(r)\bbullet\big)\big\|_{\sL(\bR^m)}\\
					&\leq
					C_g^3\eta(r,v,h)\|v\|h+C_g^4\|\gamma\|\\
					&<
					C_g^3\Ceta+2C_g^4\Clm=\frac 12.
				\end{split}
			\end{align}
			As a consequence, $D_\gamma F(\tau,\gamma)\in\bR^{m\times m}$ is invertible and
			\begin{align}\label{eq:estDgammaF}
				\begin{split}
					\big\|(D_\gamma F(\tau,\gamma))^{-1}\big\|_{\sL(\bR^m)}
					&=
					\Big\|\sum_{i=0}^\infty\big(\!-\!A(r,v,h,\gamma)\big)^i G_M^{-1}(r)\Big\|_{\sL(\bR^m)}\\
					&<
					\sum_{i=0}^\infty\frac{1}{2^i}\big\|G_M^{-1}(r)\big\|_{\sL(\bR^m)}\leq 2C_g.
				\end{split}
			\end{align}
			
			Next, we estimate the norm of $D_\tau F(\tau,\gamma)=\frac{\partial}{\partial\tau}F(\tau,\gamma)\in\bR^m$. 
			A second order Taylor expansion of $g$ at $r$ yields
			\begin{align*}
				D_\tau F(\tau,\gamma)&=g\big(r+\eta(r,v,h) v h\big)\\
				&=g(r)+Dg(r)\eta(r,v,h)vh+\frac 12 D^2g(\zeta)\big(\eta(r,v,h)vh,\eta(r,v,h)vh\big)
			\end{align*}
			for some 
			$\zeta=r+\beta\,\eta(r,v,h) v h\in\bR^n$ 
			with $\beta\in[0,1]$.
			Therefore, using also the estimate \eqref{eq:defD} from the definition of $\cD$, we have
			\begin{align}\label{eq:estDtauF}
				\|D_\tau F(\tau,\gamma)\|
				&<  \frac{C_g}{2}\eta(r,v,h) (\Ceta\wedge h)^2+\frac{C_g}{2}\eta(r,v,h)^2\|v\|^2h^2\leq C_g \Ceta^2.
			\end{align}
			The combination of \eqref{eq:estDgammaF} and \eqref{eq:estDtauF} finally yields \eqref{eq:step2_1}.
			We postpone the verification of the continuous dependence of $\hat \lm$ on $(r,v,h)$ to Step 6 since, due the lack of differentiability of $\eta$, it involves the implicit function theorem in the version of Theorem \ref{theo:localIFT} as well.
		\smallskip
		
		\noindent
		\textbf{Step 3:} 
			Here we show that $(\hat r,\hat\lm)$ from Step 1 is the \emph{only} solution to \eqref{eq:nonlinearupperconstraint} in $\bR^n\times B_{\Clm}(0)\subset\bR^n\times\bR^m$. An equivalent formulation in terms of the function $F$ from \eqref{eq:defF} is that $\hat\lm$ is the only element in $B_{\Clm}(0)\subset\bR^m$ satisfying
			$
			F(1,\hat\lm)=g\big(r+\eta(r,v,h) v h+\nabla g(r)\hat\lm\big)=0.
			$
			By the last statement of Theorem \ref{theo:globalIFT} (gIFT), it therefor suffices to show that
			\begin{align}\label{eq:step3_1}
				F(0,\gamma)\neq0\;\text{ for all }\gamma\in\bR^m\setminus\{0\}\text{ with }\|\gamma\|<2\Clm.
			\end{align}
			To this end, let $\gamma\in\bR^m\setminus\{0\}$ with $\|\gamma\|<2\Clm$ be fixed and assume that $F(0,\gamma)=0$.
			Using the definition \eqref{eq:defF} of $F$, the mean value theorem applied to $g$, a first order Taylor expansion of $Dg$ at $r$, and the identity $G_M(r)=Dg(r)M^{-1}\nabla g(r)$, we obtain
			\begin{align*}
				0
				&=
				g\big(r+\eta(r,v,h) v h+M^{-1}\nabla g(r)\gamma\big)-g\big(r+\eta(r,v,h) v h\big)\\
				&=
				Dg\big(r+\eta(r,v,h) v h+\alpha M^{-1}\nabla g(r)\gamma\big)M^{-1}\nabla g(r)\gamma\\
				&=
				G_M(r)\gamma+ D^2g(\xi)\big(\eta(r,v,h) v h+\alpha M^{-1}\nabla g(r)\gamma,M^{-1}\nabla g(r)\gamma\big)
			\end{align*}
			for some $\alpha\in[0,1]$ and 
			$\xi=r+\beta\big(\eta(r,v,h) v h+\alpha M^{-1}\nabla g(r)\gamma\big)$ with $\beta\in[0,1]$.
			Since $G_M(r)\in\bR^{m\times m}$ is invertible we have
			\begin{align*}
				\gamma=-G_M^{-1}(r)D^2g(\xi)\big(\eta(r,v,h) v h+\alpha M^{-1}\nabla g(r)\gamma,M^{-1}\nabla g(r)\gamma\big)
			\end{align*}
			As a consequence, the bounds $\|G_M^{-1}(r)\|_{\sL(\bR^m)}\leq C_g$, $\|M^{-1}\nabla g(r)\|_{\sL(\bR^m,\bR^n)}\leq C_g$ and $\|D^2g(\xi)\|_{\sL^{(2)}(\bR^n,\bR^m)}\leq C_g$, following from Assumption \ref{ass:g} and \eqref{eq:C_g}, together with the assumption that $\|\gamma\|< 2\Clm$ imply
			\begin{align*}
				\|\gamma\|
				\leq 
				C_g^3 \big(\Ceta+C_g\|\gamma\|\big)\,\|\gamma\|
				\leq
				\Big(\frac14+2C_g^4\Clm\Big)\|\gamma\|
				\leq 
				\frac{\|\gamma\|}2.
			\end{align*}
			This readily yields $\gamma=0$
			and hence \eqref{eq:step3_1}.
		\smallskip
		
		\noindent
		\textbf{Step 4:} 
			The next step is to show that there exists a unique solution to the system \eqref{eq:nonlinearlowerconstraint} given that \eqref{eq:nonlinearupperconstraint} has been solved as in Step 1. 
			It suffices to show that, given 
			$(\hat r,\hat\lm)\in
			\bR^n\times\bR^m$ from Step 1, there exists a unique $\hat\lambda=\hat\lambda(r,v,h,w)\in\bR^m$ satisfying
			\begin{align}\label{eq:lowereqmaintheo}
				\begin{split}
					0
					=Dg(\hat r)\hat v&=Dg(\hat r)v+Dg(\hat r)M^{-1}\big[\eta(r,v,h)a(r,v)h+B(r,v) w\big]\\
					&\quad+Dg(\hat r)M^{-1}\nabla g(r)\hat \lambda.
				\end{split}
			\end{align}
			By \eqref{eq:nonlinearestidefxi} and \eqref{eq:defAstep2} with $\gamma=\hat\lm\in B_{\Clm}(0)\subset\bR^m$ we have
			\begin{align}\label{eq:step4_1}
				Dg(\hat r)M^{-1}\nabla g(r)
				=
				D_\gamma F(1,\hat\lm)
				=
				G_M(r)\big(\Id+A(r,v,h,\hat\lm)\big)\;\;(\in\bR^{m\times m}),
			\end{align}
			and this matrix is invertible. 
			Indeed, $G_M(r)\in\bR^{m\times m}$ is invertible since $r\in\cM^\varepsilon$, and $\Id+A(r,v,h,\hat\lm)\in\bR^{m\times m}$ is invertible due to \eqref{eq:estAstep2}, with inverse given by the Neumann series 
			\begin{align}\label{eq:step4_2}
				\big(\Id+A(r,v,h,\hat\lm)\big)^{-1}=\sum_{i=0}^\infty\big(\!-\!A(r,v,h,\hat\lm)\big)^i\;\;(\in\bR^{m\times m}).
			\end{align}
			Combining \eqref{eq:lowereqmaintheo} and \eqref{eq:step4_1} thus yields
			\begin{align}\label{eq:step4_defhatlambda}
				\begin{split}
					\hat\lambda
					&=
					-\big(\Id+A(r,v,h,\hat\lm)\big)^{-1}G_M^{-1}(r)\\
					&\quad\times
					\Big\{Dg(\hat r)v+Dg(\hat r)M^{-1}\big[\eta(r,v,h)a(r,v)h+B(r,v) w\big]\Big\}.
				\end{split}
			\end{align}
			\smallskip
			
			\noindent
			{\bf Step 5:} 
			Here we define and estimate the remainder terms $\lambdarem(r,v,h)\in\bR^m$ and $\Lambdarem(r,v,h)\in\bR^{m\times\ell}$ in the decomposition \eqref{eq:nonlineardecompose} of the Lagrange multiplier $\hat\lambda=\hat\lambda(r,v,h,w)\in\bR^{m}$ given by \eqref{eq:step4_defhatlambda}.
			
			Concerning the terms in the second line of \eqref{eq:step4_defhatlambda}, a second order Taylor expansion of $Dg$ at $r$ yields
			\begin{align}\label{eq:step5_1}
				\begin{split}
					Dg(\hat r)v&=Dg(r)v+D^2g(r)\big(\eta(r,v,h) vh+M^{-1}\nabla g(r)\hat \lm,\,v\big)\\
					&\quad+\frac 12D^3g(\vartheta)\Big(\eta(r,v,h) v h+M^{-1}\nabla g(r)\hat \lm,\,\eta(r,v,h) v h+M^{-1}\nabla g(r) \hat \lm,\,v\Big),
				\end{split}
			\end{align}
			where $\vartheta=r+\beta \big(\eta(r,v,h) vh+M^{-1}\nabla g(r)\hat \lm\big)\in\bR^n$ for some $\beta\in[0,1]$, and a first order Taylor expansion of $Dg$ at $r$ gives
			\begin{align}\label{eq:step5_2}
				Dg(\hat r)=Dg(r)+D^2g(\chi)\big(\eta(r,v,h) v h+M^{-1}\nabla g(r)\hat \lm,\bbullet\big)\;\;(\in\bR^{m\times n}),
			\end{align}
			where $\chi=r+\alpha(\eta(r,v,h)v h+M^{-1}\nabla g(r)\hat\lm)\in\bR^n$, $\alpha\in[0,1]$, is as in \eqref{eq:nonlinearestidefxi} with $\gamma=\hat\lm$.
			Plugging \eqref{eq:step5_1}, \eqref{eq:step5_2} into \eqref{eq:step4_defhatlambda} and rearranging the involved terms shows that we can define $\lambdarem=\lambdarem(r,v,h)\in\bR^m$ and $\Lambdarem=\Lambdarem(r,v,h)\in\bR^{m\times\ell}$ fulfilling \eqref{eq:nonlineardecompose} by setting
			\begin{subequations}\label{eq:defLambdarem}
			\begin{align}
				\begin{split}
					\lambdarem&:=\Big(\Id-\big(\Id+A(r,v,h,\hat \lm)\big)^{-1}\Big)G_M^{-1}(r)\Big\{Dg(r)M^{-1}\eta(r,v,h) a(r,v) h\\
					&\qquad+D^2g(r)\big(\eta(r,v,h) vh,v\big)\Big\}\\
					&\quad-\big(\Id+A(r,v,h,\hat \lm)\big)^{-1}G_M^{-1}(r)\Big\{Dg(r)v+D^2g(r)\big(M^{-1}\nabla g(r)\hat \lm,v\big)\\
					&\qquad+D^2g(\chi)\Big(\eta(r,v,h) v h+M^{-1}\nabla g(r)\hat \lm,\,M^{-1}\eta(r,v,h) a(r,v) h\Big)\\
					&\qquad+\frac12D^3g(\vartheta)\Big(\eta(r,v,h) v h+M^{-1}\nabla g(r)\hat\lm,\,\eta(r,v,h) v h+M^{-1}\nabla g(r)\hat\lm,\,v\Big)\Big\},
				\end{split}\\
				\begin{split}
					\Lambdarem&:=\Big(\Id-\big(\Id+A(r,v,h,\hat \lm)\big)^{-1}\Big)G_M^{-1}(r)Dg(r)M^{-1} B(r,v)\\
					&\quad-\big(\Id+A(r,v,h,\hat \lm)\big)^{-1}G_M^{-1}(r)\\
					&\qquad\times D^2g(\chi)\Big(\eta(r,v,h) v h+M^{-1}\nabla g(r)\hat \lm,\,M^{-1}B(r,v)\bbullet\Big).
				\end{split}
			\end{align}
			\end{subequations}
			Next, note that the combination of \eqref{eq:step4_2}, the estimate \eqref{eq:estAstep2} of $\|A(r,v,h,\gamma)\|_{\sL(\bR^m)}$ with $\gamma=\hat\lm$ and  the estimate \eqref{eq:estihatnu} of $\|\hat\lm\|=\|\hat\lm(r,v,h)\|$  yields 
			\begin{align}\label{eq:ergaenzung1}
				\big\|\big(\Id + A(r,v,h,\hat\lm)\big)^{-1}\big\|_{\sL(\bR^m)}&\leq 2
			\end{align}
			as well as
			\begin{align}\label{eq:ergaenzung2}
			\begin{split}
				\big\|\Id-\big(\Id+A(r,v,h,\hat\kappa)\big)^{-1}\big\|_{\sL(\bR^m)}
				&\leq
				\|A(r,v,h,\hat\lm)\|_{\sL(\bR^m)}\sum_{i=1}^\infty\frac1{2^{i-1}}\\
				&\leq\frac52C_g^3\eta(r,v,h)\big((\Ceta\wedge h)+\|v\|h\big).
				\end{split}
			\end{align}
			In \eqref{eq:ergaenzung2} we use the fact that \eqref{eq:estihatnu} and \eqref{eq:estAstep2} imply 
			\begin{align*}
			\|\hat\kappa\|&<C_g^2\Ceta\eta(r,v,h)\big((\Ceta\wedge h)+\|v\|h\big)\\
			&=\frac1{4C_g}\eta(r,v,h)\big((\Ceta\wedge h)+\|v\|h\big),\\
			\|A(r,v,h,\hat\kappa)\|_{\sL(\bR^m)}&\leq C_g^3\eta(r,v,h)\|v\|h+\frac14C_g^3\eta(r,v,h)\big((\Ceta\wedge h)+\|v\|h\big)\\
			&\leq\frac54C_g^3\eta(r,v,h)\big((\Ceta\wedge h)+\|v\|h\big).
			\end{align*}
			Combining \eqref{eq:defLambdarem}, \eqref{eq:ergaenzung1}, \eqref{eq:ergaenzung2} and using
			the estimate $\|Dg(r)v\|\leq\frac{C_g}2(\Ceta\wedge h)^2$ due to \eqref{eq:defD}, we obtain
			\begin{subequations}\label{eq:addestilambda}
			\begin{align}
				\begin{split}
					\|\lambdarem\|
					&\leq
					C'\big\|\Id-(\Id+A)^{-1}\|_{\sL(\bR^m)}\Big\{\eta(r,v,h)\|a(r,v)\|h+\eta(r,v,h)\|v\|^2h\Big\}\\
					&\quad+C'\Big\{\|Dg(r)v\|+\|\hat\lm\|\|v\|+\Big(\eta(r,v,h)\|v\|h+\|\hat\lm\|\Big)\eta(r,v,h)\|a(r,v)\|h\\
					&\qquad +\Big(\eta(r,v,h)\|v\|h+\|\hat\lm\|\Big)^2\|v\|\Big\}\\
					&\leq C''\Big((\Ceta\wedge h)^2
					+ 
					\eta(r,v,h)\|a(r,v)\|h(\Ceta\wedge h) + \eta^2(r,v,h)\|v\|\|a(r,v)\|h^2\\
					&\quad+ 
					\eta^3(r,v,h)\|v\|^2\|a(r,v)\|h^3
					+
					\eta(r,v,h)\big(\|v\|+\|v\|^2\big)h(\Ceta\wedge h)^2\\
					&\quad+
					\eta^2(r,v,h)\|v\|^3h^2 + \eta^3(r,v,h)\big(\|v\|^3+\|v\|^4\big)h^3 + \eta^4(r,v,h)\|v\|^5h^4\Big)
				\end{split}\\
				\begin{split}
					\|\Lambdarem w\|
					&\leq
					C'\Big(\big\|\Id-(\Id+A)^{-1}\|_{\sL(\bR^m)}+\eta(r,v,h)\|v\|h+\|\hat\lm\|\Big)\|B(r,v)w\|\\
					&\leq
					C''\Big((\Ceta\wedge h)+\eta(r,v,h)\|v\|h+ \eta^2(r,v,h)\|v\|^2h^2\Big)\|B(r,v)w\|,
				\end{split}
			\end{align}
			\end{subequations}
			where $A:=A(r,v,h,\hat\lm)\in\bR^{m\times m}$ and where $C',\,C''\in(0,\infty)$ are suitable constants that do not depend on $(r,v,h,w)\in\cD\times\bR^\ell
			$.
			Now the second and third estimate in \eqref{eq:estLambdarem} follow directly from our assumptions on $g$, $a$, $B$ and $\eta$.
		
		\smallskip
		
			\noindent
			\textbf{Step 6:}	
			We finally show that the 
			Lagrange multiplier $\hat\kappa(r,v,h)\in\bR^m$
			and the remainder terms $\lambdarem(r,v,h)\in\bR^m$, $\Lambdarem(r,v,h)\in\bR^{m\times\ell}$ in \eqref{eq:nonlineardecompose} depend continuously on $(r,v,h)\in\cD$.
			To this end, we use a slight generalization of the classical local Implicit Function Theorem, see Theorem~\ref{theo:localIFT}, and consider the function
			$\tilde F\colon \cD\times \bR^m\to \bR^m$ defined by
			\begin{align*}
				\tilde F(\tilde r,\tilde v,\tilde h,\gamma):=g\big(\tilde r+\eta(\tilde r,\tilde v,\tilde h)vh+\nabla g(\tilde r)\gamma\big).
			\end{align*}
			Obviously, $\tilde F$ is a continuous function which is continuously differentiable in the last argument.
			As in Step 1 one sees that $D_\gamma\tilde F(r,v,h,\hat\kappa(r,v,h))\in\bR^{m\times m}$ is invertible.
			Since $\hat\kappa(r,v,h)\in B_{\Clm}(0)$ and $\tilde F(r,v,h,\hat\kappa(r,v,h))=0$, Theorem~\ref{theo:localIFT} ensures the existence of a continuous function $\tilde\gamma\colon V_1\to B_{\Clm}(0)\subset\bR^m$, defined on an open neighborhood $V_1\subset \cD$ of $(r,v,h)$, such that $\tilde\gamma(r,v,h)=\hat\kappa(r,v,h)$ and $\tilde F(\tilde r,\tilde v,\tilde h,\tilde\gamma(\tilde r,\tilde v,\tilde h))=0$ for all $(\tilde r,\tilde v,\tilde h)\in V_1$.
			Moreover, by Step 3 we know for all $(\tilde r,\tilde v,\tilde h)\in V_1$ that $\hat\kappa(\tilde r,\tilde v,\tilde h)$ is the only element in $B_{\Clm}(0)\subset\bR^m$ satisfying $\tilde F(\tilde r,\tilde v,\tilde h,\hat\kappa(\tilde r,\tilde v,\tilde h))=0$.
			Hence, $\tilde\gamma(\tilde r,\tilde v,\tilde h)=\hat\kappa(\tilde r,\tilde v,\tilde h)$ for all $(\tilde r,\tilde v,\tilde h)\in V_1$.
			We obtain that the Lagrange multiplier $\hat{\lm}(r,v,h)\in B_{\Clm}(0)\subset\bR^m$ depends continuously on $(r,v,h)\in\cD$.
			
			Concerning $\lambdarem(r,v,h)$, $\Lambdarem(r,v,h)$,			
			observe that $G_M^{-1}(r)\in\bR^{m\times m}$ depends continuously on $r\in\cM^\varepsilon$ as a consequence of Assumption~\ref{ass:g}. Moreover, $(\Id+A(r,v,h,\gamma))^{-1}\in\bR^{m\times m}$ depends continuously on $(r,v,h,\gamma)\in\cD\times B_{\Clm}(0)\subset \cD\times\bR^m$ as a consequence of the identity $(\Id+A(r,v,h,\gamma))^{-1}=\sum_{i=0}^\infty(-A(r,v,h,\gamma))^i$, the estimate \eqref{eq:estAstep2}, and the definition of $A(r,v,h,\gamma)\in\bR^{m\times m}$ in \eqref{eq:defAstep2}.
			Thus, replacing the Lagrange remainder terms appearing due to Taylor expansions on the right hand side of \eqref{eq:defLambdarem} by their corresponding integral representations,  it follows that $\lambdarem(r,v,h)\in\bR^m$ and $\Lambdarem(r,v,h)\in\bR^{m\times\ell}$ depend continuously on $(r,v,h)\in\cD$ as well.
	\end{proof}
			
		\begin{remark}
			The proof of Theorem~\ref{theo:mainonestep} shows a bit more than stated in the theorem: If we weaken  \eqref{eq:assumptioneta} to 
			$\sup_{(x,y)\in (T\mathcal M)^\varepsilon,\,h>0} \big(\eta(x,y,h) \|y\|h\big) \leq\Ceta$, 
			then the statement remains valid if we replace 
			the second and third inequality in \eqref{eq:estLambdarem} by
			$\|\lambdarem(r,v,h)\|\leq C(1+\|(r,v)\|^{p_\lambda})h^2$ and 
			$\|\Lambdarem(r,v,h)\|\leq C(1+\|(r,v)\|^{p_\Lambda})h$.
			In order to obtain the full estimate \eqref{eq:estLambdarem}, it is sufficient to additionally assume that
			$\sup_{(x,y)\in (T\mathcal M)^\varepsilon,\,h>0} \big(\eta(x,y,h)\max(\|y\|^2,\|a(x,y)\|)\,h\big)$ is finite.
		\end{remark}
		
		The following consequence of Theorem~\ref{theo:mainonestep} is immediate.
		\begin{corollary}\label{cor:applyingonestep}
			Let the assumptions in Section~\ref{subsec:assumptions} be fulfilled,
			let $T\in(0,\infty)$, $N\in\mathbb N$ and $h:=T/N$. 
			Let $C_g\in[1,\infty)$ be the constant given by \eqref{eq:C_g} and set $\Clm:=1/(8C_g^4)$
			Then there exists unique $(\cF_{kh})_{k\in\{0,\ldots,N\}}$-adapted processes 
			$(r_k)_{k\in\{0,\ldots,N\}}$, $(\lm_k)_{k\in\{1,\ldots,N\}}$, $(v_k)_{k\in\{0,\ldots,N\}}$ and $(\lambda_k)_{k\in\{1,\ldots,N\}}$
			with values in $\bR^n$, $B_{\Clm}(0)\subset\bR^m$, $\bR^n$ and $\bR^m$, respectively, 
			solving the the half-explicit drift-truncated Euler scheme \eqref{eq:approximation} with truncation function $\eta$ given by \eqref{eq:defeta}.
			
			Moreover, there exist $(\cF_{kh})_{k\in\{0,\ldots,N\}}$-predictable processes $(\lambdare_k)_{k\in\{1,\ldots,N\}}$ and $(\Lambdare_k)_{k\in\{1,\ldots,N\}}$ with values in $\bR^m$ and $\bR^{m\times\ell}$, respectively, such that
			\begin{align}
				\begin{split}\label{eq:nonlineardecompose_stoch}
					\lambda_{k+1}&=-G^{-1}_M(r_k)\Big\{Dg(r_k)M^{-1}\big[\eta(r_k,v_k,h)\,a(r_k,v_k) h+B(r_k,v_k)\Delta_hw_k\big]\\
					&\quad+\eta(r_k,v_k,h)\,D^2g(r_k)(v_k,v_k)h\Big\}+\lambdare_{k+1}+\Lambdare_{k+1}\Delta_hw_k
				\end{split}
			\end{align}
			for all $k\in\{0,\ldots,N-1\}$.
		\end{corollary}
		
		We end this section with two further remarks.
		
		\begin{remark}\label{rem:solvingapproximation}
			The results in Theorem~\ref{theo:mainonestep} and Corollary~\ref{cor:applyingonestep} are tailor-made for our strong convergence analysis in Section~\ref{sec:strongconv} below.
			For instance, the right hand side in the decomposition \eqref{eq:nonlineardecompose_stoch} can be considered as an approximation of a small increment of the Lagrange multiplier process $(\mu(t))_{t\in[0,T]}$ in Theorem~\ref{theo:solutionandinherent}, perturbed by the remainder terms $\lambdare_{k+1}$ and $\Lambdare_{k+1}\Delta_hw_k$.
			Both the drift part $\lambdare_{k+1}$ and the diffusion part $\Lambdare_{k+1}\Delta_hw_k$ are controlled by suitable powers of $h$ in the estimate \eqref{eq:estLambdarem}.
			Moreover, the crucial moment growth estimate in Lemma \ref{cor:strongrate} is heavily based on the fact that $\Lambdare_{k+1}$ is $\cF_{kh}$-measurable and hence independent of $\Delta_hw_k$, compare the derivation of the auxiliary estimate \eqref{eq:estialphanew}. 
		\end{remark}

	\begin{remark}
		While Theorem~\ref{theo:mainonestep} and Corollary~\ref{cor:applyingonestep} ensure both existence and uniqueness of a solution to our scheme, there may exist several solutions if we omit the boundedness condition on the first Lagrange multiplier.
		This issue is well-known in the literature, cf.~\cite[Section 3.3.5.1]{lelievre:b:2010}, but typically no easily verifiable general criteria to identify the correct solution are available.
		In our case, the explicit and simple boundedness criterion $\sup_{k\in\{1,\ldots,N\}}\|\kappa_k\|<\Clm$ identifies the correct solution, where ``correct'' refers to the fact that the scheme converges in the sense described in Theorem~\ref{theo:mainend} below.
	\end{remark}

\section{Strong convergence}\label{sec:strongconv}

In this section we use Theorem~\ref{theo:mainonestep} and Corollary~\ref{cor:applyingonestep} in order to prove 
the strong convergence of our numerical scheme.
In the sequel we fix $T\in(0,\infty)$. Our main result reads as follows:

\begin{theorem}[Strong convergence]\label{theo:mainend}
	Let the assumptions in Section~\ref{subsec:assumptions} be fulfilled and
	 $(r,v,\mu)$ be the unique strong solution to the SDAE~\eqref{eq:mainSDAE} with initial conditions $r_0,v_0$ 
	 according to Theorem~\ref{theo:solutionandinherent}.
	Set $\Clm:=1/(8C_g^4)$, where $C_g\in[1,\infty)$ is given by \eqref{eq:C_g}, 
	and for $N\in\mathbb N$ set $h:=T/N$. 
	Let 
	$\big((r^N_k,v^N_k,\kappa^N_k,\lambda^N_k)\big)_{k\in\{0,\ldots,N\}}$
	be the unique 
	$T\cM\times B_{\Clm}(0)\times\bR^m$-valued
	solution to the half-explicit drift-truncated Euler scheme~\eqref{eq:approximation} 
	according to Corollary~\ref{cor:applyingonestep}, where $B_{\Clm}(0)$ is the open ball in $\bR^m$ with radius $\Clm$ and center zero and where we set $\kappa_0^N:=\lambda_0^N:=0$.
	Further, let $(\mu^N_k)_{k\in\{0,\ldots,N\}}$ be the $\bR^m$-valued process defined by $\mu^N_k:=\sum_{j=0}^k\lambda_k^N$ and let 
	$\big((\tilde r^N(t),\tilde v^N(t),\tilde \mu^N(t))\big)_{t\in[0,T]}$ 
	be defined by piecewise constant or piecewise linear interpolation of 
	$\big((r^N_k,v^N_k,\mu^N_k)\big)_{k\in\{0,\ldots,N\}}$, cf.\ Remark~\ref{rem:interpol} below.
	Then we have
	\begin{align*}
	\lim_{N\to\infty}
	\bE\bigg[\sup_{t\in[0,T]}\Big(
	\big\|r(t)-\tilde r^N(t)\big\|^p
	+\big\|v(t)-\tilde v^N(t)\big\|^p
	+\big\|\mu(t)-\tilde \mu^N(t)\big\|^p
	\Big)\bigg]=0
	\end{align*}
for all $p\in[1,\infty)$. 
\end{theorem}

\begin{remark}[Interpolation]\label{rem:interpol}
By saying that 
$\big((\tilde r^N(t),\tilde v^N(t),\tilde \mu^N(t))\big)_{t\in[0,T]}$
is defined by piecewise constant or linear interpolation of 
	$\big((r^N_k,v^N_k,\mu^N_k)\big)_{k\in\{0,\ldots,N\}}$ we mean that 
either
	$\big(\tilde r^N(t),\tilde v^N(t),\tilde \mu^N(t)\big)
	=
	\big( r^N_k,v^N_k,\mu^N_k\big)$
	for $t\in[kh,(k+1)h]$
	in the  piecewise constant case,
	or
	$
	\big(\tilde r^N(t),\tilde v^N(t),\tilde \mu^N(t)\big)
	=
	\frac{(k+1)h-t}h\big( r^N_k,v^N_k,\mu^N_k\big)+\frac{t-kh}h\big(r^N_{k+1},v^N_{k+1},\mu^N_{k+1}\big)
	$
	for $t\in[kh,(k+1)h]$
	in the piecewise linear case, $k\in\{0,\ldots,N-1\}$. 
\end{remark}


The proof of Theorem~\ref{theo:mainend} is structured as follows: In Section~\ref{subsec:SDEsetting} we reformulate the problem under consideration in a suitable way; the main idea is to formally consider the half-explict drift-truncated Euler scheme as a fully explicit drift-truncated Euler scheme with an additional perturbation term.  Some auxiliary results from the literature concerning the approximation of SDEs with non-globally Lipschitz continuous coefficients are recalled in Section~\ref{subsec:auxiliary}.
In Section~\ref{subsec:convergence_rv} we 
use the results from Section~\ref{sec:approximation} in order to 
verify consistency,  semi-stability and moment growth conditions which imply the strong convergence
$\lim_{N\to\infty}\sup_{t\in[0,T]}\bE\big(
	\big\|r(t)-\tilde r^N(t)\big\|^p
	+\big\|v(t)-\tilde v^N(t)\big\|^p\big)=0$.
Pathwise uniform convergence of $\tilde r^N$, $\tilde v^N$ as well as $\tilde \mu^N$ is finally established in Section~\ref{subsec:convergence_mu}.


\subsection{Reformulation of the problem}\label{subsec:SDEsetting}

Here we reformulate the inherent SDE \eqref{eq:inherentSDE} and the scheme \eqref{eq:approximation} in a way that simplifies the proof of Theorem~\ref{theo:mainend}.
In the sequel, we always assume that the assumptions in Section~\ref{subsec:assumptions} are fulfilled.

While \eqref{eq:inherentSDE} 
is a system in $\bR^n\times\bR^n$ 
with coefficient functions defined on a neighborhood of the tangent bundle $T\cM$, it will be convenient to rewrite it as an $\bR^{2n}$-valued SDE with suitably extended coefficient functions of the form
\begin{align}\label{eq:SDEstrongsetting}
	X(t)=X_0+\int_0^t \drift\big(X(s)\big) \dl s+\int_0^t \diff\big(X(s)\big) \dl w(s),
\end{align}
where $X(t)=(r(t),v(t))^\top$, $X_0=(r_0,v_0)^\top$, and $\drift\colon \mathbb R^{2n}\to \mathbb R^{2n}$, $\diff\colon \mathbb R^{2n}\to \mathbb R^{2n\times \ell}$ are given in Lemma~\ref{lem:extended_coefficients} below.
The following notation is useful in this context.

\begin{notation}
A generic element in $\bR^{2n}$ is denoted by $x=(x^{(1)},x^{(2)})^\top$, where $x^{(1)}\in\bR^n$ and $x^{(2)}\in\bR^n$ are the vectors consisting of the first and last $n$ components of $x$, respectively.  We identify the spaces $\bR^n\times\bR^n$ and $\bR^{2n}$ and consider $T\cM$, $(T\cM)^{\varepsilon}$ etc.\ as subsets of $\bR^{2n}$.
By $\langle \cdot,\cdot\rangle_*$ we denote the inner product on $\bR^{2n}$ defined for $x=(x^{(1)},x^{(2)})^\top$ and $y=(y^{(1)},y^{(2)})^\top\in\bR^{2n}$ by
\begin{align}\label{eq:innerproduct*}
\left\langle
\begin{pmatrix}
x^{(1)}\\x^{(2)}
\end{pmatrix},
\begin{pmatrix}
y^{(1)}\\y^{(2)}
\end{pmatrix}
\right\rangle_{*}
&:=
\left\langle
\begin{pmatrix}
x^{(1)}\\M^{1/2}x^{(2)}
\end{pmatrix},
\begin{pmatrix}
y^{(1)}\\M^{1/2}y^{(2)}
\end{pmatrix}
\right\rangle,
\end{align}
i.e., $\langle x,y\rangle_*=\big\langle x^{(1)},y^{(1)}\big\rangle+\big\langle M^{1/2}x^{(2)},M^{1/2}y^{(2)}\big\rangle$.
The corresponding norm on $\bR^{2n}$ is denoted by $\|\,\cdot\,\|_*$. For $\diff\in\bR^{2n\times\ell}$ we denote by $\|\diff\|_*$ the Hilbert-Schmidt norm of $\diff$ considered as an operator from $(\bR^\ell,\langle\cdot,\cdot\rangle)$ to $(\bR^{2n},\langle\cdot,\cdot\rangle_*)$.
\end{notation}

The inner product \eqref{eq:innerproduct*} is such that the drift coefficient of the inherent SDE \eqref{eq:inherentSDE} satisfies a one-sided linear growth condition w.r.t.\ $\langle\cdot,\cdot\rangle_*$ on the tangent bundle $T\cM$.
Note that the coefficient functions in \eqref{eq:inherentSDE} are defined only on a neighborhood of $T\cM$. 
The next lemma provides suitable extensions to the whole space fulfilling (one-sided) linear growth and polynomial growth conditions.

\begin{lemma}[Extended coefficient functions]\label{lem:extended_coefficients}
There exists an open neighborhood $\cO\subset(T\cM)^{\varepsilon/2}$
of $T\cM$ and locally Lipschitz continuous mappings $\drift\colon\bR^{2n}\to\bR^{2n}$, $\diff\colon\bR^{2n}\to\bR^{2n\times\ell}$ such that the following holds:
For all $x=(x^{(1)},x^{(2)})^\top\in\cO$ we have
\begin{align}\label{eq:defbabB}
\begin{split}
\drift(x)
&=
\begin{pmatrix}
x^{(2)}\\ 
M^{-1}\Big[P_M\big(x^{(1)}\big)a\big(x^{(1)},x^{(2)}\big)-\nabla g\big(x^{(1)}\big)G_M^{-1}\big(x^{(1)}\big)D^2g\big(x^{(1)}\big)\big(x^{(2)},x^{(2)}\big)\Big]
\end{pmatrix}\\
\diff(x)
&=
\begin{pmatrix}
0\\
M^{-1}P_M\big(x^{(1)}\big)B\big(x^{(1)},x^{(2)}\big)
\end{pmatrix},
\end{split}
\end{align}
$\drift$ and $\diff$ are zero on $\bR^{2n}\setminus (T\cM)^\varepsilon$,
and there exists a constant $C\in(0,\infty)$ such that, for all $x\in\bR^{2n}$, 
\begin{align}\label{eq:assbabB}
\begin{split}
\max\big(\langle x,\drift(x)\rangle_*,\|\diff(x)\|_*^2\big)
&\leq 
C\big(1+\|x\|_*^2\big),
\\
\|\drift(x)\|_*
&\leq 
C\big(1+\|x\|_*^{\max(2,p_a)}\big).
\end{split}
\end{align}
\end{lemma}

\begin{proof}
Let $\varepsilon>0$ be as in Assumption~\ref{ass:g}. 
For  all $x=(x^{(1)},x^{(2)})^\top\in(T\cM)^{\varepsilon}$, let $\tilde\drift(x)\in\bR^{2n}$ and $\tilde\diff(x)\in\bR^{2n\times\ell}$ be defined by the right hand side of \eqref{eq:defbabB}.
As a consequence of the growth conditions in Assumption~\ref{ass:aB}, there exists a constant $C\in(0,\infty)$ such that for all $x\in T\cM$ it holds that
\begin{align}\label{eq:est_tildeAB}
\begin{split}
\max\big(\langle x,\tilde\drift(x)\rangle_*,\|\tilde\diff(x)\|_*^2\big)
&<
\frac C2\big(1+\|x\|_*^2\big),\quad
\|\tilde\drift(x)\|_*
<
\frac C2\big(1+\|x\|_*^{\max(2,p_a)}\big).
\end{split}
\end{align}
Indeed, considering for instance $\langle x,\tilde\drift(x)\rangle_*$
we have, for $x=(x^{(1)},x^{(2)})^\top\in T\cM$,
\begin{align*}
\big\langle x,\,\tilde\drift(x)\big\rangle_*
&=
\big\langle x^{(1)},\,x^{(2)}\big\rangle+
\big\langle x^{(2)},\,P_M\big(x^{(1)}\big)a\big(x^{(1)},x^{(2)}\big)\big\rangle\\
&\quad
-\big\langle x^{(2)},\,\nabla g\big(x^{(1)}\big)G_M^{-1}\big(x^{(1)}\big)D^2g\big(x^{(1)}\big)\big(x^{(2)},x^{(2)}\big)\big\rangle\\
&=
\big\langle x^{(1)},\,x^{(2)}\big\rangle+\big\langle x^{(2)},a\big(x^{(1)},x^{(2)}\big)\big\rangle,
\end{align*}
where we use the fact that, for all $z\in\bR^{n}$,
\begin{align*}
	\begin{split}
	\big\langle x^{(2)},P_M\big(x^{(1)}\big)z\big\rangle
	&=
	\big\langle M^{-1/2} Mx^{(2)},M^{-1/2}P_M\big(x^{(1)}\big)z\big\rangle
	=
	\big\langle M^{-1/2} Mx^{(2)},M^{-1/2}z\big\rangle,
	\\
	\big\langle x^{(2)},\nabla g\big(x^{(1)}\big)z\big\rangle
	&=
	\big\langle Dg\big(x^{(1)}\big)x^{(2)},z\big\rangle
	=
	\langle 0,z\rangle=0.
	\end{split}
\end{align*}
Let $\cO\subset\bR^{2n}$ be the open set consisting of all $x\in(T\cM)^{\varepsilon/2}$ such that the estimate \eqref{eq:est_tildeAB} is fulfilled, and let $\mathcal R\subset\bR^{2n}$ be defined by
\begin{align*}
\mathcal R&=\Bigg\{x\in(T\mathcal M)^{\varepsilon}:\begin{array}{c}\max\big(\langle x,\drift(x)\rangle_*,\|\diff(x)\|_*^2\big)\geq C(1+\|x\|_*^2),\\ \|\drift(x)\|_*\geq C\big(1+\|x\|_*^{\max(2,p_a)}\big)\end{array}\Bigg\}\cup \big((T\mathcal M)^{\varepsilon}\big)^c.
\end{align*}
Observe that $T\cM\subset \cO\subset\mathcal R^c \subset(T\cM)^\varepsilon$ and that $\dist(x,\cO)+\dist(x,\mathcal R)>0$ for all $x\in\bR^{2n}$. 
Thus, we can define $\drift\colon \bR^{2n} \to \bR^{2n}$ and $\diff\colon \bR^{2n} \to \bR^{2n\times \ell}$ by setting
\begin{align}\label{eq:defAB}
\drift(x):=\frac{\operatorname{dist}(x,\mathcal R)}{\operatorname{dist}(x,\mathcal R)+\operatorname{dist}(x,\cO)}\,\tilde\drift(x),\quad\diff(x):=\frac{\operatorname{dist}(x,\mathcal R)}{\operatorname{dist}(x,\mathcal R)+\operatorname{dist}(x,\cO)}\,\tilde\diff(x)
\end{align}
for $x\in\bR^{2n}$, where we consider arbitrary extensions of $\tilde\drift$ and $\tilde\diff$ to the whole space $\bR^{2n}$.
Note that the globalization factor $\dist(x,\mathcal R)/(\dist(x,\cO)+\dist(x,\mathcal R))$ is less than or equal to one and vanishes on $\mathcal R$, so that the estimate \eqref{eq:assbabB} is fulfilled for all $x\in\bR^{2n}$ due to the construction of the sets $\cO$ and $\mathcal R$.
Further, since the mappings $x\mapsto\dist(x,\cO)$ and $x\mapsto\dist(x,\mathcal R)$ are Lipschitz continuous and since for all $R\in(0,\infty)$ there exists some $C_R\in(0,\infty)$ such that 
$
\inf_{x\in B_R(0)}(\dist(x,\cO)+\dist(x,\mathcal R))\geq C_R,
$
we have that $x\mapsto\dist(x,\mathcal R)/(\dist(x,\cO)+\dist(x,\mathcal R))$ is locally Lipschitz continuous. As a consequence, 
$\drift$ and $\diff$ are locally Lipschitz continuous as well.
\end{proof}

In what follows we always consider the extended coefficient functions $\drift$ and $\diff$ constructed in Lemma~\ref{lem:extended_coefficients}. Under the assumptions in Section~\ref{subsec:assumptions}, we can argue as in Theorem~\ref{theo:solvability_app} to obtain that the SDE~\eqref{eq:SDEstrongsetting} has a unique strong solution $X=(X(t))_{t\geq0}$ with initial condition $X_0=(r_0,v_0)^\top$. It is also clear that this solution coincides with the solution to the inherent SDE~\eqref{eq:inherentSDE} with initial conditions $r_0,v_0$ in the sense that, $\bP$-almost surely, $X(t)=(r(t),v(t))^\top$ for all $t\geq0$.

Next, we reformulate the half-explicit drift-truncated Euler scheme~\eqref{eq:approximation} as an explicit one-step scheme of the form
\begin{align}\label{eq:appstrongdef}
X_{k+1}^N=X_k^N+\phi\big(X_k^N,\,h,\,\Delta_{h}w_k\big),
\end{align}
$k=0,\ldots,N-1$, where $h=T/N$ and $\Delta_hw_k=w((k+1)h)-w(kh)$ are as in Section~\ref{subsec:scheme}, $X_0^N:=X_0=(r_0,v_0)^\top$, and where
$\phi:\mathbb R^{2n}\times [0,T]\times \mathbb R^{\ell}\to\mathbb R^{2n}$ is an increment function of the type
\begin{align}\label{eq:algtype}
\phi(x,h,w)&:=\eta_0(x,h,w)+\eta_1(x,h,w)\drift(x)h+\eta_2(x,h,w)\diff(x)w
\end{align}
with Borel-measurable functions  $\eta_0\colon\mathbb R^{2n}\times[0,T]\times \mathbb R^{\ell}\to\mathbb R^{2n}$, $\eta_1\colon\mathbb R^{2n}\times [0,T]\times\mathbb R^{\ell}\to\mathbb R$ and $\eta_2\colon\mathbb R^{2n}\times[0,T]\times \mathbb R^{\ell}\to\mathbb R$. 


\begin{lemma}[Reformulation of the scheme]
For $N\in\mathbb N$ and $h:=T/N$ let
$\big((r^N_k,v^N_k,\kappa^N_k,\lambda^N_k)\big)_{k\in\{0,\ldots,N\}}$ 
with  $\kappa_0^N:=\lambda_0^N:=0\in\bR^m$
be the solution to the half-explicit drift-truncated Euler scheme \eqref{eq:approximation} according to Corollary~\ref{cor:applyingonestep}.
Let $\drift\colon\bR^{2n}\to\bR^{2n}$, $\diff\colon\bR^{2n}\to\bR^{2n\times\ell}$ be the mappings introduced in Lemma~\ref{lem:extended_coefficients} and let $(X^N_k)_{k\in\{0,\ldots,N\}}$ be the solution to the scheme \eqref{eq:appstrongdef} with initial condition $X^N_0=(r_0,v_0)^\top$ and with increment function $\phi\colon\mathbb R^{2n}\times[0,T]\times \mathbb R^{\ell}\to\mathbb R^{2n}$ defined by \eqref{eq:algtype} and 
\begin{align}\label{eq:schemeinterpret}
\begin{split}
\eta_0(x,h,w)
&:=
\mathbbm 1_{\cD}(x,h)
\begin{pmatrix}
M^{-1}\nabla g\big(x^{(1)}\big)\hat \kappa\big(x^{(1)},x^{(2)},h\big)\\
M^{-1}\nabla g\big(x^{(1)}\big)\Big[\lambdarem\big(x^{(1)},x^{(2)},h\big)+\Lambdarem\big(x^{(1)},x^{(2)},h\big)w\Big]
\end{pmatrix},
\\
\eta_1(x,h,w)&:=\eta\big(x^{(1)},x^{(2)},h\big),\\
\eta_2(x,h,w)&:=1,
\end{split}
\end{align}
where 
the open neighborhood 
$\cD\subset\bR^{2n}\times(0,\infty)$ 
of $T\cM\times(0,\infty)$ 
and 
the mappings 
$\hat\kappa\colon\cD\to\bR^m$, $\lambdarem\colon\cD\to\bR^m$, $\Lambdarem\colon\cD\to\bR^{m\times\ell}$ are given by Theorem~\ref{theo:mainonestep}
and $\eta$ is the truncation function given by \eqref{eq:defeta}.
Then we have
\begin{align*}
X^N_k
=
\big(r^N_k,v^N_k\big)^\top
\;\text{ for all } k\in\{0,\ldots,N\}.
\end{align*}
\end{lemma}

\begin{proof}
The assertion is a direct consequence of the decomposition \eqref{eq:nonlineardecompose} of the Lagrange multiplier $\hat\lambda$ in Theorem~\ref{theo:mainonestep}, respectively the decomposition \eqref{eq:nonlineardecompose_stoch} of $\lambda_{k+1}$ in Corollary~\ref{cor:applyingonestep}, and the identity $P_M(x^{(1)})=\Id-\nabla g(x^{(1)})G_M^{-1}(x^{(1)})Dg(x^{(1)})M^{-1}\in\bR^{n\times n}$ for $x^{(1)}\in\bR^n$.
\end{proof}

\begin{notation}
In the sequel, we denote by $\phi_\cM$ the increment function corresponding to the half-explicit drift-truncated Euler scheme \eqref{eq:approximation}, defined by \eqref{eq:algtype} and \eqref{eq:schemeinterpret}.
We also write $\eta(x,h)$ instead of $\eta(x^{(1)},x^{(2)},h)$ for all $x=(x^{(1)},x^{(2)})^\top\in\mathbb R^{2n}$, $h\in(0,\infty)$, and the truncation function $\eta$ in \eqref{eq:defeta}.
Moreover, we denote by $\phi_{\mathrm{EM}}$ the increment function corresponding to the classical Euler-Maruyama scheme, which is also  of the form \eqref{eq:algtype} with
\begin{align}\label{eq:schemeEM}
\eta_0(x,h,w)=0,\quad \eta_1(x,h,w)=1,\quad\eta_2(x,h,w)=1.
\end{align}
\end{notation}

\subsection{Auxiliary results}\label{subsec:auxiliary}
Here we recall some  results concerning the approximation of SDEs with non-globally Lipschitz continuous coefficients, taken from \cite{hutzenthaler:p:2015}. 
Throughout this subsection, we suppose that the assumptions in Subsection~\ref{subsec:assumptions} are fulfilled and consider the setting described in Subsection~\ref{subsec:SDEsetting}. 
In particular, $\drift$ and $\diff$ are the coefficient functions defined in Lemma~\ref{lem:extended_coefficients}.

We begin by recalling the concepts of consistency and semi-stability w.r.t\ Brownian motion. 
The function 
$\Phi$
appearing below will later be chosen as $\Phi(x,h,w)=x+\phi(x,h,w)$, where $\phi$ is an increment function as in \eqref{eq:algtype}.


\begin{definition}[Consistency, semi-stability]\label{def:consisstable} 
Let $\phi$ and $\Phi$ be Borel-measurable mappings from $\bR^{2n}\times [0,T]\times \bR^\ell$ to $\bR^{2n}$.
\begin{enumerate}[leftmargin=7mm]
		\item[(i)] We say that $\phi$ is $(\drift,\diff)$-consistent with respect to Brownian motion if
		\begin{align*}
		&\limsup_{t\searrow 0}\left(\frac{1}{\sqrt{t}}\sup_{x\in K}\bE\Big[\big\|\diff(x)w(t)-\phi(x,t,w(t))\big\|\Big]\right)=0,\\
		&\limsup_{t\searrow 0}\left(\frac 1t\sup_{x\in K}\Big\|\drift(x)\,t-\bE\big[\phi(x,t,w(t))\big]\Big\|\right)=0
		\end{align*}
		for all non-empty compact sets $K\subset \mathbb R^{2n}$.
		\item[(ii)] Let $\alpha\in (0,\infty]$ and $V\colon\mathbb R^{2n}\to [0,\infty)$ be Borel-measurable. We say that $\Phi$ is $\alpha$-semi-$V$-stable w.r.t.\ Brownian motion if there exists $\rho\in\mathbb R$ such that
		\begin{align*}
		\mathbb E\big[V\big(\Phi(x,t,w(t))\big)\big]\leq e^{\rho t} V(x)
		\end{align*}
		for all $(x,t)\in\mathbb R^{2n}\times [0,T]$ with $V(x)\leq t^{-\alpha}$ if $\alpha<\infty$ 
		and for all $(x,t)\in\mathbb R^{2n}\times [0,T]$ if $\alpha=\infty$.
	\end{enumerate}
\end{definition}


We now state the main auxiliary result we use to prove strong convergence of the half-explicit drift-truncated Euler scheme. It is a special case of \cite[Corollary~3.14]{hutzenthaler:p:2015}.

\begin{proposition}\label{cor:mainjentzen}
	Let $(X(t))_{t\geq0}$ be the strong solution to the SDE~\eqref{eq:SDEstrongsetting} with initial condition $X_0=(r_0,v_0)^\top$. Let $p\in[1,\infty)$, $\varrho\in[p,\infty)\cap(1,\infty)$ and $\alpha\in(1,\infty]$. 
	Assume that $\phi\colon\mathbb R^{2n}\times[0,T]\times \mathbb R^\ell\to \mathbb R^{2n}$ is $(\drift,\diff)$-consistent 
	w.r.t.\ Brownian motion 
	and that $\Phi\colon\mathbb R^{2n}\times[0,T]\times \mathbb R^\ell\to \mathbb R^{2n}\colon(x,h,w)\mapsto x+\phi(x,h,w)$ is $\alpha$-semi-$(1+\|\cdot\|_*^\varrho )$-stable w.r.t.\ Brownian motion. For $N\in\bN$, let $(X^N_k)_{k\in\{0,\ldots,N\}}$ be the solution to the scheme \eqref{eq:appstrongdef} with initial condition $X^N_0=(r_0,v_0)^\top$ and assume further that
\begin{align}\label{eq:momentcor314}
	\limsup_{N\to\infty}\Bigl(N^{(1-\alpha)(1-1/\varrho)}\sup_{k\in\{0,\dots,N\}} \big(\mathbb E(\|X_{k}^N\|^{p\varrho})\big)^{1/\varrho}\Bigr)<\infty.
	\end{align}
	Then $\limsup_{N\to\infty}\sup_{t\in[0,T]}\mathbb E(\|\tilde X^N(t)\|^p)<\infty$ and, for all $q\in(0,p)$,
	\begin{align}\label{eq:strongconvcor314}
	\lim_{N\to\infty}\sup_{t\in[0,T]} \mathbb E\big(\|X(t)-\tilde X^N(t)\|^{q}\big)=0,
	\end{align}
	 where $(\tilde X^N(t))_{t\in[0,T]}$ is defined by piecewise linear interpolation of $(X^N_k)_{k\in\{0,\ldots,N\}}$.
\end{proposition}

Suitable conditions for consistency and semi-stability w.r.t.\ Brownian motion can be found in \cite[Lemma 3.24]{hutzenthaler:p:2015} and \cite[Lemma 2.18]{hutzenthaler:p:2015}. Below we state special cases of these results in the context of our setting.
The first lemma concerns the $(\drift,\diff)$-consistency w.r.t.\ Brownian motion of schemes of the type \eqref{eq:algtype}. 
The second lemma is a comparison principle leading to semi-stability.

\begin{lemma}\label{lem:jentzen324}
	Let $\eta_0\colon\mathbb R^{2n}\times[0,T]\times \mathbb R^{\ell}\to\mathbb R^{2n}$, $\eta_1\colon\mathbb R^{2n}\times[0,T]\times \mathbb R^{\ell}\to\mathbb R$ and $\eta_2\colon\mathbb R^{2n}\times[0,T]\times \mathbb R^{\ell}\to\mathbb R$ be Borel-measurable functions such that 
	\begin{align}\label{eq:estilem324}
	\begin{gathered}
	\limsup_{t\searrow 0}\Big(\frac{1}{\sqrt{t}}\sup_{x\in K}\mathbb E\big[\|\eta_0(x,t,w(t))\|\big]\Big)=
	\limsup_{t\searrow 0}\Big(\frac{1}{t}\sup_{x\in K}\big\|\mathbb E\big[\eta_0(x,t,w(t))\big]\big\|\Big)=0,\\	
	\limsup_{t\searrow 0}\Big(\sup_{x\in K}\mathbb E\Big[\big|\eta_1(x,t,w(t))-1\big|+\big|\eta_2(x,t,w(t))-1\big|^2\Big]\Big)=0,\\
	\limsup_{t\searrow 0}\Big(\frac{1}{t}\sup_{x\in K}\big\|\mathbb E\big[\eta_2(x,t,w(t))\diff(x)w(t)\big]\big\|\Big)=0,\\
	\limsup_{t\searrow 0}\Big(\sup_{x\in K}\mathbb E\big[\|\eta_2(x,t,w(t))\diff(x)w(t)\|\big]\Big)<\infty
	\end{gathered}
	\end{align}
	for all non-empty compact sets $K\subset \bR^{2n}$. Then $\phi$ defined by \eqref{eq:algtype} is $(\drift,\diff)$-consistent with respect to Brownian motion.
\end{lemma}

\begin{lemma}\label{lem:jentzencomparisonprinciple}
	Let $\alpha\in(0,\infty)$, $\varrho\in[1,\infty)$ and $\Phi,\tilde\Phi\colon\mathbb R^{2n}\times[0,T]\times \mathbb R^\ell\to \mathbb R^{2n}$ be Borel-measurable functions such that $\tilde\Phi$ is $\alpha$-semi-$(1+\|\cdot\|_*^\varrho)$-stable w.r.t.\ Brownian motion. If there exist $C\in[0,\infty)$ such that
	\begin{align*}
	\Bigl(\mathbb E \big(\bigl\|\Phi(x,t,w(t))-\tilde \Phi(x,t,w(t))\bigr\|^\varrho\big)\Bigr)^{1/\varrho}\leq C\cdot t\cdot\big(1+\|x\|_*^\varrho\big)^{1/\varrho}
	\end{align*}
	for all $(x,t)\in\bR^{2n}\times(0,T]$ with $(1+\|x\|_*^\varrho)\leq t^{-\alpha}$, then $\Phi$ is also $\alpha$-semi-$(1+\|\cdot\|_*^\varrho)$-stable w.r.t.\ Brownian motion.
\end{lemma}

\subsection{Convergence of the position and velocity processes}\label{subsec:convergence_rv}
In this subsection we verify the strong convergence of $\tilde X^N=(\tilde r^N,\tilde v^N)^\top$ towards $X=(r,v)^\top$ in the sense that $\lim_{N\to\infty}\sup_{t\in[0,T]} \mathbb E\big(\|X(t)-\tilde X^N(t)\|^{p}\big)=0$ for all $p\in[1,\infty)$.
To this end, we show that the half-explicit drift-truncated Euler scheme \eqref{eq:approximation}, reformulated in Section~\ref{subsec:SDEsetting}, satisfies the abstract requirements in the auxiliary results from Section~\ref{subsec:auxiliary}.
In the sequel we always suppose that the assumptions in Section~\ref{subsec:assumptions} are fulfilled and consider the setting described in Section~\ref{subsec:SDEsetting}. 


\begin{lemma}\label{lem:varifyconsistent}
The increment function $\phi=\phi_{\cM}$ of the half-explicit drift-truncated Euler scheme, given by \eqref{eq:algtype} and \eqref{eq:schemeinterpret}, is $(\drift,\diff)$-consistent w.r.t.\ Brownian motion. 
\end{lemma}

\begin{proof}
We fix a non-empty compact set $K\subset\bR^{2n}$ and verify the assumptions in Lemma~\ref{lem:jentzen324}.
As a consequence of the definition of $\eta_0$ in \eqref{eq:schemeinterpret}, Assumption~\ref{ass:g}, and the estimates \eqref{eq:estLambdarem} of $\hat\kappa$, $\lambdarem$, $\Lambdarem$ in Theorem~\ref{theo:mainonestep}, there exists a constant $C\in(0,\infty)$ such that
\begin{align*}
&\limsup_{t\searrow 0}\bigg(\frac{1}{t}\sup_{x\in K}\bE\big[\|\eta_0(x,t,w(t))\|\big]\bigg)\\ &\leq C\limsup_{t\searrow 0}\bigg(\frac{1}{t}\sup_{x\in K}\bE\Big[\big(1+\|x\|^{\max(p_\lambda,2)}\big)t^2+(1+\|x\|^{p_\Lambda})\|w(t)\|t\Big]\bigg)=0,
\end{align*} 
which implies both assertions concerning $\eta_0$ in \eqref{eq:estilem324}. 
Since $\eta_2\equiv 1$ and $\diff$ is bounded on compact sets, the assertions concerning $\eta_2$ in \eqref{eq:estilem324} are obviously fulfilled. 
As $\eta_1(x,h,w)=\eta(x,h)$ in \eqref{eq:schemeinterpret} does not depend on $w\in\bR^\ell$, it is left to show that
$
\limsup_{t\searrow 0}\sup_{x\in K}|\eta(x,t)-1| =0.
$
For $x=(x^{(1)},x^{(2)})^\top\in\bR^{2n}$ let us set $f(x):=\max(\|x^{(2)}\|,\|x^{(2)}\|^2,\|a(x^{(1)},x^{(2)})\|)$, so that we can rewrite the truncation function $\eta$ defined in \eqref{eq:defeta} in the form
$\eta(x,t)=\min\big(1,C_\eta/(f(x)t)\big)$, $(x,t)\in\bR^{2n}\times(0,\infty)$.
Using the boundedness of $\eta$, the fact that $s\mapsto -\min(1,\Ceta/s)$ is continuous and non-decreasing on $(0,\infty)$, and the finitenesss of $\sup_{x\in K}f(x)$, we obtain
\begin{align*}
\limsup_{t\searrow 0}\sup_{x\in K}|\eta(x,t)-1| 
&=
\limsup_{t\searrow 0}\sup_{x\in K}(1-\eta(x,t))\\
&=1-\min\Big(1,\frac{\Ceta}{\limsup_{t\searrow 0}\sup_{x\in K}(
f(x)t)}\Big)=1-1=0.
\end{align*} 
This finishes the proof of Lemma~\ref{lem:varifyconsistent}.
\end{proof}

In the next lemma we verify the $\alpha$-semi-$V$-stabilty of the half-explicit drift-truncated Euler scheme for the Lyapunov-type function $V=(1+\|\cdot\|_*^\varrho)$ with $\varrho\in[3,\infty)$ and a certain range of $\alpha$.
To this end, we use the comparison principle from Lemma~\ref{lem:jentzencomparisonprinciple} and the fact that the standard Euler-Maruyama scheme for the inherent SDE in the form \eqref{eq:SDEstrongsetting}, i.e.,
\begin{align*}
\Phi_{\mathrm{EM}}(x,h,w)=x+\phi_{\mathrm{EM}}(x,h,w)=x+\drift(x)h+\diff(x)w,
\end{align*}
is $\alpha$-semi-V-stable w.r.t.~Brownian motion for  $V=(1+\|\cdot\|_*^\varrho)$ with $\varrho\in[3,\infty)$ and all 
$\alpha\in (0,\varrho/(2\max(1,p_a-1))]$.
The last assertion follows from the growth properties \eqref{eq:assbabB} of the coefficients $\drift$ and $\diff$ and a straightforward application of \cite[Theorem~2.13]{hutzenthaler:p:2015}, compare \cite[Corollary~2.16]{hutzenthaler:p:2015}. 
\begin{lemma}\label{lem:verifystable}
	Fix $\varrho\in[3,\infty)$, let $\phi_{\cM}\colon\mathbb R^{2n}\times [0,T] \times \mathbb R^\ell\to \mathbb R^{2n}$, defined by \eqref{eq:algtype} and \eqref{eq:schemeinterpret}, be the increment function of the half-explicit drift truncated Euler scheme and let $\phi_{\mathrm{EM}}\colon\mathbb R^{2n}\times [0,T] \times \mathbb R^\ell\to \mathbb R^{2n}$, defined by \eqref{eq:algtype} and \eqref{eq:schemeEM}, be the increment function of the Euler Maruyama scheme. Then, for all $\alpha\in(0,\varrho/(2p_a+1)]$ there exists $C\in[0,\infty)$ such that the estimate
	\begin{align}\label{eq:diffEMHEDT}
	\Bigl(\mathbb E\big(\bigl\|\phi_{\cM}(x,t,w(t))-\phi_{\mathrm{EM}}(x,t,w(t))\bigr\|^{\varrho}\big)\Bigr)^{1/\varrho}\leq C\cdot t\cdot\big(1+\|x\|_*^{\varrho}\big)^{1/\varrho}
	\end{align}
	holds for all $(x,t)\in\bR^{2n}\times (0,T]$ with $(1+\|x\|_*^{\varrho})\leq t^{-\alpha}$. In particular, $\Phi_{\cM}$ defined by $\Phi_{\cM}(x,h,w)=x+\phi_{\cM}(x,h,w)$ is $\alpha$-semi-$(1+\|\cdot\|_*^{\varrho})$-stable w.r.t.\ Brownian motion for all $\alpha\in(0,\varrho/(2p_a+1)]$.
\end{lemma}

\begin{proof}
By the construction of $\drift$, $\diff$ in Lemma~\ref{lem:extended_coefficients} and the definition of $\phi_{\mathrm{EM}}$, $\phi_\cM$, we have $\phi_{\mathrm{EM}}(x,t,w(t))=\phi_{\cM}(x,t,w(t))=0$ for all $x\in\bR^{2n}\setminus(T\cM)^\varepsilon$, $t\in[0,T]$. 
Therefore, consider $x\in (T\cM)^\varepsilon$, $t\in[0,T]$. 
Throughout this proof, let $\eta_0$, $\eta_1$ and $\eta_2$ be the functions corresponding to the half-explicit drift truncated Euler scheme as defined in \eqref{eq:schemeinterpret}.
We also use the notation 
\begin{align*}
k(x):=M^{-1}\nabla g(x^{(1)})G_M^{-1}(x^{(1)})D^2g(x^{(1)})(x^{(2)},x^{(2)}),
\end{align*}
$x=(x^{(1)},x^{(2)})^\top\in\bR^{2n}$, and denote by $C$ a finite constant depending only on $p$, $T$, $C_a$, $C_B$, $C_g$, $M$, $\alpha$ that may change its value with every new appearance. Since the globalization factor in the definition \eqref{eq:defAB} of $\drift$ is bounded by one, we have
\begin{equation}\label{eq:verifystable1}
\begin{aligned}
&\bigl\|\phi_{\cM}(x,t,w(t))-\phi_{\mathrm{EM}}(x,t,w(t))\bigr\|_{L_\varrho}\\
&\leq \|\eta_0(x,t,w(t))\|_{L_\varrho}+\big\|(\eta(x,t)-1)\drift(x)\big\|t+0\\
&\leq C\Big(\big\|\eta_0^{(1)}(x,t,w(t))\big\|_{L_\varrho}+\big\|\eta_0^{(2)}(x,t,w(t))\big\|_{L_\varrho}\\
&\quad+|\eta(x,t)-1|\big\{\|x^{(2)}\|+\big\|a(x^{(1)},x^{(2)})\|+\|k(x)\|\big\}t\Big).
\end{aligned}
\end{equation}
Using the estimate \eqref{eq:estLambdarem} for $\|\hat\kappa\|$, $\|\hat\lambda\|$, $\|\hat\Lambda\|$, we obtain
\begin{equation}\label{eq:verifystable2}
\begin{aligned}
\big\|\eta_0^{(1)}(x,t,w(t))\big\|_{L_\varrho}&\leq C \big(1+\|x\|^2\big) t^2,\\
\big\|\eta_{0}^{(2)}(x,t,w(t))\big\|_{L_\varrho}&\leq C \big(1+\|x\|^{p_\lambda}\big)t^2+ C\big(1+\|x\|^{p_\Lambda}\big) t^{3/2} \|w(1)\|_{L_\varrho}.
\end{aligned}
\end{equation}
Moreover, the definition \eqref{eq:defeta} of $\eta$, the polynomial growth property of $a$ in Assumption~\ref{ass:aB}, and the fact that  $\eta(x,t)\in (0,1]$ imply
\begin{align*}
|\eta(x,t)-1|&=|1-\eta(x,t)^{-1}|\cdot\eta(x,t)\\
&\leq \eta(x,t)^{-1}-1\\
&=\max\Big(0,\,C_\eta^{-1}\max\big(\|x^{(1)}\|,\|x^{(1)}\|^2,\|a(x^{(1)},x^{(2)})\|\big)t-1\Big)\\
&\leq C_\eta^{-1}\big(\|x^{(1)}\|+\|x^{(1)}\|^2+\|a(x^{(1)},x^{(2)})\|\big)t\\
&\leq C(1+\|x\|^2+\|x\|^{p_a})t
\end{align*}
and further, using also the estimates for the norms of $\nabla g(x^{(1)})$, $G_M^{-1}(x^{(1)})$ and $D^2g(x^{(1)})$ following from Assumption~\ref{ass:g},
\begin{equation}\label{eq:verifystable3}
\begin{aligned}
|\eta(x,t)-1| \|x^{(2)}\| t
&\leq C(1+\|x\|^3+\|x\|^{p_a+1})t^2,\\
|\eta(x,t)-1| \|a(x^{(1)},x^{(2)})\| t
&\leq C(1+\|x\|^{p_a+2}+\|x\|^{2p_a})t^2,\\
|\eta(x,t)-1| \|k(x)\| t
&\leq C(1+\|x\|^{4}+\|x\|^{p_a+2})t^2.
\end{aligned}
\end{equation}
As a consequence of the estimate \eqref{eq:addestilambda} and the growth properties of $a$ and $B$ in Assumption \ref{ass:aB}, the exponents $p_\lambda$ and $p_\Lambda$ in the estimate \eqref{eq:estLambdarem} of $\|\hat\kappa\|$, $\|\hat\lambda\|$, $\|\hat\Lambda\|$ can be chosen as $p_\lambda=\max(p_a+1,3)$ and $p_\Lambda=2$. 
Thus, combining \eqref{eq:verifystable1}, \eqref{eq:verifystable2}, \eqref{eq:verifystable3} and using the equivalence of the norms $\|\cdot\|$ and $\|\cdot\|_*$ on $\bR^{2n}$, we have 
\begin{equation}\label{eq:verifystable4}
\begin{aligned}
\bigl\|\phi_{\cM}(&x,t,w(t))-\phi_{\mathrm{EM}}(x,t,w(t))\bigr\|_{L_\varrho}\\
&\leq C t (1+\|x\|_*^{\varrho})^{1/\varrho}\Big[t(1+\|x\|_*^{2p_a+1})+t^{1/2}(1+\|x\|_*)\Big].
\end{aligned}
\end{equation}
Note that if $\alpha\in(0,\varrho/(2p_a+1)]$ and 
$(x,t)\in\bR^{2n}\times[0,T]$ is such that 
$(1+\|x\|_*^\varrho)\leq t^{-\alpha}$, then 
\begin{equation}\label{eq:verifystable5}
\begin{aligned}
t(1+\|x\|_*^{2p_a+1}&)\leq C \big(t^{\frac{\varrho}{2p_a+1}}(1+\|x\|_*^\varrho)\big)^{\frac{2p_a+1}{\varrho}}\leq C \big(t^{\frac{\varrho}{2p_a+1}}t^{-\alpha}\big)\big)^{\frac{2p_a+1}{\varrho}}\leq C,
\\
t^{1/2}(1+\|x\|_*)&\leq C\big(t^{\varrho/2}(1+\|x\|_*^\varrho)\big)^{1/\varrho}\leq C\big(t^{\varrho/2}t^{-\alpha}\big)^{1/\varrho}\leq C.
\end{aligned}
\end{equation}
As remarked above, the Euler-Maruyama scheme is  $\alpha$-semi-$(1+\|\cdot\|_*^\varrho)$-stable w.r.t.\ Brownian motion for all $\alpha\in(0,\varrho/(2\max(1,p_a-1))]$. The estimates \eqref{eq:verifystable4}, \eqref{eq:verifystable5} and Lemma~\ref{lem:jentzencomparisonprinciple} together with the fact that $\varrho/(2p_a+1)<\varrho/(2\max(1,p_a-1))$ thus imply that the half-explicit drift-truncated Euler scheme is $\alpha$-semi-$(1+\|\cdot\|_*^\varrho)$-stable w.r.t.\ Brownian motion for all $\alpha\in(0,\varrho/(2p_a+1)]$.
\end{proof}

Next, we verify the moment growth condition \eqref{eq:momentcor314} in Proposition~\ref{cor:mainjentzen} in the context of our setting.
\begin{lemma}\label{cor:strongrate}
	For $N\in\bN$ let $(X^N	_k)_{k\in\{0,\ldots,N\}}$ be the solution to the half-explicit drift-truncated Euler scheme in the form \eqref{eq:appstrongdef}, with increment function $\phi=\phi_\cM$ given by \eqref{eq:algtype} and \eqref{eq:schemeinterpret}.
	 Let $p\in[1,\infty)$, $\varrho\in [2,\infty)$, $\alpha\in [2p+1,\infty)$.
	Then the moment condition \eqref{eq:momentcor314} is fulfilled, i.e.,
	\begin{align*}
	\limsup_{N\to\infty}\Bigl(N^{(1-\alpha)(1-1/\varrho)}\sup_{k\in\{0,\dots,N\}} \big(\mathbb E(\|X_{k}^N\|^{p\varrho})\big)^{1/\varrho}\Bigr)<\infty.
	\end{align*}
\end{lemma}

\begin{proof}
We rewrite the scheme \eqref{eq:appstrongdef}--\eqref{eq:schemeinterpret} as
\begin{align*}
X_{k+1}^N&=X_k^N+\eta(X_{k}^N,h)\drift(X_{k}^N)h+\diff(X_{k}^N)\Delta_hw_k\\
&\quad+\vartheta_1(X_{k}^N,h)+\vartheta_2(X_{k}^N,h)\Delta_hw_k,
\end{align*}
where $\vartheta_1\colon\bR^{2n}\times [0,T]\to\bR^{2n}$ and $\vartheta_2\colon\bR^{2n}\times [0,T]\to\bR^{2n\times\ell}$ are defined by 
\begin{align*}
\vartheta_1(x,h)&:=\mathbbm 1_{\cD}(x,h)\begin{pmatrix}M^{-1} \nabla g(x^{(1)}) \hat \kappa(x,h)\\M^{-1} \nabla g(x^{(1)}) \lambdarem(x,h)\end{pmatrix},\\
\vartheta_2(x,h)&:=\mathbbm 1_{\cD}(x,h)\begin{pmatrix}0\\ M^{-1} \nabla g(x^{(1)})\Lambdarem(x,h)\end{pmatrix}.
\end{align*}
As before,
the open neighborhood 
$\cD\subset\bR^{2n}\times(0,\infty)$ 
of $T\cM\times(0,\infty)$ 
and 
the mappings 
$\hat\kappa\colon\cD\to\bR^m$, $\lambdarem\colon\cD\to\bR^m$, $\Lambdarem\colon\cD\to\bR^{m\times\ell}$ are given by Theorem~\ref{theo:mainonestep}
and $\eta$ is the truncation function given by \eqref{eq:defeta}.
Observe that
\begin{equation}\label{eq:moment_bound_1}
\begin{aligned}
&\|X_{k+1}^N\|^2\\
&= \big\|X_k^N+\eta(X_{k}^N,h)\drift(X_k^N)h+\vartheta_1(X_k^N,h)\big\|^2+\big\|\big(\diff(X_k^N)+\vartheta_2(X_k^N,h)\big)\Delta_hw_k\big\|^2\\
&\quad+2\Big\langle X_k^N+\drift(X_k^N)\eta(X_{k}^N,h)h+\vartheta_1(X_k^N,h),\,\big(\diff(X_k^N)+\vartheta_2(X_k^N,h)\big)\Delta_hw_k\Big\rangle\\
&\leq \big\|X_k^N+\eta(X_{k}^N,h)\drift(X_k^N)h+\vartheta_1(X_k^N,h)\big\|^2+\big\|\eta(X_{k}^N,h)\drift(X_k^N)h+\vartheta_1(X_k^N,h)\big\|^2\\
&\quad+2\big\|\big(\diff(X_k^N)+\vartheta_2(X_k^N,h)\big)\Delta_hw_k\big\|^2+2\big\langle X_k^N,\,\big(\diff(X_k^N)+\vartheta_2(X_k^N,h)\big)\Delta_hw_k\big\rangle.
\end{aligned}
\end{equation}
In the last step we use the estimate $\langle x+y,z\rangle\leq \langle x,z\rangle+\|y\|\|z\|\leq \langle x,z\rangle+\|y\|^2/2+\|z\|^2/2$ holding for all $x,y,z\in\mathbb R^{2n}$. Next, note that
there exists $C\in[1,\infty)$ such that, for all $(x,h)\in \bR^{2n}\times [0,T]$,
\begin{equation}\label{eq:moment_bound_2}
\begin{aligned}
\|\eta(x,h)\drift(x)h\|&\leq C/2, &&
&\|\diff(x)\|&\leq C/2(1+\|x\|),\\ \|\vartheta_1(x,h)\|&\leq C/2, &&
&\|\vartheta_2(x,h)\|&\leq C/2(1+\|x\|).
\end{aligned}
\end{equation}
The first estimate in \eqref{eq:moment_bound_2} is due the definition \eqref{eq:defeta} of $\eta$, the definition \eqref{eq:defAB} of $\drift$ in the proof of Lemma~\ref{lem:extended_coefficients}, and Assumption~\ref{ass:g} on $g$; the estimate concerning $\diff$ is a consequence of \eqref{eq:assbabB}; the estimates concerning $\vartheta_1$ and $\vartheta_2$ follow from the Lagrange multiplier estimates \eqref{eq:estLambdarem} in Theorem~\ref{theo:mainonestep}.
Combining \eqref{eq:moment_bound_1} and \eqref{eq:moment_bound_2} leads to
\begin{align*}
\|X_{k+1}^N\|^2
&\leq  (\|X_k^N\|+C)^2+C^2+2C^2 (1+\|X_k^N\|)^2\|\Delta_hw_k\|^2\\
&\quad+2\big\langle X_k^N, \big(\diff(X_k^N)+\vartheta_2(X_k^N,h)\big) \Delta_hw_k\big\rangle.
\end{align*}
For the sake of better readability we use the formal notation 
$$\alpha_k(x):=\frac{\big\langle x, \big(\diff(x)+\vartheta_2(x,h)\big) \Delta_hw_k\big\rangle}{(\|x\|+2C)^2},\quad x\in\bR^{2n}.$$
Since $(|\xi|+C)^2+C^2\leq (|\xi|+2C)^2$ and $(1+\xi)\leq \exp(\xi)$ holds for all $\xi\in\mathbb R$ we have 
\begin{align*}
\|X_{k+1}^N\|^2 &\leq \big(\|X_k^N\|+2C\big)^2
\big(1+2C^2\|\Delta_hw_k\|^2+2\alpha_k(X_k^N)\big)\\
&\leq \big(\|X_k^N\|+2C\big)^2
\operatorname{exp}\big(2C^2\|\Delta_hw_k\|^2+2\alpha_k(X_k^N)\big)
\end{align*}
and therefore
\begin{align*}
\|X_{k+1}^N\|\leq \big(\|X_k^N\|+2C\big)
\operatorname{exp}\big(C^2\|\Delta_hw_k\|^2+\alpha_k(X_k^N)\big).
\end{align*}
This yields by recursion
\begin{align}\label{eq:tmpsplitcorveri}
\begin{split}
\|X_{k}^N\|&\leq\|X_{0}^N\| \prod_{j=0}^{k-1}\operatorname{exp}\big(C^2\|\Delta_hw_j\|^2+\alpha_j(X_j^N)\big)\\
&\quad+2C\sum_{i=0}^{k-1}\prod_{j=i}^{k-1}\operatorname{exp}\big(C^2\|\Delta_hw_j\|^2+\alpha_j(X_j^N)\big)
\end{split}
\end{align}
for all $k\in\{1,2,\dots,N\}$. 
Further, by the monotonicity of the exponential function we have
$
\exp\big(\sum_{j=i}^{k-1}C^2\|\Delta_hw_j\|^2\big)\leq \exp\big(\sum_{j=0}^{k-1}C^2\|\Delta_hw_j\|^2\big).
$
Using also H\"older's inequality and recalling that $X_0^N=X_0=(r_0,v_0)^\top$
we obtain
\begin{equation}\label{eq:moment_bound_3}
\begin{aligned}
\|X_{k}^N\|_{L_{p\varrho}}
&\leq \Big\|\prod_{j=0}^{k-1}\operatorname{exp}\big(C^2\|\Delta_hw_j\|^2\big)\Big\|_{L_{2p\varrho}}\\
&\quad \bigg\|\|X_{0}\| \prod_{j=0}^{k-1}\exp\big(\alpha_j(X_j^N)\big)+2C\sum_{i=0}^{k-1}\prod_{j=i}^{k-1}\exp\big(\alpha_j(X_j^N)\big)\bigg\|_{L_{2p\varrho}}\\
&\leq \Big\|\operatorname{exp}\Big(\sum_{j=0}^{k-1}C^2\|\Delta_hw_j\|^2\Big)\Big\|_{L_{2p\varrho}}\bigg(\|X_{0}\|_{L_{4p\varrho}}\Big\|\exp\Big(\sum_{j=0}^{k-1}\alpha_j(X_j^N)\Big)\Big\|_{L_{4p\varrho}}\\
&\quad
+2C\sum_{i=0}^{k-1}\Big\|\operatorname{exp}\Big(\sum_{j=i}^{k-1}\alpha_j(X_j^N)\Big)\Big\|_{L_{2p\varrho}}\bigg).
\end{aligned}
\end{equation}
We use the following estimates which are proven below: There exists a constant $N_{p\varrho}\in \bN$ such that for all $N\in\bN$ with $N\geq N_{p\varrho}$ we have 
\begin{align}
\sup_{k\in\{1,\dots, N\}}\mathbb E\Big[\exp\Big(2p\varrho C^2\sum_{j=0}^{k-1}\|\Delta_hw_j\|^2\Big)\Big]&\leq\exp\big(4p\varrho C^2T\ell\big)\label{eq:estiW2new}\\
\intertext{and for all $q\in[1,\infty)$, $N\in\bN$ it holds that}
\sup_{k\in\{1,\dots, N\}}\sup_{i\in\{0,\dots, k-1\}}\mathbb E\Big [\operatorname{exp}\Big ( q\sum_{j=i}^{k-1}\alpha_{j}(X_j^N)\Big)\Big]&\leq \exp\Big(\frac{q^2C^2T}2\Big).\label{eq:estialphanew}
\end{align}
Recall that $\ell\in\bN$ is the dimension of the driving Brownian motion.
Combining \eqref{eq:moment_bound_3}, \eqref{eq:estiW2new} and \eqref{eq:estialphanew}, once with $q=4p\varrho$ and once with $q=2p\varrho $, 
yields that there exists a constant $K_{p\varrho }\in[1,\infty)$, depending also on $T,\ell$ and $\|X_{0}\|_{L_{4p\varrho }}$, such that for all $N\in \bN$ with $N\geq N_{p\varrho }$ and all $k\in\{1,\dots, N\}$
\begin{align*}
\|X_{k}^N\|_{L_{p\varrho }}&\leq \exp(2C^2T\ell)\Big(\|X_{0}\|_{L_{4p\varrho }}\exp(2p\varrho C^2T)+2Ck\exp(p\varrho C^2T)\Big)\leq K_{p\varrho }k.
\end{align*}
As a immediate consequence we have 
$(\mathbb E\|X_{k}^N\|^{p\varrho })^{1/\varrho}\leq K_{p\varrho }^pk^p$ and therefore
\begin{equation}\label{eq:moment_bound_4}
\limsup_{N\to\infty} \sup_{k\in\{0,\dots,N\}} \Bigl(N^{-p}(\mathbb E\|X_{k}^N\|^{p\varrho })^{1/\varrho}\Bigr) \leq K_{p\varrho }^p<\infty.
\end{equation}
The assertion of Lemma~\ref{cor:strongrate} now follows from \eqref{eq:moment_bound_4} in combination with fact that $(1-\alpha)(1-1/\varrho)\leq(1-\alpha)\cdot 1/2\leq -2p\cdot 1/2 =-p$. 
The latter is a direct consequence of the assumption that $\varrho\in[2,\infty)$, $\alpha\in[2p+1,\infty)$.

It is left to verify the estimates \eqref{eq:estiW2new} and \eqref{eq:estialphanew}. Concerning \eqref{eq:estialphanew}, note that for all $q\in[1,\infty)$, $j\in\{0,\dots,N-1\}$ and $x\in\mathbb R^{2n}$ the random variable $\alpha_j(x)$ is normally distributed with mean zero, so that $\exp(q\alpha_j(x))$ is log-normally distributed with mean $\exp\big(\frac12\mathbb E|q\alpha_{j}(x)|^2\big)$.
By H\"older's inequality and \eqref{eq:moment_bound_2} we have
\begin{align*}
\mathbb E|q\alpha_{j}(x)|^2&=\frac{q^2}{(\|x\|+2C)^4}\mathbb E\big|\big\langle x, \big(\diff(x)+\vartheta_2(x,h)\big) \Delta_hw_j\big\rangle\big|^2\\
&\leq \frac{q^2}{(\|x\|+2C)^4}\|x\|^2\big\|\diff(x)+\vartheta_2(x,h)\big\|^2 h\\
&\leq \frac{q^2}{(\|x\|+2C)^4}\|x\|^2C^2(1+\|x\|)^2 h\\
&\leq q^2C^2h.
\end{align*}
Since $X_j^N$ is $\mathcal F_{jh}$-measurable and the Brownian increment $\Delta_hw_j$ is independent of $\mathcal F_{jh}$, it follows that
\begin{align*}
\mathbb E\big(\operatorname{exp}(q\alpha_{j}(X_j^N))\big|\mathcal F_{jh}\big)=\Big[\bE\big(\exp(q\alpha_j(x))\big)\Big]_{x=X_j^N}\leq \exp\Big(\frac{q^2C^2h}2\Big)
\end{align*}
and thus, taking also into account that $\alpha_j(X_j^N)$ is $\cF_{(k-1)h}$-measurable for all $j\in\{0,\ldots,k-2\}$,
\begin{align*}
\mathbb E\Big [\operatorname{exp}\Big (q \sum_{j=i}^{k-1}\alpha_{j}(X_j^N)\Big)\Big]
&= \mathbb E\Big [\operatorname{exp}\Big ( q\sum_{j=i}^{k-2}\alpha_{j}(X_j^N)\Big)\mathbb E\Big (\!\operatorname{exp}\big (q\alpha_{k-1}(X_{k-1}^N)\big)\Big|\mathcal F_{(k-1)h}\Big)\Big]\\
&\leq \mathbb E\Big [\operatorname{exp}\Big ( q\sum_{j=i}^{k-2}\alpha_{j}(X_j^N)\Big)\exp\Big(\frac{q^2C^2h}2\Big)\Big]
\end{align*}
Applying this estimate recursively yields \eqref{eq:estialphanew}.

Concerning the estimate \eqref{eq:estiW2new}, we use the independence of $(\Delta_hw_j)_{j\in\{0,\dots,N-1\}}$ and the scaling property of Brownian motion  to obtain 
\begin{align}\label{eq:estiW2tmp}
\begin{split}
\mathbb E\Big(\exp\Big(2p\varrho C^2\sum_{j=0}^{k-1}\|\Delta_hw_j\|^2\Big)\Big)&=\prod_{j=0}^{k-1}\mathbb E\big(\exp(2p\varrho C^2\|\Delta_hw_j\|^2)\big)\\
&=\Big(\mathbb E\big(\exp(2p\varrho C^2h\|\Delta_1w_0\|^2)\big)\Big)^{k}.
\end{split}
\end{align}
An elementary calculation shows that 
$\mathbb E\big(\exp(q\|\Delta_1w_0\|^2)\big)\leq\exp(2q\ell)$ for all $q\in[0,1/4]$, cf.~\cite[Lemma 3.2]{hutzenthaler:p:2012}.
Thus, choosing $N_{p\varrho }\in\bN$ large enough such that $2p\varrho C^2h=2p\varrho C^2T/N\leq 1/4$ for all $N\geq N_{p\varrho }$ we obtain the estimate 
\begin{equation}\label{eq:moment_bound_5}
\mathbb E\big(\exp(2p\varrho C^2h\|\Delta_1w_0\|^2)\big)\leq\exp(4p\varrho C^2h\ell).
\end{equation}
Finally, \eqref{eq:estiW2tmp} and \eqref{eq:moment_bound_5} imply \eqref{eq:estiW2new}, which finishes the proof of Lemma~\ref{cor:strongrate}.
\end{proof}

Altogether the previous results lead to the following strong convergence result.

\begin{corollary}\label{cor:uniform_conv_1}
		Let $(X(t))_{t\geq0}$ be the strong solution to the SDE~\eqref{eq:SDEstrongsetting} with initial condition $X_0=(r_0,v_0)^\top$. 
		For $N\in\bN$ let $(X^N	_k)_{k\in\{0,\ldots,N\}}$ be the solution to the half-explicit drift-truncated Euler scheme in the form \eqref{eq:appstrongdef}, with increment function $\phi=\phi_\cM$ given by \eqref{eq:algtype}, \eqref{eq:schemeinterpret}, and let $(\tilde X^N(t))_{t\in[0,T]}$ be defined by piecewise constant or piecewise linear interpolation of $(X^N_k)_{k\in\{0,\ldots,N\}}$. Then $\limsup_{N\to\infty}\sup_{t\in[0,T]}\mathbb E(\|\tilde X^N(t)\|^p)<\infty$ and
		\begin{align*}
			\lim_{N\to\infty}\sup_{t\in[0,T]} \mathbb E\big(\|X(t)-\tilde X^N(t)\|^{p}\big)=0
		\end{align*}
for all $p\in[1,\infty)$.	
\end{corollary}\label{cor:conv_position_velocity}
\begin{proof}
	We first treat the case where $(\tilde X^N(t))_{t\in[0,T]}$ is defined by piecewise linear interpolation.
	Let $p\in[1,\infty)$, $\alpha\in[2p+1,\infty)$ and $\varrho\in[\alpha(2p_a+1),\infty)$. Thanks to the consistency result in Lemma~\ref{lem:varifyconsistent}, the semi-stability result in Lemma~\ref{lem:verifystable} and the moment growth property from Lemma~\ref{cor:strongrate}, we are able to apply Proposition~\ref{cor:mainjentzen}. We obtain that
	$\limsup_{N\to\infty}\sup_{t\in[0,T]}\mathbb E(\|\tilde X^N(t)\|^p)<\infty$ and, for all $q\in(0,p)$,
	$
	\lim_{N\to\infty}\sup_{t\in[0,T]} \mathbb E\big(\|X(t)-\tilde X^N(t)\|^{q}\big)=0.
	$
	Since $p$ has been chosen arbitrarily this yields the assertion.
	Now let $(\tilde X^N(t))_{t\in[0,T]}$ be defined by piecewise constant interpolation. 
	In this case the assertion follows from he first part of this proof and the continuity of the mapping $[0,T]\ni t\mapsto X(t)\in L_p(\Omega;\bR^{2n})$ for all $p\in[1,\infty)$.
	The latter is a consequence of the dominated convergence theorem, using the sample path continuity of $(X(t))_{t\in[0,T]}$ and the moment bound $\sup_{t\in[0,T]}\bE(\|X(t)\|^p)=\sup_{t\in[0,T]}\bE(\|(r(t),v(t))\|^p)<\infty$ following from Theorem~\ref{theo:solutionandinherent}.
\end{proof}

\subsection{Pathwise uniform convergence}\label{subsec:convergence_mu}

Here we refine the convergence result from Corollary~\ref{cor:uniform_conv_1} and show pathwise uniform convergence of $\tilde X^N=(\tilde r^N,\tilde v^N)^\top$ and $\tilde \mu^N$. The assertion of Theorem~\ref{theo:mainend} is a direct consequence of Propositions~\ref{prop:uniform_conv_1}, \ref{prop:uniform_conv_2} and Lemma~\ref{lem:uniform_conv_2} below.
As before, we always suppose that the assumptions in Section~\ref{subsec:assumptions} are fulfilled and consider the setting described in Section~\ref{subsec:SDEsetting}. 

\begin{proposition}\label{prop:uniform_conv_1}
		Let $(X(t))_{t\geq0}$ be the strong solution to the SDE~\eqref{eq:SDEstrongsetting} with initial condition $X_0=(r_0,v_0)^\top$. 
		For $N\in\bN$ set $h:=T/N$ and let $(X^N	_k)_{k\in\{0,\ldots,N\}}$ be the solution to the half-explicit drift-truncated Euler scheme in the form \eqref{eq:appstrongdef}, with increment function $\phi=\phi_\cM$ given by \eqref{eq:algtype}, \eqref{eq:schemeinterpret}. Then, for all $p\in[1,\infty)$ it holds that
	\begin{align*}
	\lim_{N\to\infty}\mathbb E\Big(\sup_{k\in\{0,\ldots,N\}}\big\|X(kh)-X^N_k\big\|^{p}\Big)=0.
	\end{align*}
\end{proposition}

\begin{proof}
To simplify the exposition we introduce the notation $k_s:=k^N_s:=\lfloor s/h\rfloor
=\lfloor s/(T/N)\rfloor$ for $s\in[0,T]$, $N\in\bN$. Using also the notation from Section~\ref{subsec:SDEsetting},
we have
\begin{equation}\label{eq:uniform_conv_1}
\begin{aligned}
X^N_k&=X_0+\sum_{j=0}^{k-1}\Big(\eta_0\big(X^N_j,h,\Delta_hw_j\big)+\eta\big(X^N_j,h\big)\drift(X^N_j)h+\diff(X^N_j)\Delta_hw_j\Big),
\end{aligned}
\end{equation}
where $\eta_0\colon\bR^{2n}\times[0,T]\times\bR^\ell\to\bR^{2n}$ is defined by \eqref{eq:schemeinterpret} and $\eta\colon\bR^{2n}\cong\bR^n\times\bR^n\to(0,1]$ is the truncation function defined by \eqref{eq:defeta}. As a consequence,
\begin{align*}
X(kh)-X^N_k
&=
-\sum_{j=0}^{k-1}\eta_0\big(X^N_j,h,\Delta_hw_j\big)
+
\int_0^{kh}\big(\drift(X(s))-\eta(X_{k_s}^N,h)\drift(X^N_{k_s})\big)\dl s\\
&\quad+
\int_0^{kh}\big(\diff(X(s))-\diff(X^N_{k_s})\big)\dl w(s).
\end{align*}
Without loss of generality we assume that $p\geq 2$. Minkowski's integral inequality and the Burkholder-Davis-Gundy inequality in the form \cite[Lemma~7.7]{daprato:b:2014} applied to \eqref{eq:uniform_conv_1} yield that there exists a constant $C_p\in(0,\infty)$ that does not depend on $N$ such that
\begin{equation}\label{eq:uniform_conv_2}
\begin{aligned}
\Big\|\sup_{k\in\{0,\ldots,N\}}\big\|X(kh)-X^N_k\big\|\Big\|_{L_p}
&\leq
\sum_{j=0}^{N-1}\big\|\eta_0\big(X^N_j,h,\Delta_hw_j\big)\big\|_{L_p}\\
&\quad+
\int_0^T\big\|\drift(X(s))-\eta(X_{k_s}^N,h)\drift(X^N_{k_s})\big\|_{L_p}\dl s\\
&\quad+
C_p \Big(\int_0^T\big\|\diff(X(s))-\diff(X^N_{k_s})\big\|_{L_p}^2\dl s\Big)^{1/2}\\
&=:I^N+I\!I^N+I\!I\!I^N.
\end{aligned}
\end{equation}
In view of the definition \eqref{eq:schemeinterpret} of $\eta_0$ we can use Lemma~\ref{lem:uniform_conv_1} below and Assumption~\ref{ass:g} to obtain that $\lim_{N\to\infty}I^N=0$ in \eqref{eq:uniform_conv_2}. 
Concerning the term $I\!I^N$ in \eqref{eq:uniform_conv_2} first note that
\begin{align}\label{eq:uniform_conv_3}
\eta(X^N_{k_s},h)\drift(X^N_{k_s})\xrightarrow[N\to\infty]{\bP}\drift(X(s))\;\text{ for all }s\in[0,T].
\end{align} 
Indeed, since $(X^N_{k_s},h)\xrightarrow{\bP}(X(s),0)$ for all $s\in[0,T]$ due to Corollary~\ref{cor:conv_position_velocity}, the continuous mapping theorem applied to the function $\bR^{2n}\times[0,\infty)\ni(x,h)\mapsto \eta(x,h)\drift(x)\in\bR^{2n}$ and the fact that $\eta(x,0)=1$ for all $x\in\bR^{2n}$ imply the convergence in probability \eqref{eq:uniform_conv_3}.
Moreover, the polynomial growth of $\drift$ stated in \eqref{eq:assbabB} and the moment bounds $\limsup_{N\to\infty}\sup_{k\in\{0,\ldots,N\}}\bE(\|X^N_k\|^q)<\infty$, $\sup_{s\in[0,T]}\bE(\|X(s)\|^q)<\infty$, $q\in[1,\infty)$, following from Corollary~\ref{cor:conv_position_velocity} and Theorem~\ref{theo:solutionandinherent} imply that for all $p\in[1,\infty)$ there exists $N_0\in\bN$ such that the family of random variables
\begin{equation}\label{eq:uniform_conv_4}
\begin{aligned}
\big\{\|\eta(X^N_k,h)\drift(X^N_k)\|^p:N\geq N_0,\,k\in\{0,\ldots,N\}\big\}\cup\big\{\|\drift(X(s))\|^p:s\in[0,T]\big\}
\end{aligned}
\end{equation}
is uniformly integrable. Thus, Vitali's convergence theorem and \eqref{eq:uniform_conv_3}, \eqref{eq:uniform_conv_4} yield that $
\eta(X^N_{k_s},h)\drift(X^N_{k_s})\xrightarrow{L_p}\drift(X(s))$ for all $s\in[0,T]$, $p\in[1,\infty)$. This and the fact that $\limsup_{N\to\infty}\sup_{s\in[0,T]}\|\drift(X(s))-\eta(X_{k_s}^N,h)\drift(X^N_{k_s})\|_{L_p}<\infty$ allows us to apply the dominated convergence theorem in order to deduce that $\lim_{N\to\infty}I\!I^N=0$.
One can finally use completely analogous arguments as for the term $I\!I^N$ in \eqref{eq:uniform_conv_2} to show that also $\lim_{N\to\infty}I\!I\!I^N=0$. 
\end{proof}

\begin{lemma}\label{lem:uniform_conv_1}
Consider the setting described in Theorem~\ref{theo:mainend} and let the mappings
$\hat\kappa\colon\cD\to\bR^m$, $\lambdarem\colon\cD\to\bR^m$, $\Lambdarem\colon\cD\to\bR^{m\times\ell}$  with $T\cM\times(0,\infty)\subset\cD\subset\bR^{2n}\times(0,\infty)$ 
be given by Theorem~\ref{theo:mainonestep}. Then, for all $p\in[1,\infty)$ there exist constants $C\in(0,\infty)$, $N_0\in\bN$ such that
\begin{equation}\label{eq:uniform_conv_5}
\begin{aligned}
\sum_{j=0}^{N-1}\Big(\big\|\hat\kappa\big(r^N_j,v^N_j,h\big)\big\|_{L_p}+\big\|\lambdarem\big(r^N_j,v^N_j,h\big)\big\|_{L_p}\Big)&\leq Ch,\\
\sum_{j=0}^{N-1}\big\|\Lambdarem\big(r^N_j,v^N_j,h\big)\Delta_hw_{j}\big\|_{L_p}&\leq C\sqrt{h},
\end{aligned}
\end{equation}
for all $N\in\bN$ with $N\geq N_0$.
\end{lemma}

\begin{proof}
The assertion is a direct consequence of the Lagrange multiplier estimates \eqref{eq:estLambdarem} and the moment bound $\limsup_{N\to\infty}\sup_{k\in\{0,\ldots,N\}}\bE(\|(r^N_k,v^N_k)\|^q)<\infty$, $q\in[1,\infty)$, following from Corollary~\ref{cor:conv_position_velocity}. For instance, we have
\begin{align*}
\sum_{j=0}^{N-1}\big\|\Lambdarem\big(r^N_j,v^N_j,h\big)\Delta_hw_{j}\big\|_{L_p}
&\leq
\sum_{j=0}^{N-1}\big\|\Lambdarem\big(r^N_j,v^N_j,h\big)\big\|_{L_{2p}}\|\Delta_hw_{j}\|_{L_{2p}}\\
&\leq
\sum_{j=0}^{N-1}h\big\|1+\|(r^N_j,v^N_j)\|^{p_\Lambda}\big\|_{L_{2p}}\sqrt{h}\|\Delta_1w_0\|_{L_{2p}}\\
&\leq 
C_p\sum_{j=0}^{N-1}h^{3/2}=C_pT\sqrt{h}
\end{align*}
for all $N\geq N_0$, where $C_p\in(0,\infty)$ does not depend on $N$ and where $N_0\in\bN$ is chosen such that $\sup_{N \geq N_0}\big\|1+\|(r^N_j,v^N_j)\|^{p_\Lambda}\big\|_{L_{2p}}<\infty$. The first estimate in \eqref{eq:uniform_conv_5} can be shown analogously.
\end{proof}

\begin{proposition}\label{prop:uniform_conv_2}
Consider the setting described in Theorem~\ref{theo:mainend}. Then, for all $p\in[1,\infty)$ it holds that
	\begin{align*}
	\lim_{N\to\infty}\mathbb E\Big(\sup_{k\in\{0,\ldots,N\}}\big\|\mu(kh)-\mu^N_k\big\|^{p}\Big)=0.
	\end{align*}
\end{proposition}

\begin{proof}
Let the mappings $f\colon\cM^\varepsilon\times\bR^n\to\bR^m$ and $F\colon \cM^\varepsilon\times\bR^n\to\bR^{m\times\ell}$ be defined by
\begin{align*}
f(x,y)&:=-G_M^{-1}(x)Dg(x)M^{-1}\big[a(x,y)+D^2g(x)(y,y)\big]\\
F(x,y)&:=-G_M^{-1}(x)Dg(x)M^{-1}B(x,y).
\end{align*}
Then the representation \eqref{eq:defmu} of $(\mu(t))_{t\in[0,T]}$
can be written as
\begin{equation}
\begin{aligned}
\mu(t)
&=
\int_0^tf(r(s),v(s))\,\dl s+\int_0^tF(r(s),v(s))\,\dl w(s)
\end{aligned}
\end{equation}
and, according to Theorem~\ref{theo:mainonestep} and Corollary~\ref{cor:applyingonestep}, the process $(\mu^N_k)_{k\in\{1,\ldots,N\}}$ satisfies
\begin{equation}
\begin{aligned}
\mu^N_{k}
&=
\sum_{j=0}^{k-1}\bigg\{\eta(r^N_{j},v^N_{j},h)\,f(r^N_j,v^N_j)\,h+F(r^N_j,v^N_j)\,\Delta_hw_{j}\\
&\quad+\lambdarem(r^N_j,v^N_j,h)+\Lambdarem(r^N_j,v^N_j,h)\,\Delta_hw_{j}\bigg\}.
\end{aligned}
\end{equation}
As in the proof of Proposition~\ref{prop:uniform_conv_1} we use the notation $k_s:=k^N_s:=\lfloor s/h\rfloor
=\lfloor s/(T/N)\rfloor$ for $s\in[0,T]$, $N\in\bN$. Consider the time-interpolated perturbation $(\bar\mu^N(t))_{t\in[0,T]}$  of $(\mu^N_k)_{k\in\{1,\ldots,N\}}$ defined by
\begin{equation}
\begin{aligned}
\bar\mu^N(t)
&:=
\int_0^t\eta(r^N_{k_s},v^N_{k_s},h)\,f(r^N_{k_s},v^N_{k_s})\,\dl s+\int_0^tF(r^N_{k_s},v^N_{k_s})\,\dl w(s).
\end{aligned}
\end{equation}
Next, we assume without loss of generality that $p\geq 2$ and apply Minkowski's integral inequality and the Burkholder-Davis-Gundy inequality in the form \cite[Lemma~7.7]{daprato:b:2014} to obtain
\begin{equation}\label{eq:uniform_conv_6}
\begin{aligned}
\Big\|\sup_{t\in[0,T]}\big\|\mu(t)-\bar\mu^N(t)\big\|\Big\|_{L_p}
&\leq
\int_0^T\big\|f(r(s),v(s))-\eta(r^N_{k_s},v^N_{k_s},h)f(r^N_{k_s},v^N_{k_s})\big\|_{L_p}\dl s\\
&\quad+
C_p \Big(\int_0^T\big\|F(r(s),v(s))-F(r^N_{k_s},v^N_{k_s})\big\|_{L_p}^2\dl s\Big)^{1/2},
\end{aligned}
\end{equation}
where $C_p\in(0,\infty)$ does not depend on $N$.
Using analogous arguments as for the terms in the right hand side of \eqref{eq:uniform_conv_2} in the proof of Proposition~\ref{prop:uniform_conv_1}, one sees that the terms on the right hand side of \eqref{eq:uniform_conv_6} converge to zero as $N\to\infty$.
Finally, the identity $\hat\mu^N(kh)-\mu^N_k=-\sum_{j=0}^{k-1}\big(\lambdarem(r^N_j,v^N_j,h)+\Lambdarem(r^N_j,v^N_j,h)\,\Delta_hw_{j}\big)$ and Lemma~\ref{lem:uniform_conv_1} yield 
\begin{equation}\label{eq:uniform_conv_7}
\begin{aligned}
\lim_{N\to\infty}\Big\|\sup_{k\in\{1,\ldots,N\}}\big\|\bar\mu^N(kh)-\mu^N_k\big\|\Big\|_{L_p}=0.
\end{aligned}
\end{equation}
Combining \eqref{eq:uniform_conv_6} and \eqref{eq:uniform_conv_7} finishes the proof.
\end{proof}

The following lemma completes the proof of Theorem~\ref{theo:mainend}. Its verification is straightforward and therefore left to the reader.

\begin{lemma}\label{lem:uniform_conv_2}
Let $(Y(t))_{t\in[0,T]}$, $(Y^N_k)_{k\in\{0,\ldots,N\}}$, $N\in\bN$, be stochastic processes taking values in a Banach space $(B,\|\cdot\|)$, defined on the same underlying probability space. 
Assume that $\lim_{N\to\infty}\bE\big(\sup_{k\in\{0,\ldots,N\}}\|Y(kT/N)-Y^N_k\|^p\big)=0$ and $\bE\big(\sup_{t\in[0,T]}\|Y(t)\|^p\big)<\infty$ for some $p\in[1,\infty)$, and that $(Y(t))_{t\in[0,T]}$ has continuous sample paths. Let $(\tilde Y(t))_{t\in[0,T]}$ be defined by piecewise constant or piecewise linear interpolation of $(Y^N_k)_{k\in\{0,\ldots,N\}}$, cf.~Remark~\ref{rem:interpol}. Then we also have
$
\lim_{N\to\infty}\bE\big(\sup_{t\in[0,T]}\|Y(t)-\tilde Y^N(t)\|^p\big)=0.
$
\end{lemma}

\section{Examples}\label{sec:examples}
Here we present examples of SDAEs which fit into our setting.

\subsection{Stochastic pendulum}\label{subsec:pendel}
The most commonly known example from the DAE context is the pendulum. 
We use it as a toy example in order to illustrate our theory.
For simplicity, a $2$-dimensional setting is considered.

We are interested in approximating two $\bR^2$-valued processes $(r(t))_{t\in[0,T]}$ and $(v(t))_{t\in[0,T]}$ which model the position and velocity of the endpoint of an 
idealized pendulum with unit length. 
We consider two external forces, 
a gravitational force scaled by some constant $c_g\geq0$ and a stochastic force modelled by the increments of an $\bR^2$-valued Brownian motion $(w(t))_{t\in[0,T]}$. Assuming also unit mass $M=\Id$, 
the corresponding stochastic differential-algebraic system reads
\begin{equation*}
\begin{aligned}
\dl r(t)&=v(t)\,\dl t\\
\dl v(t)&=\begin{pmatrix}
0\\-c_g
\end{pmatrix}\dl t+\dl w(t)-r(t)\,\dl \mu(t)\\
\|r(t)\|&=1.
\end{aligned}
\end{equation*}
Note that this is an SDAE of the type \eqref{eq:mainSDAE} if we 
rewrite the constraint equivalently as $g(r(t))=0$ with the quadratic constraint function
$$g(x):=\frac12(1-\|x\|^2),\quad x\in\bR^2.$$
In this context, the half-explicit drift-truncated Euler scheme \eqref{eq:approximationintro}, respectively \eqref{eq:approximation}, reads 
\begin{subequations}
\begin{align}
\begin{split}\label{eq:pendulumeqnonlin}
r_{k+1}&=r_k+\eta_k\, v_k\, h-r_k\, \kappa_{k+1}\\
\frac12(1-\|r_{k+1}\|^2)&=0
\end{split}\\
\begin{split}\label{eq:pendulumeqlin}
v_{k+1}&=v_k+\eta_k\left(\begin{array}{c} 0\\-1\end{array}\right)h+\Delta_hw_k-r_k\,\lambda_{k+1}\\
-\langle r_{k+1},v_{k+1}\rangle&=0
\end{split}
\end{align}
\end{subequations}
with deterministic initial values $(r_0,v_0)\in T\cM$, say.
As before,  $h=T/N$ is the step size, $\eta_k=\eta(r_k,v_k,h)$ is defined by \eqref{eq:defetaintro}, respectively \eqref{eq:defeta}, and $\Delta_hw_k=w((k+1)h)-w(kh)$ denotes an increment of Brownian motion.

In order to illustrate our algorithm from a geometrical point of view, let us first consider what happens if we omit the truncation factor $\eta_k$. In the constellation shown in Figure~\ref{fig:sub1}, the 
system \eqref{eq:pendulumeqnonlin} with $1$ in place of $\eta_k$ 
has two solutions, but it may also happen that there is no solution at all as shown in Figure~\ref{fig:sub2}. 
The truncation factor $\eta_k$ ensures that the latter case does not occur. Note that Assumption \ref{ass:g} is fulfilled for all $\varepsilon\in(0,1)$ and, depending on the choice of $\varepsilon$, the constants appearing in Theorem~\ref{theo:mainonestep} and Corollary~\ref{cor:applyingonestep} are such that $C_g\in(1,2)$, $\Ceta\in(1/32,1/4)$ and $\Clm\in(1/128,1/8)$. As $\|\eta_kv_kh\|$ is bounded by $C_\eta$, the system \eqref{eq:pendulumeqnonlin} thus has always two solutions but only one with $|\kappa_{k+1}|< C_\kappa$.
\begin{figure}[h!]
	\centering
	\begin{subfigure}{.5\textwidth}
		\centering
			\begin{tikzpicture}[scale=1.6,cap=round]
			\tikzstyle{axes}=[]
			\colorlet{wrongcolor}{red}
			\colorlet{problemcolor}{red}
			\colorlet{rightcolor}{green}
			\colorlet{stepcolor}{blue}
			\draw [very thick] circle(1cm);
			\begin{scope}[style=axes]
			\draw (-1.5,0) -- (2.2,0) node[right] {};
			\draw (0,-1.5) -- (0,1.5) node[above] {};
			\foreach \x/\xtext in {-1,1}
			\draw[xshift=\x cm] (0pt,1pt) -- (0pt,-1pt) node[below,fill=white]
			{$\xtext$};
			\foreach \y/\ytext in {-1,1}
			\draw[yshift=\y cm] (1pt,0pt) -- (-1pt,0pt) node[left,fill=white]
			{$\ytext$};
			\end{scope}
			\draw[->,thick] (0,0) -- node[above left=1pt,fill=white] {$r_k$} (0.5,0.866);
			\draw[->,thick] (0,0) -- node[below=2pt,fill=white] {$r_{k+1}$} (0.9928,0.1196);
			\draw [help lines,rightcolor](1.6928,1.332) -- (0.0928,-1.44) node[below] {};
			\draw[->,very thick,stepcolor]
			(0.5,0.866) -- node[above right=3pt,fill=white] {$v_k h$} (1.1928,0.466);
			\draw[->,very thick,wrongcolor]
			(1.1928,0.466) -- node[below right=3pt,fill=white] {\small{solution 2}} (0.3928,-0.9196);
			\draw[->,very thick,rightcolor]
			(1.1928,0.466) -- node[below right=3pt,fill=white] {\small{solution 1}} (0.9928,0.1196);
			\end{tikzpicture}
		\caption{Step with two solutions}
		\label{fig:sub1}
	\end{subfigure}%
	\begin{subfigure}{.5\textwidth}
		\centering
			\begin{tikzpicture}[scale=1.6,cap=round]
			\tikzstyle{axes}=[]
			\colorlet{wrongcolor}{red}
			\colorlet{problemcolor}{red}
			\colorlet{rightcolor}{green}
			\colorlet{stepcolor}{blue}
			\draw [very thick] circle(1cm);
			\begin{scope}[style=axes]
			\draw (-1.5,0) -- (2.2,0) node[right] {};
			\draw (0,-1.5) -- (0,1.5) node[above] {};
			\foreach \x/\xtext in {-1,1}
			\draw[xshift=\x cm] (0pt,1pt) -- (0pt,-1pt) node[below,fill=white]
			{$\xtext$};
			\foreach \y/\ytext in {-1,1}
			\draw[yshift=\y cm] (1pt,0pt) -- (-1pt,0pt) node[left,fill=white]
			{$\ytext$};
			\end{scope}
			\draw[->,thick] (0,0) -- node[above left=1pt,fill=white] {$r_k$} (0.5,0.866);
			\draw [help lines,wrongcolor](2.0392,1.132) -- (0.5392,-1.466) node[below] {};
			\draw[->,very thick,stepcolor]
			(0.5,0.866) -- node[above right=3pt,fill=white] {$v_k h$} (1.5392,0.266);
			\end{tikzpicture}
		\caption{Step without solution}
		\label{fig:sub2}
	\end{subfigure}
	\caption{An illustration of two non-truncated cases}
	\label{fig:test}
\end{figure}
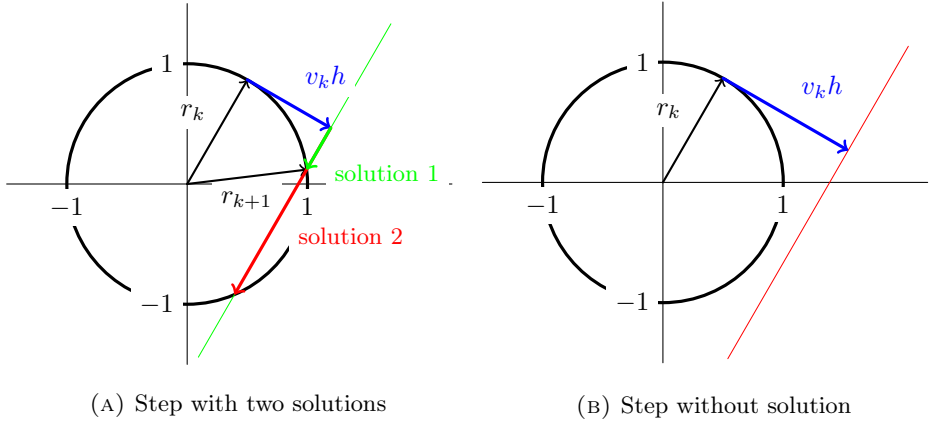

The stochastic pendulum also shows that even if the coefficients $a$ and $B$ in \eqref{eq:mainSDAE} are chosen to be constant, the drift coefficient $\drift$ of the corresponding inherent SDE \eqref{eq:inherentSDE} in the form \eqref{eq:SDEstrongsetting} may fail to be globally one-sided Lipschitz continuous. 
Recall that the locally Lipschitz continuous mapping $\drift\colon\bR^{4}\to\bR^{4}$ is said to be globally one-sided Lipschitz continuous if there exists a constant $C\in(0,\infty)$ such that
$\langle x-y,\drift(x)-\drift(y)\rangle\leq C\|x-y\|^2\,\text{ for all }x,y\in\bR^{4}$.
This property is of interest as it is well-known that many truncated and tamed Euler-type schemes have strong order of convergence $1/2$ if the drift coefficient of the underlying SDE fulfills, among others, a one-sided Lipschitz condition, see, e.g.,~\cite{beyn:p:2016,fang:p:2016,higham:p:2002,hutzenthaler:p:2015,hutzenthaler:p:2012,mao:p:2016,sabanis:p:2013}. 
To see that such a condition does not hold in our example let us set $c_g=0$ for simplicity. The inherent SDE \eqref{eq:inherentSDE} then takes the form
\begin{align*}
\dl r(t)&=v(t)\,\dl t\\
\dl v(t)&=-\frac{\|v(t)\|^2}{\|r(t)\|^2}r(t)\,\dl t+\Big(\Id-\frac1{\|r(t)\|^2}r(t)(r(t))^\top\Big)\,\dl w(t),
\end{align*}
where $\Id\in\bR^{2\times 2}$ is the identity matrix. Thus, using the notation from Section~\ref{subsec:SDEsetting}, we have 
$\drift(x)=\big(x^{(2)},-\|x^{(2)}\|^2x^{(1)}\big)^\top$ 
for all $x=(x^{(1)},x^{(2)})^\top\in T\cM\subset\bR^4$.
Now let $C\in(0,\infty)$ be arbitrary and consider $x=(x^{(1)},x^{(2)})^\top$, $y=(y^{(1)},y^{(2)})^\top\in T\cM$ such that $x^{(1)},y^{(1)}\in\bR^2$ 
are orthogonal 
and $x^{(2)}=2(C\vee 1)y^{(1)}$, $y^{(2)}=2(C\vee 1)x^{(1)}$.
A simple calculation shows 
\begin{align*}
\langle x-y,\,\drift(x)-\drift(y)\rangle
&=
\big\langle x^{(1)}-y^{(1)},\,x^{(2)}-y^{(2)}\big\rangle\\
&\quad+
\big\langle x^{(2)}-y^{(2)},\,-\|x^{(2)}\|^2x^{(1)}+\|y^{(2)}\|^2y^{(1)}\big\rangle\\
&=
\big(-2(C\vee 1)+(2(C\vee 1))^3\big)\big\|x^{(1)}-y^{(1)}\big\|^2\\
&>C\big(1+4(C\vee 1)^2\big)\big\|x^{(1)}-y^{(1)}\big\|^2\\
&=
C\|x-y\|^2,
\end{align*}
i.e., $\drift$ is not one-sided Lipschitz continuous.

\subsection{Constrained Langevin dynamics}\label{subsec:langevin}
The probably best studied class of equations that fits into our setting are the constrained Langevin-type equations used in molecular dynamics, see~\cite{leimkuhler:p:2016,lelievre:p:2012,ciccotti:p:2006} and the references therein.
For instance, in \cite{lelievre:p:2012} the authors consider equations of the form
\begin{align*}
&\dl q_t=M^{-1}p_t\,\dl t\\
&\dl p_t=-\nabla V(q_t)\,\dl t-\gamma (q_t) M^{-1}p_t\,\dl t+\sigma (q_t)\,\dl W_t+\nabla \xi (q_t)\,\dl \lambda_t\\
&\xi (q_t)=z,
\end{align*}
where $V$ is a potential and $\gamma$ a friction function.
Under mild assumptions on the coefficients an the constraint function $\xi$ this system fits into our setting with $r(t)=q_t$, $v(t)=M^{-1}p_t$, $\mu(t)=\lambda_t$, $a(x,y)=[-\nabla V(x)-\gamma(x)y]$, $B(x,y)=\sigma(x)$ and $g(x)=\xi(x)-z$. 
Note that for such equations it is by now standard to use splitting schemes, consisting of consecutive pushes in separate directions which correspond to different parts of the dynamics, cf.~\cite{lelievre:p:2012,leimkuhler:p:2016}.
Such schemes are tailor-made for the considered specific types of equations and lie beyond the scope of the present article.

\subsection{Stochastic fiber dynamics}\label{subsec:fiber}
In the context of the modeling and simulation of industrial production processes of non-woven technical textiles, one is interested in the dynamics of long slender elastic inextensible fibers in a turbulent airflow \cite{marheineke:p:2006,marheineke:p:2011}. 
Considering a fixed spatial semi-discretization of the corresponding stochastic space-time model leads to the dynamics of a chain of $N$ equidistant points $\boldsymbol r_i(t)$, $i=1,\ldots,N$, in three-dimensional Euclidean space, described by an SDAE of the form
\begin{align*}
\dl r(t)&=v(t)\,\dl t\\
\dl v(t)&=\Big[f_{\text{int}}(r(t))+ f_{\text{grav}}+f_{\text{air}}\big(r(t),v(t)\big)\Big] \dl t\\
&\quad +L_{\text{air}}\big(r(t),r(t)\big)D_{\text{turb}}\big(r(t),v(t)\big)\,\dl w(t)+\nabla g(r(t))\,\dl \mu(t)\\
g(r(t))&=0,
\end{align*}
see  \cite{lindner:p:2016,lindner:p:2018}.
Here $r(t)=(\boldsymbol r_i(t))_{i\in\{1,\ldots,N\}}$ is the $3N$-dimensional position vector, $v(t)=(\boldsymbol v_i(t))_{i\in\{1,\ldots,N\}}$ is the corresponding $3N$-dimensional velocity vector and $w$ is a $N$-dimensional Brownian motion.
The inextensibility constraint function $g\colon\bR^{3N}\to\bR^{N-1}$ is given by
$$g(x)=\Big(\frac12-\frac{\|\boldsymbol x_{i+1}-\boldsymbol x_i\|^2}{2(\Delta s)^2}\Big)_{i\in\{1,\ldots,N-1\}},$$
$x=(\boldsymbol{x}_i)_{i\in\{1,\ldots,N\}}\in\bR^{3N}$, where $\Delta s>0$ is a fixed spatial discretization parameter. 
Accordingly,  the Lagrange multiplier process $\mu$ takes values in $\bR^{N-1}$;
 the corresponding term $\nabla g(r(t)) \dl \mu(t)$ models the inner traction forces.
The internal forces $f_{\text{int}}$ are due to the bending stiffness of the fiber and are described by a linear operator which represents a discretized higher-order spatial derivative.
Further, $f_{\text{grav}}$ is a  constant gravitational force and the turbulent air-drag force consists of a deterministic part $f_{air}$ and a stochastic part involving the matrix functions $L_{\text{air}}$ and $D_{\text{turb}}$, see \cite{lindner:p:2016,marheineke:p:2011} for details concerning the modeling of the coefficients.

\begin{appendix}
\section{Solvability of the SDAE}\label{sec:appsolv}

Here we sketch of the key steps of the proof of Theorem~\ref{theo:solutionandinherent}, which is a slight generalization of \cite[Theorem~3.1]{lindner:p:2016}. For more details we refer to the respective arguments in \cite{lindner:p:2016}.
Theorem~\ref{theo:solutionandinherent} is a consequence of Lemma~\ref{lem:inherentSDE_app} and Theorem~\ref{theo:solvability_app} below.

\begin{lemma}
\label{lem:inherentSDE_app}
Let the assumptions in Section~\ref{subsec:assumptions} be fulfilled.
	If $(r,v,\mu)$ is a global strong solution to the SDAE~\eqref{eq:mainSDAE} with initial conditions $r_0,v_0$, then $\mu$ has the representation \eqref{eq:defmu} and $(r,v)$ is a global strong solution to the inherent SDE~\eqref{eq:inherentSDE}.
	Conversely, consider arbitrary locally Lipschitz continuous extensions of the coefficients in the inherent SDE~\eqref{eq:inherentSDE}
to the whole space $\bR^n\times\bR^n$, and let $(r,v)$ be a global strong solution to \eqref{eq:inherentSDE} with initial conditions $r_0,v_0$.
Then we have $\bP\big(g(r(t))=0\text{ for all }t\geq0\big)=1$ and, if further $\mu$ is the continuous, $\bR^{m}$-valued semimartingale defined by \eqref{eq:defmu}, then $(r,v,\mu)$ is a global strong solution to the SDAE~\eqref{eq:mainSDAE}.
\end{lemma}

\begin{proof}
	Let $(r,v,\mu)$ be a global strong solution to the SDAE~\eqref{eq:mainSDAE} with initial conditions $r_0,v_0$. With probability one we have 
	$Dg(r(t))v(t)=\frac{\dl}{\dl t} g(r(t))=0$ for all $t\geq0$.
	As $g$ is a $C^3$-function by Assumption~\ref{ass:g}, we can apply It\^{o}'s formula to the function $f\colon\mathbb R^n\times \mathbb R^n\to\mathbb R^m\colon (x,y)\mapsto Dg(x)y$  and the process $(r,v)$ to obtain
	\begin{align}\label{eq:Ito_app_1}
	0
	&=
	\int_0^t D^2g(r(s))(v(s),v(s))\,\dl s+\int_0^t Dg(r(s))\,\dl v(s)
	\end{align}
	for all $t\geq0$, $\bP$-a.s..
	Note that all second order terms in It\^{o}'s formula vanish since $r$ is a bounded variation process and $f$ is linear in $y$.
	Integrating $(G_M^{-1}(r(t)))_{t\geq0}$ w.r.t the process on the right hand side of \eqref{eq:Ito_app_1} and using that
	\begin{align*}
	\dl v(t)	
	&=
	M^{-1}\Big[a(r(t),v(t))\,\dl t+ B(r(t),v(t))\,\dl w(t)+\nabla g(r(t))\,\dl \mu(t)\Big]
	\end{align*}
	by assumption, we obtain that $\mu$ satisfies \eqref{eq:defmu}. As a consequence, $(r,v)$ solves the inherent SDE~\ref{eq:inherentSDE}.
	The converse statement is shown similarly.
\end{proof}

\begin{theorem}
\label{theo:solvability_app}
Under the assumptions in Section~\ref{subsec:assumptions}, there exists a unique (up to indistinguishability) global strong solution $(r,v)$ the the inherent SDE~\eqref{eq:inherentSDE}.
For all $p,T\in[1,\infty)$  the $p$-th moment $\bE\big(\sup_{t\in[0,T]}\|(r(t),v(t))\|^p\big)$ is finite.
\end{theorem}
\begin{proof}
	Due to our assumptions, the coefficent functions of the inherent SDE \eqref{eq:inherentSDE} can be extended to locally Lipschitz continuous functions defined on the whole space $\bR^n\times \bR^n$. 
	Hence, there exists a unique (up to indistinguishability) strong solution $(r,v)$ to \eqref{eq:inherentSDE} up to its explosion time, see, e.g., 
	\cite[Theorem~1.1.8]{hsu:b:2002}. 
	The latter is a $\bP$-a.s.\ uniquely defined $(\cF_t)$-stopping time $\sigma\colon \Omega\to (0,\infty]$ such that $\lim_{t\nearrow \sigma}\|(r(t),v(t))\|=\infty$ $\bP$-a.s.\ on $\{\sigma<\infty\}$.
	The solution does not depend on the choice of the extension of the coefficient functions in \eqref{eq:inherentSDE} since, $\bP$-almost surely, $(r(t),v(t))\in T\cM$ for all $t\in[0,\sigma)$.
	Define $(\cF_t)$-stopping times $\sigma_K:=\inf\{t\in[0,\sigma):\|(r(t),v(t))\|\geq K\}$, $K\in\bN$, with $\inf\emptyset:=\infty$, so that $\sigma_K\nearrow\sigma$ $\bP$-a.s..  We have
	\begin{align}\label{eq:localSDE_app}
	\begin{split}
	\dl r(t\wedge\sigma_K)&=\mathds 1_{[0,\sigma_K]}(t)v(t)\,\dl t,\\
	\dl v(t\wedge\sigma_K)&=\mathds 1_{[0,\sigma_K]}(t)M^{-1}\Big\{P_M(r(t))\big[a(r(t),v(t))\,\dl t+ B(r(t),v(t))\,\dl w(t)\big]\\
	&\quad-\nabla g (r(t))\,G_M^{-1}(r(t))\,D^2g(r(t))(v(t),v(t))\,\dl t\Big\}.
	\end{split}
	\end{align}
	In order to verify Theorem~\ref{theo:solvability_app} it suffices to show that for all $T\in[0,\infty)$, $p\in[4,\infty)$ there exists a constant $C_{T,p}\in[0,\infty)$, depending on $T$, $p$, $\bE\|(r_0,v_0)\|^p$, $C_a$, $C_B$, $C_g$, $M$ but not on $K$, such that
	\begin{align}\label{eq:normestimate_app}
	\bE \sup_{t\in [0,T]}\|(r(t\wedge \sigma_K),v(t\wedge \sigma_K))\|^p\leq C_{T,p}.
	\end{align}
	Indeed, arguing as in \cite[Proposition~3.1]{lindner:p:2016} it follows that $\bP(\sigma=\infty)=1$, and by monotone convergence one also obtains the second assertion in Theorem~\ref{theo:solvability_app}.
	In order to verify \eqref{eq:normestimate_app} we apply It\^{o}'s formula to the function $\bR^n\times\bR^n\ni(x,y)\mapsto\|(x,M^{1/2}y)\|^2\in[0,\infty)$ and the stopped process $((r(t\wedge\sigma_K),v(t\wedge\sigma_K)))_{t\geq0}$ satisfying \eqref{eq:localSDE_app} to  obtain
	\begin{align}
	\begin{split}\label{eq:Ito_app}
	&\big\|\big(r(t\wedge\sigma_K),M^{1/2}v(t\wedge\sigma_K)\big)\big\|^2\\
	&=
	\big\|\big(r_0,M^{1/2}v_0\big)\big\|^2+\int_0^r\mathds 1_{[0,\sigma_K]}(s)2\langle r(s),v(s)\rangle\,\dl s\\
	&\quad+
	\int_0^t\mathds 1_{[0,\sigma_K]}(s)2\Big\langle v(s),P_M(r(s))\big[a(r,v)\,\dl s+ B(r(s),v(s))\,\dl w(s)\big]\Big\rangle\\
	&\quad-
	\int_0^t\mathds 1_{[0,\sigma_K]}(s)2\Big\langle v(s),\nabla g (r(s))\,G_M^{-1}(r(s))\,D^2g(r(s))(v(s),v(s))\,\dl s\Big\rangle\\
	&\quad+
	\int_0^t\mathds 1_{[0,\sigma_K]}(s)\big\|M^{-1/2}P_M(r(s))B(r(s),v(s))\big\|^2\,\dl s
	\end{split}
	\end{align}
	Next, observe that for all $(x,y)\in T\cM$ and $z\in\bR^n$ we have
	\begin{align}\label{eq:innerproducts_app}
	\begin{split}
	\langle y,P_M(x)z\rangle
	&=
	\big\langle M^{-1/2} My,M^{-1/2}P_M(x)z\big\rangle
	=
	\big\langle M^{-1/2} My,M^{-1/2}z\big\rangle
	=
	\langle y,z\rangle,
	\\
	\langle y,\nabla g(x)z\rangle
	&=
	\langle Dg(x)y,z\rangle
	=
	\langle 0,z\rangle=0
	\end{split}
	\end{align}
	Now \eqref{eq:normestimate_app} follows by taking $\bE(\sup_{t\in[0,T]}|\ldots|^{p/2})$ on both sides of \eqref{eq:Ito_app} and using \eqref{eq:innerproducts_app}, the Burkholder-Davis-Gundy inequality, Jensen's inequality, the growth conditions on
	 $a$ and $B$ from Assumption~\ref{ass:aB} as well as Gronwall's lemma. 
\end{proof}

\section{Globalized implicit function theorem}\label{sec:appGIFT}
As usual for differential-algebraic systems, the analysis of our numerical scheme is based on suitable local and global versions of the implicit function theorem (IFT). 
We first recall a version of the classical local IFT which does not assume differentiability in all variables. It is needed in Step 6 of the proof of Theorem~\ref{theo:mainonestep}. The proof is standard and therefore omitted.

\begin{theorem}\label{theo:localIFT}
Let $U_1\subset\bR^k$, $U_2\subset \bR^m$ be open sets, $(x_0,y_0)\in U_1\times U_2$, 
and let $F\colon U_1\times U_2\to \mathbb R^m\colon(x,y)\mapsto F(x,y)$ be continuous 
such that the partial derivative $D_yF(x,y)\in\bR^{m\times m}$ exists for all $(x,y)\in U_1\times U_2$ and the mapping $(x,y)\mapsto D_yF(x,y)$ is continuous.
Assume that $D_y F(x_0,y_0)\in\bR^{m\times m}$ is invertible.
Then there exists an open neighborhood $V_1\subset U_1$ of $x_0$,  
an open neighborhood $V_2\subset U_2$ of $y_0$,
and a continuous function $\gamma\colon V_1\to V_2\colon x\mapsto \gamma(x)$ such that $\gamma(x_0)=y_0$ and 
	\begin{align*}
		\forall\, (x,y)\in V_1\times V_2:\,F(x,y)=F(x_0,y_0)\,\Longleftrightarrow\,y=\gamma(x).
	\end{align*}
Further, $V_1$ can be chosen in such a way  that
$D_yF(x,\gamma(x))$ is inverible for all $x\in V_1$ and the following assertion holds: 
If $F$ is continuously differentiable, then
$\gamma$ is continuously differentiable with derivative
	\begin{align}\label{eq:Dgamma}
		D\gamma(x)=-\big(D_y F(x,\gamma(x))\big)^{-1} D_x F(x,\gamma(x)).
	\end{align}
\end{theorem}

Next we state a suitably globalized version of the IFT which is used in Step~2 and Step~3 of the proof of Theorem~\ref{theo:mainonestep}.  
For the sake of completeness we also present a short proof.
Recall that for $y\in\bR^m$ and $r\in(0,\infty)$ we denote by $B_r(y)$ the open ball in $\bR^m$ with radius $r$ and center $y$.

\begin{theorem}\label{theo:globalIFT}
	Let  $U_1\subset\bR$, $U_2\subset \bR^m$ be open sets, $(x_0,y_0)\in U_1\times U_2$, and let $C\in(0,\infty)$ be such that $[x_0,x_0+1]\times B_{2C}(y_0)\subset U_1\times U_2$. Let $F\colon U_1\times U_2 \to \bR^m$ be continuously differentiable and assume that, for all $(x,y)\in[x_0,x_0+1]\times B_{2C}(y_0)$ such that $F(x,y)=F(x_0,y_0)$, the Jacobian matrix $D_yF(x,y)\in\bR^{m\times m}$ is invertible and 
	\begin{align}\label{eq:gIFT_1}
		\big\|\big(D_yF(x,y)\big)^{-1}D_xF(x,y)\big\|< C.
	\end{align}
	Then there exists a unique continuously differentiable function $\gamma\colon[x_0,x_0+1]\to B_{C}(y_0)$ satisfying $\gamma(x_0)=y_0$ and
	\begin{align*}
		\forall\,x\in[x_0,x_0+1]:\,F(x,\gamma(x))=F(x_0,y_0).
	\end{align*}
	The derivative of $\gamma$ is given by \eqref{eq:Dgamma}.
	Moreover, if $F(x_0,y)\neq F(x_0,y_0)$ for all $y\in B_{2C}(y_0)\setminus \{y_0\}$, then we have
	\begin{align*}
		\forall\,(x,y)\in[x_0,x_0+1]\times B_{C}(y_0):\, F(x,y)=F(x_0,y_0)\,\Longleftrightarrow\, y=\gamma(x).
	\end{align*}
\end{theorem}

\begin{proof}[Proof of Theorem~\ref{theo:globalIFT}]
	W.l.o.g.~ we can assume that $x_0=0$, $y_0=0$ and $F(x_0,y_0)=0$, else look at the function $\tilde F(x,y)=F(x-x_0,y-y_0)-F(x_0,y_0)$.
	
Let $I$ be the set consisting of all $h\in(0,1]$ such that there exists a unique continuous function $\gamma_h\colon[0,h]\to B_{2C}(0)$ with the property that $\gamma_h(0)=0$ and $F(x,\gamma_h(x))=0$ for all $x\in[0,h]$. Clearly, $I$ is non-empty as a consequence of the local IFT, Theorem~\ref{theo:localIFT}. Let $h^*:=\sup I$. Due to the uniqueness of the functions $\gamma_h$, $h\in I$, we can define a continuous function $\gamma\colon [0,h^*)\to B_{2C}(0)$ by setting $\gamma(x):=\gamma_h(x)$ for $x\in [0,h]$.
Using again the local IFT, we know that $\gamma$ is continuously differentiable. 
Its derivative is given by \eqref{eq:Dgamma} and hence bounded due to \eqref{eq:gIFT_1}. Consequently, there exists a continuous extension $\gamma\colon[0,h^*]\to B_{2C}(0)$ of $\gamma$ to the closed interval $[0,h^*]$. 
By the continuity of $F$, we have $F(h^*,\gamma(h^*))=0$. 
Next, observe that $h^*=1$. Indeed, if $h^*$ was strictly smaller than one, we would obtain a contradiction to the definition of $h^*$ by applying the local IFT at the point $(h^*,\gamma(h^*))$.
A further application of the local IFT and \eqref{eq:gIFT_1} show that $\gamma$ is an element of $\mathcal C^1([0,1],B_C(0))$, with derivative $D\gamma(x)$ given by \eqref{eq:Dgamma} for all $x\in[0,1]$. Moreover, $\gamma$ is the only continuous function from $[0,1]$ to $B_{2C}(0)$ satifying $\gamma(0)=0$ and $F(x,\gamma(x))=0$ for all $x\in[0,1]$.
Indeed, if there was another such function $\tilde\gamma$, we would obtain a contradiction to 
the local IFT by considering the point $\sup\{x\in[0,1]:\gamma(x)=\tilde\gamma(x)\}\in[0,1)$.

In order to verify the last assertion, consider $(x_1,y_1)\in[0,1]\times B_C(0)$ such that $F(x_1,y_1)=0$. 
We have to show that $y_1=\gamma(x_1)$. W.l.o.g.~we can assume that $x_1\neq0$.
Let $\tilde I$ be the set consisting of all $h\in[0,x_1)$ such that there exists a unique continuous function $\tilde\gamma_h\colon[h,x_1]\to B_{2C}(0)$ with the property that $\tilde\gamma_h(x_1)=y_1$ and $F(x,\tilde\gamma_h(x))=0$ for all $x\in[h,x_1]$. Define $h_*:=\inf\tilde I$.	
Arguing as above, we see that $h_*=0$ and there exists a unique continuously differentiable function $\tilde\gamma\colon[0,x_1]\to B_{2C}(0)$ such that $\tilde\gamma(h_1)=y_1$ and $F(x,\tilde\gamma(x))=0$ for all $x\in[0,x_1]$. The additional assumption on $F$ yields that $\tilde\gamma(0)=\gamma(0)=0$. Hence, 
a further application of
the local IFT implies that $\gamma$ and $\tilde\gamma$ coincide on $[0,x_1]$, so that $\gamma(x_1)=\tilde\gamma(x_1)=y_1$.	
\end{proof}
\end{appendix}

\bibliographystyle{siam}
\bibliography{literatur}

\begin{thebibliography}{10}

\bibitem{beyn:p:2016}
{\sc W.-J. Beyn, E.~Isaak, and R.~Kruse}, {\em Stochastic {C}-stability and
  {B}-consistency of explicit and implicit {E}uler-type schemes}, J. Sci.
  Comput., 67 (2016), pp.~955--987.

\bibitem{beyn:p:2017}
\leavevmode\vrule height 2pt depth -1.6pt width 23pt, {\em Stochastic
  {C}-stability and {B}-consistency of explicit and implicit {M}ilstein-type
  schemes}, J. Sci. Comput., 70 (2017), pp.~1042--1077.

\bibitem{brenan:b:1989}
{\sc K.~E. Brenan, S.~L. Campbell, and L.~R. Petzold}, {\em Numerical solution
  of initial-value problems in differential-algebraic equations}, North-Holland
  Publishing Co., New York, 1989.

\bibitem{daprato:b:2014}
{\sc G.~Da~Prato and J.~Zabczyk}, {\em Stochastic equations in infinite
  dimensions}, vol.~152 of Encyclopedia of Mathematics and its Applications,
  Cambridge University Press, Cambridge, second~ed., 2014.

\bibitem{fang:p:2016}
{\sc W.~Fang and M.~B. Giles}, {\em Adaptive {E}uler-{M}aruyama method for
  {SDE}s with non-globally {L}ipschitz drift: {P}art {I}, finite time
  interval}, arXiv preprint arXiv:1609.08101,  (2016).

\bibitem{gear:p:1985}
{\sc C.~W. Gear, G.~K. Gupta, and B.~Leimkuhler}, {\em Automatic integration of
  {E}uler-{L}agrange equations with constraints}, in Proceedings of the
  international conference on computational and applied mathematics ({L}euven,
  1984), vol.~12/13, 1985, pp.~77--90.

\bibitem{gyongy:p:1998}
{\sc I.~Gy{\"o}ngy}, {\em A note on {E}uler's approximations}, Potential Anal.,
  8 (1998), pp.~205--216.

\bibitem{hairer:b:1989}
{\sc E.~Hairer, C.~Lubich, and M.~Roche}, {\em The numerical solution of
  differential-algebraic systems by {R}unge-{K}utta methods}, vol.~1409 of
  Lecture Notes in Mathematics, Springer-Verlag, Berlin, 1989.

\bibitem{hairer:b:1991}
{\sc E.~Hairer and G.~Wanner}, {\em Solving ordinary differential equations.
  {II}}, vol.~14 of Springer Series in Computational Mathematics,
  Springer-Verlag, Berlin, second~ed., 1996.
\newblock Stiff and differential-algebraic problems.

\bibitem{higham:p:2002}
{\sc D.~J. Higham, X.~Mao, and A.~M. Stuart}, {\em Strong convergence of
  {E}uler-type methods for nonlinear stochastic differential equations}, SIAM
  J. Numer. Anal., 40 (2002), pp.~1041--1063.

\bibitem{hong:p:2011}
{\sc J.~Hong, S.~Zhai, and J.~Zhang}, {\em Discrete gradient approach to
  stochastic differential equations with a conserved quantity}, SIAM J. Numer.
  Anal., 49 (2011), pp.~2017--2038.

\bibitem{hsu:b:2002}
{\sc E.~P. Hsu}, {\em Stochastic analysis on manifolds}, vol.~38 of Graduate
  Studies in Mathematics, American Mathematical Society, Providence, RI, 2002.

\bibitem{hutzenthaler:p:2014}
{\sc M.~Hutzenthaler and A.~Jentzen}, {\em On a perturbation theory and on
  strong convergence rates for stochastic ordinary and partial differential
  equations with non-globally monotone coefficients}, arXiv preprint
  arXiv:1401.0295,  (2014).

\bibitem{hutzenthaler:p:2015}
\leavevmode\vrule height 2pt depth -1.6pt width 23pt, {\em Numerical
  approximations of stochastic differential equations with non-globally
  {L}ipschitz continuous coefficients}, Mem. Amer. Math. Soc., 236 (2015),
  pp.~v+99.

\bibitem{hutzenthaler:p:2011}
{\sc M.~Hutzenthaler, A.~Jentzen, and P.~E. Kloeden}, {\em Strong and weak
  divergence in finite time of {E}uler's method for stochastic differential
  equations with non-globally {L}ipschitz continuous coefficients}, Proc. R.
  Soc. Lond. Ser. A Math. Phys. Eng. Sci., 467 (2011), pp.~1563--1576.

\bibitem{hutzenthaler:p:2012}
\leavevmode\vrule height 2pt depth -1.6pt width 23pt, {\em Strong convergence
  of an explicit numerical method for {SDE}s with nonglobally {L}ipschitz
  continuous coefficients}, Ann. Appl. Probab., 22 (2012), pp.~1611--1641.

\bibitem{kupper:p:2012}
{\sc D.~K{\"u}pper, A.~Kv{\ae}rn{\o}, and A.~R{\"o}{\ss}ler}, {\em A
  {R}unge-{K}utta method for index 1 stochastic differential-algebraic
  equations with scalar noise}, BIT, 52 (2012), pp.~437--455.

\bibitem{kupper:p:2015}
\leavevmode\vrule height 2pt depth -1.6pt width 23pt, {\em Stability analysis
  and classification of runge--kutta methods for index 1 stochastic
  differential-algebraic equations with scalar noise}, Applied Numerical
  Mathematics, 96 (2015), pp.~24--44.

\bibitem{leimkuhler:b:2015}
{\sc B.~Leimkuhler and C.~Matthews}, {\em Molecular Dynamics}, Springer, 2015.

\bibitem{leimkuhler:p:2016}
\leavevmode\vrule height 2pt depth -1.6pt width 23pt, {\em Efficient molecular
  dynamics using geodesic integration and solvent-solute splitting}, Proc. A.,
  472 (2016), pp.~1--22, 20160138.

\bibitem{lelievre:b:2010}
{\sc T.~Leli\`evre, M.~Rousset, and G.~Stoltz}, {\em Free energy computations},
  Imperial College Press, London, 2010.
\newblock A mathematical perspective.

\bibitem{lelievre:p:2012}
\leavevmode\vrule height 2pt depth -1.6pt width 23pt, {\em Langevin dynamics
  with constraints and computation of free energy differences}, Math. Comp., 81
  (2012), pp.~2071--2125.

\bibitem{lindner:p:2016}
{\sc F.~Lindner, N.~Marheineke, H.~Stroot, A.~Vibe, and R.~Wegener}, {\em
  Stochastic dynamics for inextensible fibers in a spatially semi-discrete
  setting}, Stoch. Dyn., 17 (2017), pp.~1--29, 1750016.

\bibitem{lindner:p:2018}
{\sc F.~Lindner, H.~Stroot, and R.~Wegener}, {\em Semi-discretized stochastic
  fiber dynamics: Non-linear drag force}, in Progress in industrial mathematics
  at {ECMI} 2016, vol.~26 of Mathematics in Industry, Springer International
  Publishing, 2017, to appear.

\bibitem{loetstedt:p:1986}
{\sc P.~L\"otstedt and L.~Petzold}, {\em Numerical solution of nonlinear
  differential equations with algebraic constraints. {I}. {C}onvergence results
  for backward differentiation formulas}, Math. Comp., 46 (1986), pp.~491--516.

\bibitem{lubich:p:1989}
{\sc C.~Lubich}, {\em h2-extrapolation methods for differential-algebraic
  systems of index 2}, IMPACT of Computing in Science and Engineering, 1
  (1989), pp.~260--268.

\bibitem{malham:p:2008}
{\sc S.~J.~A. Malham and A.~Wiese}, {\em Stochastic {L}ie group integrators},
  SIAM J. Sci. Comput., 30 (2008), pp.~597--617.

\bibitem{mao:p:2015}
{\sc X.~Mao}, {\em The truncated {E}uler-{M}aruyama method for stochastic
  differential equations}, J. Comput. Appl. Math., 290 (2015), pp.~370--384.

\bibitem{mao:p:2016}
\leavevmode\vrule height 2pt depth -1.6pt width 23pt, {\em Convergence rates of
  the truncated {E}uler-{M}aruyama method for stochastic differential
  equations}, J. Comput. Appl. Math., 296 (2016), pp.~362--375.

\bibitem{marheineke:p:2006}
{\sc N.~Marheineke and R.~Wegener}, {\em Fiber dynamics in turbulent flows:
  general modeling framework}, SIAM J. Appl. Math., 66 (2006), pp.~1703--1726.

\bibitem{marheineke:p:2011}
\leavevmode\vrule height 2pt depth -1.6pt width 23pt, {\em Modeling and
  application of a stochastic drag for fibers in turbulent flows},
  International Journal of Multiphase Flow, 37 (2011), pp.~136--148.

\bibitem{milstein:p:2002}
{\sc G.~N. Milstein, Y.~M. Repin, and M.~V. Tretyakov}, {\em Numerical methods
  for stochastic systems preserving symplectic structure}, SIAM J. Numer.
  Anal., 40 (2002), pp.~1583--1604.

\bibitem{milstein:p:20022}
\leavevmode\vrule height 2pt depth -1.6pt width 23pt, {\em Symplectic
  integration of {H}amiltonian systems with additive noise}, SIAM J. Numer.
  Anal., 39 (2002), pp.~2066--2088.

\bibitem{milstein:b:2004}
{\sc G.~N. Milstein and M.~V. Tretyakov}, {\em Stochastic numerics for
  mathematical physics}, Scientific Computation, Springer-Verlag, Berlin, 2004.

\bibitem{ostermann:p:1993}
{\sc A.~Ostermann}, {\em A class of half-explicit {R}unge-{K}utta methods for
  differential-algebraic systems of index {$3$}}, Appl. Numer. Math., 13
  (1993), pp.~165--179.
\newblock Sixth Conference on the Numerical Treatment of Differential Equations
  (Halle, 1992).

\bibitem{roemisch:p:2003}
{\sc W.~R\"omisch and R.~Winkler}, {\em Stochastic {DAE}s in circuit
  simulation}, in Modeling, simulation, and optimization of integrated circuits
  ({O}berwolfach, 2001), vol.~146 of Internat. Ser. Numer. Math., Birkh\"auser,
  Basel, 2003, pp.~303--318.

\bibitem{sabanis:p:2013}
{\sc S.~Sabanis}, {\em A note on tamed {E}uler approximations}, Electron.
  Commun. Probab., 18, no. 47 (2013), pp.~1--10.

\bibitem{sabanis:p:2015}
\leavevmode\vrule height 2pt depth -1.6pt width 23pt, {\em Euler approximations
  with varying coefficients: the case of superlinearly growing diffusion
  coefficients}, Ann. Appl. Probab., 26 (2016), pp.~2083--2105.

\bibitem{schein:p:1998}
{\sc O.~Schein and G.~Denk}, {\em Numerical solution of stochastic
  differential-algebraic equations with applications to transient noise
  simulation of microelectronic circuits}, J. Comput. Appl. Math., 100 (1998),
  pp.~77--92.

\bibitem{ciccotti:p:2006}
{\sc E.~Vanden-Eijnden and G.~Ciccotti}, {\em Second-order integrators for
  {L}angevin equations with holonomic constraints}, Chemical physics letters,
  429 (2006), pp.~310--316.

\bibitem{walter:p:2011}
{\sc J.~Walter, C.~Hartmann, and J.~Maddocks}, {\em Ambient space formulations
  and statistical mechanics of holonomically constrained {L}angevin systems},
  The European Physical Journal Special Topics, 200 (2011), pp.~153--181.

\bibitem{winkler:p:2003}
{\sc R.~Winkler}, {\em Stochastic differential algebraic equations of index 1
  and applications in circuit simulation}, J. Comput. Appl. Math., 157 (2003),
  pp.~477--505.

\bibitem{zhou:p:2016}
{\sc W.~Zhou, L.~Zhang, J.~Hong, and S.~Song}, {\em Projection methods for
  stochastic differential equations with conserved quantities}, BIT, 56 (2016),
  pp.~1497--1518.

\end{thebibliography}

\bigskip

\bigskip

\noindent{\bf Acknowledgement.} Helpful scientific discussions with Raimund Wegener at the Fraunhofer Institute for Industrial Mathematics ITWM are gratefully acknowledged.

\bigskip

\end{document}